%% file: ms_singularities.tex
\documentclass[12pt]{amsart}

\usepackage{amscd, graphicx, psfrag}
 \usepackage{ifthen}    %used for boolean features
 \usepackage{amssymb}   %includes \mathbb command
 \usepackage{amstext}
 \usepackage{amsfonts}
 \usepackage{amsmath}   %includes \boldsymbol command
 \usepackage{amscd}     %commutative diagrams
 \usepackage{latexsym}  %includes \Box command
 \usepackage{color}
 \usepackage{epsfig}
 \usepackage{epic}
 \usepackage{pstricks, pst-node,pst-text,pst-tree}
 \usepackage[all]{xy}
 \usepackage{array}
 \usepackage{amsthm}
 \usepackage{eucal, mathrsfs} %caligraphic and script fonts
 \usepackage{hyperref}

\newcommand{\numberset}[1]{\ensuremath{\mathbb{#1}}}    %amsfonts
\newcommand{\C}{\numberset{C}}  % amsfonts
\newcommand{\R}{\numberset{R}}  % amsfonts
\newcommand{\Z}{\numberset{Z}}  % amsfonts
\newcommand{\PP}{\numberset{P}}  % amsfonts
\newcommand{\half}{\ensuremath{\frac{1}{2}}}
\newcommand{\inn}[2]{ \left\langle {#1}, {#2} \right\rangle}
\newcommand{\DF}{\mathcal D^bFuk}
\newcommand{\DC}{\mathcal D^bCoh}
\newcommand{\ext}{\mathcal{E}xt}
\newcommand{\ch}{\mathcal{H}om}
\theoremstyle{definition}

\newtheorem{thm}{Theorem}[section]
\newtheorem{prop}[thm]{Proposition}
\newtheorem{lem}[thm]{Lemma}
\newtheorem{cor}[thm]{Corollary}
\theoremstyle{remark}
\newtheorem{rem}[thm]{Remark}
\newtheorem{ex}[thm]{Example}
\newtheorem{defi}[thm]{Definition}
\newtheorem{con}[thm]{Conjecture}
\numberwithin{equation}{section}

\DeclareMathOperator{\Arg}{Arg}

\DeclareMathOperator{\Crit}{Crit} 

\DeclareMathOperator{\conv}{Conv} 
\DeclareMathOperator{\sgn}{sign} 
 \DeclareMathOperator{\Gl}{GL}
\DeclareMathOperator{\spn}{span} 
\DeclareMathOperator{\Hom}{Hom}
\DeclareMathOperator{\Ext}{Ext}
\DeclareMathOperator{\hes}{Hess}

\DeclareMathOperator{\Log}{Log}

\DeclareMathOperator{\rd}{red} 
\DeclareMathOperator{\Sk}{Skel}
\DeclareMathOperator{\vl}{Vol}

\begin{document}

\title[On HMS of toric Calabi-Yau threefolds]{On homological mirror symmetry of toric Calabi-Yau threefolds}

\author[M. Gross and D. Matessi]{Mark Gross and Diego Matessi}

\begin{abstract}
We use Lagrangian torus fibrations on the mirror $X$ of a toric Calabi-Yau threefold $\check X$ to construct Lagrangian sections and various Lagrangian spheres
on $X$. We then propose an explicit correspondence between the sections and line bundles on $\check X$ and between spheres and sheaves supported on the toric divisors of $\check X$. We conjecture that these correspondences induce an embedding of the relevant derived Fukaya category of $X$ inside the derived category of coherent sheaves on $\check X$. 
\end{abstract}

\maketitle

\tableofcontents

\section{Introduction}
An example of mirror symmetry which has been studied a great deal in recent years is so-called local mirror symmetry of (open) toric Calabi-Yau manifolds. If $\check X$ is a smooth toric Calabi-Yau manifold, then a construction of its mirror $X$ appeared first in the physics literature in various articles, see for instance \cite{kkv} or \cite{CKYZ}. Later, in \cite{Gross_spLagEx} and \cite{TMS}, the first author proved that $\check X$ and $X$ admit dual torus fibrations $\check f: \check X \rightarrow \R^n$ and $f: X \rightarrow \R^n$. This suggests that they are mirror to each other in the sense of the SYZ conjecture \cite{SYZ}. The construction in \cite{Gross_spLagEx} implies that $\check f$ is special Lagrangian; moreover, using the results in \cite{CB-M}, we can assume that $f$ is Lagrangian (see also \cite{AAK} for another construction of a Lagrangian fibration on $X$). More recently, mirror symmetry of $X$ and $\check X$ has been proved at deeper levels and in various other aspects, see for instance \cite{AAK}, \cite{ChanLauLeung}, \cite{ChanChoLauTseng}, \cite{GS:ICM}. 

In this article we consider a toric Calabi-Yau threefold $\check X$ and we give rather explicit constructions of Lagrangian sections of $f: X \rightarrow \R^3$ and La\-gran\-gian $3$-spheres in $X$. The Lagrangian spheres we construct represent a homology class which we show can be described via a difference of sections which coincide outside a compact set. Sections of Lagrangian fibrations have long been expected to correspond to line bundles on $\check X$, and our construction of the sections naturally suggests a precise such correspondence. Moreover the relationship between the Lagrangian spheres and different sections then suggests an explicit correspondence between the spheres and line bundles supported on the compact toric divisors of $\check X$. We conjecture that this correspondence induces an embedding of a suitable derived Fukaya category generated by the spheres and the sections into the derived category of coherent sheaves of $\check X$. The two-dimensional version of this conjecture has been partially proved by Chan and Chan-Ueda respectively in \cite{chan:an} and \cite{chan:ueda}. We then prove some results which support our conjecture. For instance, we show that by studying the differential topology of the intersection points between a sphere and a section one finds numbers of intersection points agreeing with the dimension
of the  morphism space between the corresponding line bundle and sheaf. In some special cases (when there are no Floer differentials for degree reasons), we can show that the intersection between a sphere and a section is transverse and that the Floer homology group between the section and the sphere coincides with the group of morphisms between the line bundle and the sheaf. We have not attempted to carry out a detailed analysis of Floer cohomology in this paper.

We apply our results to study the mirror symmetry of $A_{2d-1}$-sin\-gu\-la\-rities in dimension $3$. In particular we describe the vanishing cycles in a smoothing of an $A_{2d-1}$-singularity using our construction of Lagrangian spheres. Then we prove that in the mirror, our correspondence gives an $A_{2d-1}$-configuration of spherical objects, in the sense of Seidel and Thomas \cite{ST}. This refines and makes more explicit a conjecture of Seidel and Thomas in \cite{ST}. 

\subsection{Local mirror symmetry} Let $N \cong \Z^{n-1}$ be a lattice and $M$ its dual. Let $N_{\R}= N \otimes_{\Z} \R$ and $M_{\R} = M \otimes_{\Z} \R$. If $P \subset N_{\R}$ is a convex lattice polytope, let $C(P) \subset \R \times N_{\R}$ be the cone over $\{ 1 \} \times P$. A subdivision of $P$ in smaller lattice polytopes gives a subdivision of $C(P)$, i.e., a fan denoted $\Sigma$, and hence a toric variety $V_{\Sigma}$ of dimension $n$. If $P$ is subdivided in elementary simplices, then $V_{\Sigma}$ is a smooth Calabi-Yau variety of dimension $n$, in the sense that it has trivial canonical bundle. These toric Calabi-Yau varieties are called semi-projective in \cite{cox:little:schenck}, p.332 and are those whose fan has convex support. The variety $\check X$ is obtained by removing from $V_{\Sigma}$ a principal divisor which does not intersect the compact toric divisors of $V_{\Sigma}$.  The mirror $X$ is an affine variety defined in \eqref{mirror}.   We consider the $3$-dimensional case ($n=3$). The discriminant locus of the Lagrangian fibration $f: X \rightarrow \R^3$ can be thought as (a thickening of) the tropical curve $\Gamma$ defined by $P$ and its subdivision (see Figures \ref{can_p2} and \ref{a2d:trop} showing some examples). The critical locus of $f$ is a submanifold $S \subset X$ of real dimension $2$ which maps to $\Gamma$. We have that $X$ has a Hamiltonian $S^1$-action, preserving the fibres of $f$, with moment map $\mu$ such that $S \subset \mu^{-1}(0)$ and $S$ is the fixed point set of this action. It turns out that we can (partially) identify the reduced space $\mu^{-1}(0)/S^1$ with $N_{\R} \times (M_{\R} / M)$, with its standard symplectic form. 

\subsection{Lagrangian sections and spheres} We now describe the construction of the sections and of the spheres. From toric geometry it follows that a line bundle on $\check X$ is described by a piecewise integral affine function $\phi: P \rightarrow \R$, called a support function, whose domains of affineness are unions of the polytopes in the subdivision. Using such a function we construct a Lagrangian section of $f$ as follows. First we extend $\phi$ in a piecewise affine way to all of $N_{\R}$, then we smooth $\phi$ by convoluting it with a ``mollifier'' (see \eqref{mollifier} and \eqref{convolution} for the definition of a mollifier and of the convolution product). This gives a smooth function $\tilde \phi_{\epsilon}$, which essentially is $\phi$ with its  ``corners'' smoothed out. The graph of the differential of $\tilde \phi_{\epsilon}$ gives a Lagrangian section of the reduced space $N_{\R} \times (M_{\R} / M)$.  We then argue that this section can be lifted and extended to give a section of $f$. This is essentially the content of Theorem \ref{sections_thm}. 

The construction of spheres is similar. The complement of the tropical curve $\Gamma$ in $N_{\R}$ has a finite number of bounded connected components and there is a one-to-one correspondence between these connected components and the interior vertices in the subdivision of $P$. We denote by $C$ a connected component and by $v_C$ the corresponding vertex. Moreover, interior vertices are in one-to-one correspondence with compact toric divisors of $\check X$, which we denote by $D_C$. The tangent wedges to the polytopes containing the vertex $v_C$ of the subdivision form a fan $\Sigma_C$ which is the fan of the toric divisor $D_C$. We define the notion of a semi-integral support function (see Definition~\ref{semi:integral}), i.e., a continuous function $\vartheta: |\Sigma_C| \rightarrow \R$ which is linear on the cones of $\Sigma_C$ and satisfies a certain compatibility condition \eqref{semi-int}. Then let $\tilde \vartheta_{\epsilon}$ be a smoothing of $\vartheta$ (obtained by convolution with a mollifier). We interpret $\tilde \vartheta_{\epsilon}$ as a real function on $C$. It then turns out that the graph of the differential of $\tilde \vartheta_{\epsilon}$ defines a map $\lambda: C \rightarrow N_{\R} \times (M_{\R} / M)$ which maps the boundary of $C$ to the critical surface $S$. Then  
$\alpha^{-1}(\lambda(C))$ gives the Lagrangian sphere, where $\alpha$ is the quotient map $\alpha: X \rightarrow X/S^1$, and the reduced space is identified
with $\mu^{-1}(0)/S^1 \subset X/S^1$. In other words, the sphere is an
$S^1$-fibration over $C$, with the circles degenerating to points over the
boundary of $C$. This is the content of Theorem \ref{lag_spheres}. 

It is worth pointing out that our construction of Lagrangian spheres is rather general and also works in the case of compact Calabi-Yau threefolds with the Lagrangian fibrations constructed in \cite{CB-M}. 
In Example \ref{spheres_quintic} we show how to construct Lagrangian spheres inside the quintic threefold in $\PP^4$.

 In \cite{Ab_HMS}, M. Abouzaid studied homological mirror symmetry of a compact toric variety $Y$. The mirror of $Y$ is a Landau-Ginzburg model $((\C^*)^n, W)$, where $W: (\C^*)^n \rightarrow \C$ is a Laurent polynomial. Abouzaid proves that the bounded derived category of coherent sheaves on $Y$ embeds as a full subcategory inside the derived Fukaya category of $((\C^*)^n, W)$.  He defines such a Fukaya category by considering certain Lagrangian submanifolds with boundary on $W^{-1}(0)$.  Abouzaid's construction is similar to ours and it is likely that his correspondence is strongly related to the one we propose. 
\subsection{The correspondence} 
The correspondence between Lagrangian sections of $f: X \rightarrow \R^3$ and line bundles of $\check X$ is stated in Conjecture \ref{hms:dim2}. Essentially, given the piecewise affine function $\phi: P \rightarrow \R$, the section constructed from $\phi$ as described above should correspond to the line bundle given by the function $- \phi$. In order to decide which sheaves should correspond to the Lagrangian spheres we proceed as follows. Suppose we have a pair of sections $\sigma$ and $\sigma'$ which coincide outside a compact set $K$ homeomorphic to a $3$-ball, then we say that $\sigma$ is compactly supported with respect to $\sigma'$.  We can view a (topological) sphere as two copies of $K$ glued along their boundaries. A map from this sphere to $X$ is constructed by defining it to be $\sigma$ on one copy of $K$ and $\sigma'$ on the other copy. This map defines a homology class in $H_3(X, \Z)$ which we denote by $[\sigma \sigma']$. If this homology class is represented by a Lagrangian sphere $L$, then we expect the sheaf $\mathscr E_{L}$ corresponding to $L$ to fit in a short exact sequence of the type 
\[ 0 \longrightarrow \mathscr L_{\sigma'} \longrightarrow \mathscr L_{\sigma} \longrightarrow \mathscr E_{L} \longrightarrow 0. \]
The motivation for this is that it seems reasonable to expect that $\sigma$ can be obtained as the Lagrangian connected sum $\sigma' \# L$. Then a general argument in mirror symmetry (see \cite{A13}, Section 3.3.2) 
suggests that $\mathscr L_{\sigma}$ should be an extension of $\mathscr 
L_{\sigma'}$ by $\mathscr E_{L}$.

In Section \ref{compact_support} we investigate which pairs of sections $\sigma$ and $\sigma'$ can be assumed (up to isotopy) to coincide outside some compact set $K$ and we determine the Lagrangian sphere whose homology class coincides with $[\sigma \sigma']$. This is the content of Proposition \ref{com_sup:dim3} and Theorem \ref{diff_sec}.

Using the above arguments, the outcome of our proposed correspondence is as follows (see the beginning of Section \ref{hms}). Given a semi-integral support function $\vartheta$ on $\Sigma_C$, let $\psi_{K_C}: |\Sigma_C| \rightarrow \R$ be the support function corresponding to the canonical bundle of $D_C$. Define $\psi$ as
 \[ \psi = \frac{1}{2} \psi_{K_C} - \vartheta. \]
Then the sheaf corresponding to the sphere defined by $\vartheta$ is the line bundle supported on $D_C$ defined by $\psi$.  This is the correspondence stated in Conjecture \ref{hms:spheres}.  We point out that in dimension two, in \cite{chan:an}, Chan constructs the same correspondence as ours between spheres and sheaves supported on compact toric divisors. Then he proves that this correspondence defines an embedding of the derived Fukaya category generated by the spheres inside the derived category of coherent sheaves of the mirror. Similarly, in \cite{chan:ueda}, Chan and Ueda construct, in the two-dimensional case, a correspondence between Lagrangian sections of $f$ and line bundles on $\check X$ and prove that it gives an embedding of derived categories. Chan, Ueda and Pomerleano obtained similar results in \cite{CPU:hms:conifold} in the three-dimensional example of the conifold. Moreover in \cite{CPU:sections} the same authors propose a similar construction of Lagrangian sections and mirror correspondence with line bundles and they prove that the wrapped Floer homology ring of the zero section is isomorphic to the algebra of functions on $\check X$ (we thank the authors for sending us a preliminary version of their work). We also point out the result of P. Seidel \cite{Seidel:susp} where, in the case $\check{X}$ is the total space of the canonical bundle of a smooth toric del Pezzo surface $Y$, he finds a full embedding of triangulated categories between the bounded derived category of sheaves supported on $Y$ and the derived Fukaya category of the mirror $X$. Seidel's method is based on the ``suspension'' of Lefschetz fibrations. We do not know what is the relation between our proposed correspondence and Seidel's result. 

\subsection{Other results} We obtain some results which support our conjectures. For instance, in the case that a semi-integral support function $\vartheta$ (or its opposite $- \vartheta$) is strictly convex, we prove that the sphere corresponding to $\vartheta$ intersects the zero section transversely (see Corollary \ref{convex:intersection}). 

This implies that the Floer homology group between the zero section and the sphere is isomorphic to the group of morphisms between the corresponding sheaves (see Remark \ref{floer}).  In fact in this case we argue that for topological reasons the Floer differential must be zero. Thus the dimension of the Floer homology group coincides with the number of intersection points. Then we show that these are in one-to-one correspondence with the integral points in the interior of a convex polytope with vertices on the half lattice $\frac{1}{2} M$ (see Figure \ref{triangle_intersection} for some examples). It is not hard to show that the number of such points is precisely the dimension of $H^2(D_C, \mathscr L_\psi)$ (or respectively $H^0(D_C, \mathscr L_{\psi})$), where $\mathscr L_{\psi}$ is the line bundle on $D_C$ corresponding to the support function $\psi$ defined above.    

More generally, in Theorem \ref{coho_winding} we prove that given any $\vartheta$, the differential topology of the intersection points between the zero section and the sphere corresponding to $\vartheta$ allows us to recover the morphism groups between the corresponding sheaves. More precisely, the half-integral support function $\vartheta$ defines a piecewise linear closed curve in $M_{\R}$. We prove that the dimension of $H^{\text{even}}(D_C, \mathscr L_\psi)$ is the number of integral points which have positive winding number with respect to this curve. Moreover, the dimension of $H^{\text{odd}}(D_C, \mathscr L_\psi)$ is obtained as minus the sum of the winding numbers (with respect to this curve) of all integral points whose winding number is negative. 

\subsection{$A_{2d-1}$-singularities} In Section \ref{spherical:an} we study the mirror symmetry of $A_{2d-1}$ singularities in dimension $3$. In fact, the mirror of a smoothing of an $A_{2d-1}$ singularity in dimension $3$ is a toric Calabi-Yau. We can view the $2d-1$ vanishing cycles using our construction above.  For instance, in the case $d=3$, the smoothing $X$ has a Lagrangian fibration with discriminant locus the tropical curve $\Gamma$ as in Figure \ref{a2d:trop}. Then, the complement of $\Gamma$ has two bounded regions. We obtain one vanishing cycle for each of these bounded regions. There are another three vanishing cycles constructed over some one dimensional edges of $\Gamma$ (see \S \ref{vc:edges}).  
Our correspondence gives us the conjectural mirror objects. 

%The final result is as follows. We obtain $d$ spheres constructed over $d-1$ bounded connected components of the complement of $\Gamma$, denote them by $L_1, \ldots, L_{d-1}$, and $d$ spheres constructed over some edges of $\Gamma$, denote them by $N_1, \ldots, N_d$. Then $(N_1, L_1, N_2, \ldots, L_{d-1}, N_d)$ forms an $A_{2d-1}$ sequence of objects in the sense of Seidel and Thomas \cite{ST}. The mirror of $N_j$ is a line bundle $\mathscr E_{L_j}$ supported on a toric divisor isomorphic to a one point blow up of the Hirzebruch surface of type $j$. The mirror of $N_j$ is the line bundle $\O_{\PP^1}(-1)$ supported on an embedded (toric) $\PP^1$. 
In Proposition \ref{a2d:sequence} we prove that the mirror objects form an $A_{2d-1}$-sequence in the sense of Seidel and Thomas \cite{ST}. 

\subsection*{Acknowledgments} The authors would like to thank
Mohammed Abouzaid, Ricardo Casta\~no-Bernard and Bernd Siebert for many useful 
conversations.  Mark Gross was partially
supported by the NSF grant DMS-1105871. 
Diego Matessi was partially supported by the grant FIRB 2012 ``Moduli spaces and their applications'' and by the national research project ``Geometria delle variet\`a proiettive'' PRIN 2010-11. Both authors were supported by the NSF
grant DMS-0854987.

\section{Local mirror symmetry and torus fibrations} \label{local_msym}

\subsection{Local mirror symmetry of toric singularities} 
Let $N \cong \Z^{n-1}$ be a lattice and $M = \Hom(N, \Z)$ its dual lattice.  Let 
$N_{\R} = N \otimes_{\Z} \R$ and $M_{\R} =  M \otimes_{\Z} \R$. Identify $(\Z \oplus N) \otimes \R$ with $ \R \times N_{\R}$. We also denote by $T_M$ the complex algebraic torus $M \otimes \C^*$. 
Given a convex lattice polytope $P \subseteq N_{\R}$, let $C(P)  \subseteq \R \times N_{\R}$ be the cone over $\{ 1 \} \times P$ and let $V_{P}$ be the $n$-dimensional toric variety defined by $C(P)$. In general $V_P$ is singular with an isolated Gorenstein singularity. Given a subdivision of $P$ in convex lattice polytopes $P_1, \ldots, P_k$, we obtain a fan $\Sigma$ by considering the cones over the faces of the $P_j$'s . Let $V_{\Sigma}$ be the toric variety associated to this fan. When the subdivision of $P$ is smooth, i.e., all $P_j$'s are elementary simplices, then $V_\Sigma$ is a smooth Calabi-Yau, providing a crepant resolution of $V_{P}$. We will always assume that the subdivision is smooth. The lattice point $(1, 0) \in \Z \times M$ defines a regular monomial $z^{(1,0)}: V_{\Sigma} \rightarrow \C$. Define 
\begin{equation} \label{Xcek}
 \check X = V_{\Sigma} - \{ z^{(1,0)} = 1 \}.
\end{equation}

A mirror of $\check X$ was first predicted by phycists, e.g. in \cite{kkv} or \cite{CKYZ}. In general it will be a family of varieties roughly parameterized by the
K\"ahler moduli space of $\check X$. Because we will only be concerned about the symplectic geometry of the mirror, we define a convenient one-real parameter
sub-family, as follows. Consider a function $\nu: P \rightarrow \R$ which is piecewise affine and strictly convex with respect to the given subdivision of $P$ (i.e., the domains of linearity coincide precisely with the polytopes $P_j$). Assume also that $\nu$, restricted to $P \cap N$, has integer values.  Elements $j \in N$ correspond to monomials, or characters, on the torus $T_M$ which we denote by $z^j$. For $t \in \R_{>0}$, consider the family of Laurent polynomials
\[ h_t = \sum_{j \in P \cap N} t^{\nu(j)} a_j z^j, \]
where the coefficients $a_j \in \C^{\ast}$. We will construct the mirror of $\check X$ by fixing $t$ and taking the $n$-dimensional variety
\begin{equation} \label{mirror}
   X = \{ (x,y, z) \in \C^2 \times T_M \, | \, xy = h_t(z) \} 
\end{equation}
We will typically be interested in the case that $t$ is very close to $0$.
It can be shown that $X$ is Calabi-Yau, i.e. it has trivial canononical bundle (see \cite{Gross_spLagEx} for an explicit global holomorphic $n$-form). 

\subsection{The tropical hypersurface $\Gamma$} \label{trop_hyp}
Let $\nu: P \rightarrow \R$ be the function defined above. Then it extends to a piecewise linear strictly convex function $\tilde{\nu}: |\Sigma| \rightarrow \R$, by defining $\tilde{\nu}(1,v) = \nu(v)$ for all $v \in P$ and extending linearly. The function $\tilde{\nu}$ also corresponds to a choice of ample line bundle on $V_{\Sigma}$, whose Newton polyhedron $\sigma$ is given by
\begin{equation} \label{image_moment}
 \sigma = \{ (t, m) \in \R \times M_{\R} \, | \, \inn{(1,v)}{(t,m)} + \nu(v) \geq 0, \ 
             \text{for all} \ v \in P \cap N \}.
\end{equation}
Now, over $M_{\R}$, define the following piecewise affine function
\begin{equation} \label{lt:nu} 
\check{\nu}(m) = \min \{ \inn{v}{m} + \nu(v), \,  v \in P \cap N \}.
\end{equation}
This function is the discrete Legendre transform of $\nu$. Clearly 
\[ \sigma = \{ (t, m) \in \R \times M_{\R} \, | \, t \geq - \check{\nu}(m) \}. \]
The subset of $M_{\R}$ where $\check{\nu}$ fails to be smooth is a polyhedral complex $\Gamma$ whose maximal cells have dimension $n-2$ and lie on affine subspaces of rational slope. The complex $\Gamma$ is also called the non-archimedean amoeba of the polynomial $h_t$ or the {\it tropical hypersurface} defined by $h_t$. In the case $n=2$, $\Gamma$ is just a finite set of points.  In the case $n = 3$, $\Gamma$ is a tropical curve, topologically a graph with trivalent vertices. 

We have that $\Gamma$ defines a polyhedral subdivision of $M_{\R}$, whose maximal cells are the closures of the connected components of $M_{\R} - \Gamma$. This subdivision is dual to the subdivision of $P$, in the sense that there is a one-to-one, inclusion reversing correspondence between $k$-dimensional cells in the subdivision of $P$ and codimension $k$ cells in the subdivision of $M_{\R}$.  In particular, there is a one-to-one correspondence between connected components of $M_{\R} - \Gamma$ and vertices of the subdivision of $P$. If $C$ is a connected component of $M_{\R}- \Gamma$, we denote by $v_{C}$ the corresponding vertex. Vertices on the boundary (resp. in the interior) of $P$ correspond to unbounded (resp. bounded) components. 

In this paper we will mostly consider the case $\dim M_{\R} = 2$, i.e., $n=3$, so let us fix some notation for this case. 
An edge $\check e$ in the subdivision of $P$ corresponds to an edge $e$ of $\Gamma$. If $v_C$ and $v_{C'}$ are the vertices of $\check e$, then $e$ is the common edge of $C$ and $C'$. Clearly $e$ is an infinite ray if and only if $v_C$ and $v_{C'}$ both lie on the boundary of $P$. If $n_e \in M_{\R}$ is a primitive integral tangent vector to $e$ and $n_{\check e}$ is a primitive integral tangent vector to $\check e$ then 
\begin{equation} \label{edge_dual}
\inn{n_e}{n_{\check e}}= 0.
\end{equation}
For every edge $\check e$ in the subdivision of $P$, we fix the following choices:
\begin{itemize}
\item[a)] we fix $n_{\check e}$ and $n_{e}$, primitive integral tangent vectors to $\check e$ and $e$ respectively; 
\item[b)] if $\check e$ is an interior edge, we label the two $2$-dimensional simplices containing $\check e$ by $P^+_{e}$ and $P^-_{e}$,  so that for every $q^+ \in P^+_e$ and $q^- \in P^-_e$ we have
\[ \inn{n_{e}}{q^+-q^-} \geq 0; \]
\item[c)] if $\check e$ is a boundary edge, we label by $P_e$ the unique simplex containing $\check e$.
\end{itemize}
Observe that if $p^+_e$ is the vertex of $\Gamma$ corresponding to $P^+_e$ and $p^-_e$ is the vertex corresponding to $P^-_e$ then $n_e$ points outward from $p^+_e$ in the direction of $p^-_e$. This can be deduced from the definition of $\Gamma$. 

Given a vertex $p$ of $\Gamma$, the following balancing condition holds
\begin{equation} \label{balance}
 \sum_{p \in e} \epsilon_e n_{e} = 0, 
\end{equation}
where $\epsilon_e = 1$ if $n_e$ points outward from $p$, otherwise $\epsilon_e = -1$. 
The fact that this holds for such a choice of $\epsilon_e$ follows from the
fact that all edges of $P$ are affine length $1$, as the subdivision of
$P$ is smooth.
Moreover, for any pair of edges $e_1$ and $e_2$ emanating from $p$, $n_{e_1}$ and $n_{e_2}$ form a basis of $M$, again because the subdivision of $P$ is smooth. 

\subsection{Torus fibrations} \label{torus_fibrations}
In the following we restrict to dimension $n=2$ or $3$. The claim that $X$ and $\check X$ are mirror to each other has been proved to various degrees of precision in the literature. See for instance \cite{ChanLauLeung}, \cite{ChanChoLauTseng}, \cite{AAK}, \cite{GS:ICM}.  In \cite{Gross_spLagEx}, the first author shows that $X$ and $\check{X}$ admit torus fibrations which are dual to each other in the sense of the SYZ-conjecture, as refined in \cite{TMS}.  For our purposes we will only need the fibration on $X$, so we give a description of this fibration only. In fact we will give three constructions.

Choosing an isomorphism $M \cong \Z^{n-1}$ identifies the torus $T_M$ with $(\C^*)^{n-1}$. We denote by $(z_1, \ldots, z_{n-1})$ the standard coordinates on $(\C^*)^{n-1}$. On $\C^2 \times (\C^{\ast})^{n-1}$ consider the 
symplectic form 
\begin{equation} \label{stsymp}
  \omega = \frac{i}{2} \left( dx \wedge d \bar x + dy \wedge d \bar y +  \sum_{j=1}^{n-1} \frac{dz_j \wedge d \bar z_j}{|z_j|^2}  \right) 
\end{equation}
and restrict it to $X$. We have a Hamiltonian $S^1$-action on $X$ given by
\[ e^{i\theta} \cdot (x,y,z) = (  e^{i\theta} x,  e^{-i\theta}y, z) \]
with moment map 
\[ (x,y, z) \mapsto \frac{|x|^2 - |y|^2}{2}. \]
Consider the $(n-1)$-torus fibration $\Log: (\C^{\ast})^{n-1} \rightarrow \R^{n-1}$ defined by 
\[ \Log z = (\log |z_1|, \ldots, \log |z_{n-1}|). \]
 Then the fibration $f: X \rightarrow \R \times \R^{n-1}$ is defined by 
\begin{equation} \label{dualfib}
 f (x,y,z) = (|x|^2 - |y|^2, \Log z). 
\end{equation}
When $n \geq 3$ this fibration is not Lagrangian, but later we will describe a different construction in dimension $n=3$ of an equivalent fibration which is Lagrangian (after \cite{CB-M}).  In this article we will not use this model of fibration but its simple and explicit form motivates the other constructions. 

Let $Y = X / S^1$. Notice that $ Y \cong \R \times (\C^{\ast})^{n-1}$ and $f$ is the composition of the projection $\alpha: X \rightarrow Y$ with the map  $\R \times  (\C^{\ast})^{n-1} \rightarrow \R \times \R^{n-1}$ given by $(s, z) \mapsto (s, \Log z)$. The fixed point locus of the $S^1$-action is the set of points where $x=y=0$. So, let $S_t \subset (\C^{\ast})^{n-1}$ be the zero set of the polynomial $h_t$ and let $Y'= Y - (\{ 0 \} \times S_t)$. Then $\alpha: \alpha^{-1}(Y') \rightarrow Y'$ forms a principal $S^1$-bundle. The Chern class of this
bundle is described as follows. The relative cohomology exact sequence yields
\[
0\rightarrow H^2(Y,\Z)\rightarrow H^2(Y',\Z)\rightarrow
H^3(Y, Y\setminus S_t,\Z)\cong \Z\rightarrow 0
\]
which is split by choosing a fibre $T^{n-1}\subset Y$ of $\R\times (\C^*)^{n-1}\rightarrow \R\times \R^{n-1}$ disjoint from
$S_t$ and using the composition $H^2(Y',\Z)\rightarrow H^2(T^{n-1},\Z)
\cong H^2(Y,\Z)$. Then the Chern class of $\alpha$ is
$(0, \pm 1 ) \in H^{2}(Y', \Z) \cong H^2(Y, \Z) \oplus \Z$. Let 
\[ \mathcal A = \Log (S_t). \] 
Then the singular fibres of $f$ lie over the set $\{ 0 \} \times \mathcal A$. The set $\mathcal A$ is called the amoeba of $S_t$ (in dimension $n=2$ it is just a finite set of points). 

\subsection{The complex tropical model} \label{topological}
We now describe a topological model of a fibration similar to the one given above. Since the total space of this fibration is homeomorphic to $X$, we will continue to denote it by $X$, although it is constructed in a different way. Here we consider only the cases $n=2$ or $3$. Let $\Gamma \subset M_{\R}$ be the tropical hypersurface defined in (\ref{trop_hyp}). One can identify $(\C^*)^{n-1}$ with $M_{\R} \times (M_{\R} / M)$, where $\Log$ becomes the projection onto $M_{\R}$. 
Suppose that $S \subset M_{\R} \times (M_{\R} / M)$ is a $2n-4$-dimensional real submanifold such that $\Log(S) = \Gamma$.  The first author, in  \cite{TMS}, makes the following construction. Let  $Y = \R \times M_{\R} \times (M_{\R} / M)$ and $Y' = Y - (\{ 0 \} \times S)$. If $\alpha': X' \rightarrow Y'$ is a principal $S^1$-bundle over $Y'$ with Chern class $(0, \pm 1 ) \in H^{2}(Y', \Z) \cong H^2(Y, \Z) \oplus \Z$ as before, then there is a topological manifold $X$ containing $X'$ and a commutative diagram 
\begin{equation*} \begin{CD}
 X' @>>>  X\\
@V \alpha' VV @VV \alpha V \\
Y' @>>> Y
\end{CD}\end{equation*}
such that $\alpha$ is proper and the $S^1$-action on $X'$ extends to an $S^1$-action on $X$ with $\alpha^{-1}(\{ 0 \} \times S) \cong S$, i.e., such that $S$ coincides with the fixed point locus of the $S^1$-action. Now a fibration $f:  X \rightarrow \R \times M_{\R}$ is defined by composing $\alpha:  X \rightarrow  Y$ with the $T^{n-1}$-fibration $\R \times M_{\R} \times (M_{\R} / M) \rightarrow \R \times M_{\R}$. 

In the two-dimensional case $S$ is just a finite set of points. If $\Log$ maps $S$ bijectively to $\Gamma$, then the singular fibres of $f$ are just once pinched $2$-tori.  

Let us treat the case $n=3$ and construct a suitable surface $S$ as follows. We use the notation of \S \ref{trop_hyp}.  
For an edge $e$ of $\Gamma$, consider the circle in $M_{\R}/ M$ given by 
\begin{equation} \label{cylinder_cycle}
 \delta_e = \left \{ [m] \in M_{\R}/M \, | \, \inn{n_{\check e}}{m} = \frac{1}{2}  \mod \Z \right \}. 
\end{equation}
Observe that from (\ref{edge_dual}) it follows that the slope of $\delta_e$ is $n_e$. 
Then consider the following cylinder inside $M_{\R} \times (M_{\R}/ M)$:
\begin{equation} \label{cylinder:e}
 S_e := e \times \delta_e.
\end{equation}
Clearly $\Log(S_e) = e$. Now we want to glue together all these cylinders to form a topological surface $S$. We do this by filling in what is left out at the vertices. More precisely, identifying $H_1(M_{\R}/M, \Z)$ with $M$, for every vertex $p \in \Gamma$ and edge $e$ containing $p$, orient the circle $S_e \cap \Log^{-1}(p)$ in the direction of $\epsilon_e n_e$, where $\epsilon_e=1$ if $n_e$ points outward from $p$ and $\epsilon_e = -1$ otherwise. Then the circle $S_e \cap \Log^{-1}(p)$ represents the class $\epsilon_en_e$. It then follows from (\ref{balance}) that there is a suitable $2$-chain $T_p$ in $\Log^{-1}(p)$ such that 
\[ \partial T_p = \bigcup_{p \in e} S_e \cap \Log^{-1}(p). \]
The topological surface $S$ is then defined as
\[ S = \bigcup_{\text{edges}} S_e \cup  \bigcup_{\text{vertices}} T_p. \]
Clearly $\Log(S) = \Gamma$. 

We have that $X$ is homeomorphic to the space defined in (\ref{mirror}) and the fibration $f$ constructed here is isotopic to the map defined in (\ref{dualfib}). In the two-dimensional case this is straight forward. In dimension $3$, this can be proved using Mikhalkin's results in \cite{mikh_pants}. The argument is as follows. Let $\Log_t: (\C^{\ast})^2 \rightarrow \R^2$ be the map $\Log_t(z) = ( \log_t |z_1|, \log_t |z_2| )$ and define 
\[ \mathcal A_t = \Log_t(S_t). \]
Then, using Viro's patchworking technique \cite{viro_patch}, Mikhalkin proves the following 
\begin{thm}
The sets $\mathcal A_t$ converge in the Hausdorff topology to $\Gamma$ as $t \rightarrow 0$. 
\end{thm}

In fact, it is true that for small $t$, there is a $C^0$ isotopy of $(\C^*)^2$ 
which takes $S_t$ to $S$. It does not appear that there is a good reference for
this in the literature; however, it can be shown using techniques of
\cite{viro}. See \cite{caputo} for a complete proof of
this fact.

\begin{rem} \label{algae:rem} 
Let us identify $M_{\R} \times (M_{\R} / M)$ with $(\C^*)^2 \cong \R^{2} \times T^2$. Let $v$ and $v'$ be the vertices of the edge $\check e$. We then have that $\delta_e$, as defined in (\ref{cylinder_cycle}), is the projection onto $T^2$ of the 
subset of $(\C^*)^2$ given by the equation
\[ z^v +  z^{v'} = 0. \]
Similarly, every vertex $p \in \Gamma$ corresponds to a simplex $P_j$ of the subdivision of $P$. Suppose that $v_0, v_1$ and $v_2$ are the vertices of $P_j$, then it can be verified that a $2$-chain $T_p$ is given by the closure of the projection onto $T^2$ of the set defined by the equation
\[ z^{v_0} + z^{v_1} + z^{v_2} = 0. \]
Figure~\ref{algae} depicts the set $T_p$ when $v_0=(0,0)$, $v_1=(1,0)$ and $v_2 = (0,1)$. 
\begin{figure}[!ht] 
\begin{center}
\includegraphics{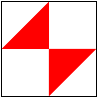}
\caption{A $2$-chain $T_p = S \cap \Log^{-1}(p)$ for a vertex $p \in \Gamma$. } \label{algae}
\end{center}
\end{figure}
The surface $S$ resulting from this construction is also called the complex tropical curve associated to the polynomial $h_t$, when all coefficients $a_j = 1$ (see \cite{Mikh-tropical}). 
\end{rem}

\subsection{Affine manifolds with singularities} \label{affine manifolds} Gross and Siebert have developed a program where mirror symmetry can be understood in terms of a duality of so-called  affine manifolds with singularities (see \cite{G-Siebert2003}, \cite{GrSi_re_aff_cx} or \cite{Gross_SYZ_rev} for a survey). We show here how these manifolds can be used to construct mirror symmetric torus fibrations. In fact these fibrations can be made into Lagrangian torus fibrations, as shown in \cite{CB-M}. An integral affine structure $\mathscr A$ on a topological manifold is an atlas of charts whose change of coordinate maps are affine maps with integral linear part, i.e., elements of $\R^n \rtimes \operatorname{SL}_n(\Z)$. Observe that an affine manifold comes with a natural flat connection $\nabla$. If $(x_1, \ldots, x_n)$ are affine coordinates, then the tangent vectors $\partial_{x_1}, \ldots, \partial_{x_n}$ form a basis of parallel sections of the tangent bundle, while the parallel one-forms $dx_1, \ldots, dx_n$ yield the dual basis. The $\Z$-span of the vectors $\partial_{x_j}$ forms a well-defined maximal lattice $\Lambda \subseteq TM$ (this is due to integrality of the affine structure). The dual lattice $\Lambda^*\subseteq T^*M$ is the $\Z$-span of the forms $dx_j$. 
The basic idea of the construction is to start with an integral affine manifold with singularities, $(B,\Delta,\mathscr A)$, where $B$ 
is a topological $n$-manifold and $\Delta$ is a closed set such that $B_0=B-\Delta$ is dense and has an integral affine structure $\mathscr A$. The set $\Delta$ is called the discriminant locus and we require that it has codimension $2$.  Then we have a symplectic manifold $X_0(B)$ defined by the exact sequence
\[
0\rightarrow \Lambda^* \rightarrow T^\ast B_0\rightarrow X_0(B) \rightarrow
0.
\]
The symplectic form on $X_0(B)$ is induced from the standard symplectic form on $T^*B_0$. This gives us a Lagrangian $T^n$ bundle $f_0:X_0(B) \rightarrow B_0$. Under certain hypotheses on $\Delta$ and on the affine structure one can (partially) compactify $X_0(B)$, in the sense that one can find a smooth $2n$-manifold $X(B)$ and a proper surjective map $f: X(B) \rightarrow B$ such that there is commutative diagram
\begin{equation} \label{compactify}
\begin{array}{ccc}
X_0(B) & \hookrightarrow & X(B) \\ 
\downarrow & \  & \downarrow \\ 
B_0& \hookrightarrow & B
\end{array}
\end{equation}
where the upper arrow is an open embedding and the lower arrow is the inclusion. Gross and Siebert define the notion of {\it positive} and {\it simple} integral affine manifold with singularities. These conditions are equivalent (in dimension $2$ and $3$) to certain restrictions on $\Delta$ and on the monodromy of $\Lambda$ around $\Delta$. With this assumption, a topological compactification was found by the first author in \cite{TMS}. In \cite{CB-M}, the second author and R. Casta\~no-Bernard found a symplectic compactification of $X_0(B)$, i.e., a symplectic structure on $X(B)$ which extends the standard one on $X_0(B)$ and such that $f$ is a Lagrangian fibration. In fact, in dimension $n=3$ the precise statement of the result is slightly more complicated. In dimension $n=2$, $\Delta$ consists of a finite collection of points and the symplectic compactification of $X_0(B)$ is achieved by gluing a standard model of a Lagrangian fibration over a disc with a nodal central fibre; 
this model is known in symplectic geometry as a simple focus-focus fibration. This construction gives compact symplectic 4-manifolds with Lagrangian 2-torus fibrations (e.g., a $K3$ surface). In dimension $n=3$, the positive and simple assumptions imply that $\Delta$ is a graph with trivalent vertices, labeled either positive or negative. In this case the affine structure around edges, positive, and negative vertices is isomorphic
to the one induced on the base of three different models of local Lagrangian fibrations: respectively the so-called \textit{generic}, \textit{positive} and \textit{negative} fibrations. The models for generic and positive fibrations can be regarded as $3$-dimensional analogues of focus-focus fibrations; in particular, they have a $T^2$-symmetry, they
have codimension 2 discriminant and are given by smooth fibration maps. On the other hand, the model for a negative fibration is $S^1$-invariant, the fibration is piecewise smooth and its discriminant locus has mixed codimension 1 and 2. 
This model can be regarded as a perturbation Gross's topological version of the negative fibration used in \cite{TMS}. 
These difficulties require us to redefine $\Delta$ by locally fattening the graph near negative vertices (see Figure \ref{negative_thick}) so that it has codimension $1$. As a consequence, the result in  \cite{CB-M} must be formulated as follows:

\begin{thm} \label{lagfib_thm} Let $(B,\Delta,\mathscr A)$ be a three-dimensional integral affine manifold with singularities which is positive and simple. For every negative vertex $v^- \in \Delta$ there is a small embedded $2$-disc $D_{v^-} \subset B$, containing a neighborhood of $v^-$ in $\Delta$, such that if we let $\Delta_{\blacklozenge} = \left( \bigcup_{v^{-}} D_{v^-} \right) \cup \Delta$, $B_{\blacklozenge} = B - \Delta_{\blacklozenge}$ and $X_{\blacklozenge}(B) = T^*B_{\blacklozenge}/ \Lambda^*$, then we can find a symplectic $X(B)$ and a proper surjective Lagrangian fibration $f: X(B) \rightarrow B$ such that diagram (\ref{compactify}) holds if we replace $B_0$ with $B_{\blacklozenge}$ and $X_0(B)$ with $X_{\blacklozenge}(B)$. Moreover $f$ is smooth over $X(B) - f^{-1} ( \bigcup_{v^{-}} D_{v^-})$.
\end{thm}

Notice that outside the discs $D_{v^-}$ the discriminant locus is codimension $2$. We refer to \cite{CB-M} for a more detailed description of these discs $D_{v^-}$ and the construction of the local models. 

\begin{figure}[!ht] 
\begin{center}
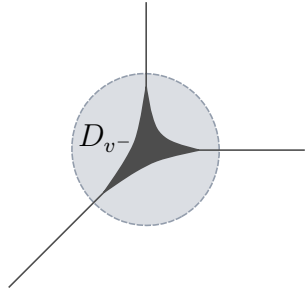
\caption{The discs $D_{v^-}$ containing the codimension $1$ part of $\Delta$ at negative vertices. Outside these discs the fibration is smooth.} \label{negative_thick}
\end{center}
\end{figure}
An important invariant of integral affine manifolds with singularities is the monodromy representation of the fundamental group of $B_0$. In fact the existence of the flat connection and integrality gives a representation of $$ \rho: \pi_1(B_0,p) \rightarrow \Gl(\Lambda_p) \cong \Gl(\Z, n).$$ 
Similarly we have the dual representation $\rho^*: \pi_1(B_0,p) \rightarrow \Gl(\Lambda^*_p)$.
 
In \cite{G-Siebert2003} and \cite{GrSi_re_aff_cx}, Gross and Siebert describe how, in certain cases, from an affine manifold with singularities $(B, \Delta, \mathscr A)$ one can construct a mirror one, denoted $(\check{B}, \check{\Delta}, \check{\mathscr A})$. The associated spaces $X(B)$ and $X(\check B)$ are mirror to each other. We do not wish to go into the details of the general construction, but only give ad hoc constructions for the examples we need. 

\subsection{The Lagrangian model} \label{lag_model} Here we define a mirror pair of affine manifolds with singularities whose spaces $X(B)$ and $X(\check B)$ from Theorem \ref{lagfib_thm} are homeomorphic, respectively, to $X$ and $\check X$ and whose Lagrangian fibrations are equivalent to the ones already described.  

\begin{ex} \label{bpos} $\mathbf{(X(\check B) \cong \check{X})}$ We describe $(\check B, \check \Delta, \check A)$. Let $\check{B} = \R \times M_{\R}$. Consider the function $\check \nu$ defined in (\ref{lt:nu}) and the associated tropical hypersurface $\Gamma$. Then $\check{\Delta} = \{ 0 \} \times \Gamma$. Define open subsets of $\check B$: 
\[ U^{+} = \check{B} - (\R_{\geq 0} \times \Gamma) \ \ \text{and} \ \  U^{-} = \check{B} - (\R_{\leq 0} \times \Gamma) \]
Let $\phi^-: U^- \rightarrow \R \times M_{\R}$ be the inclusion and define $\phi^+: U^+ \rightarrow \R \times M_{\R}$ by
\[ \phi^+(t,m) = (t + \check \nu(m), m), \]
where $t \in \R$ and $m \in M_{\R}$. The charts $(U^{+}, \phi^+)$ and $(U^{-}, \phi^-)$ define an integral affine structure on $\check{B} - \check{\Delta}$. Notice that at every point $p \in \check{B}_0$, we can identify $T^{\ast}_p \check{B}_0$ with $\R \times N_{\R}$ and we write a $1$-form as $a dt + v$, where $a \in \R$ and $v \in N_{\R}$. The space of monodromy invariant one-forms with respect to the monodromy representation $\rho^*$ can be identified with $\{ 0 \} \times N_{\R}$. All vertices of $\check{\Delta}$ are of positive type (see Example~\ref{aff+} below). It follows from \cite{TMS}, \S 3, that $X(\check{B})$ is homeomorphic to $\check{X}$.
\end{ex}

\begin{ex} \label{bneg} $\mathbf{(X(B) \cong X)}$.
Let us now describe $(B, \Delta, \mathscr{A})$. Here $B = \R \times N_{\R}$. The discriminant locus $\Delta \subset \{ 0 \} \times N_{\R}$ is constructed as follows. When $n=2$, $\Delta$ is just the set of barycenters of the segments $P_j$ forming the subdivision of $P$. When $n=3$, then $\Delta$ is the union of the following segments and straight rays. For every interior edge $\check e$ in the subdivision of $P$, take the two segments from the barycenter of $\check e$ to the barycenters of $P_e^+$ and $P_e^-$ respectively (see end of \S \ref{trop_hyp} for notation). The union of these two segments forms an edge of $\Delta$ which we denote by $e$. For every boundary edge $\check e$ in the subdivision of $P$, take the straight ray emanating from the barycenter of $P_e$ and passing through the barycenter of $\check e$ (see Figure~\ref{cek_delta}). Also in this case we denote this ray by $e$.

\begin{figure}[!ht] 
\begin{center}
\includegraphics{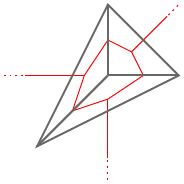}
\caption{The polytope $P$, its subdivision and the set $\Delta$.} \label{cek_delta}
\end{center}
\end{figure}

Observe that $\Delta$ is homeomorphic to the tropical curve $\Gamma$, and moreover it shares with $\Gamma$ the same combinatorial relationship with the subdivision of $P$. Every connected component $C$ of $N_{\R} - \Delta$ contains precisely one vertex  of the subdivision of $P$ which we denote by $v_C$. For every $C$, form the following open set of $B$: 
\[ V_C = C \cup \{ (t,v) \in \R \times N_{\R} \ | \ t \neq 0 \}, \]
and on $V_C$ define the map 
\[ \Phi_C(t,v) = \begin{cases}
	(t,v) &\ t< 0 \\
			(t, v + tv_C) &\ t \geq 0.
                 \end{cases} \]
The charts $(V_C, \Phi_C)$ define an integral affine structure on $B - \Delta$. We now compute monodromy. Given the point $p = (0,v_C) \in B_0$, we can identify $T^{*}_pB_0$ with $\R \times M_{\R}$ and denote $1$-forms by $a dt + m$, with $a \in \R$ and $m \in M_{\R}$.  Now let 
\[ V^+ = \{ t > 0 \} \ \ \text{and} \ \ V^- = \{ t < 0 \}. \]
Consider an edge $e$ of $\Delta$. It passes through the barycenter of the edge $\check e$. If $v$ and $v'$ are the vertices of $\check e$ such that $n_{\check e} = v'-v$, consider a path $\gamma_e$ going from $v'$ to $v$ passing inside $V^+$ and then going back to $v'$ passing inside $V^-$. It can then be calculated that 
\begin{equation} \label{bneg:mon}
 \rho^*(\gamma_e): a dt + m \mapsto (a - \inn{m}{n_{\check e}}) dt + m. 
\end{equation}
Then we have that the $\rho^*(\gamma_e)$-invariant one-forms are of type $a dt + m$, where $m \in \ker(n_{\check e})$. All trivalent vertices of $\Delta$ are of negative type. 
\end{ex}

\begin{ex} \label{ff-aff}Let $P = [0,1] \subseteq \R$ and $\nu: P \rightarrow \R$ the zero map. Then
\[ \check \nu(x) = \min \{0,x \} \]
and $\Gamma = \{ 0 \}$. Applying the constructions of Examples~\ref{bpos} and \ref{bneg} to this case we obtain isomorphic affine manifolds with singularities with only one singular point. These are called focus-focus models and all points of $\check{\Delta}$ or $\Delta$ in the $2$-dimensional versions of Examples~\ref{bpos} and \ref{bneg} are locally isomorphic to this example. Notice that at every point $p \in B_0$, the space of monodromy invariant tangent vectors is one-dimensional and the distribution of monodromy invariant tangent vectors is integrable. We call the integral lines of this distribution eigenlines. Notice that there is one eigenline passing through the singular point. When $\Delta$ has more than one point, each point has its own eigenlines given by local monodromy around it. 
\end{ex}

\begin{ex} \label{aff+} Suppose that $P$ is the standard simplex in $\R^2$ with vertices $(0,0)$, $(1,0)$ and $(0,1)$, and let $\nu: P \rightarrow \R$ the zero map. Then we have
\[ \check \nu(x,y) = \min \{ 0, x, y \} \] 
and $\Gamma$ is the set 
\[ \Gamma = \{ (-t,-t), \ t \geq 0 \} \cup \{ (t,0), \ t \geq 0 \} \cup \{ (0,t), \ t \geq 0 \}. \] 
Applying the construction of Example~\ref{bpos} to this case, we obtain $\check{\Delta}$ with a vertex of positive type. The corresponding Lagrangian torus fibration $\check{f}: X(\check{B}) \rightarrow \R^3$ has a positive singular fibre over the vertex of $\check{\Delta}$; see for example
\cite{splagI}, Example 1.2 for a full analysis of this fibration. The fibres over the edges of $\check{\Delta}$ are of generic type. Applying the construction of Example~\ref{bneg}, we obtain a vertex of negative type. 
\end{ex}

When the subdivision of $P$ is maximal, then all vertices of $\check{\Delta}$ (resp. $\Delta$) arising from Example~\ref{bpos} (resp. Example~\ref{bneg}) are of positive type (resp. negative type). 

The fibration $f: X(B) \rightarrow B$ constructed from Example \ref{bneg} via Theorem~\ref{lagfib_thm} has the following properties: 
\begin{itemize}
\item[a)] There is a Hamiltonian $S^1$-action on $X(B)$ whose fixed point set coincides with $\Crit f$. Let $\mu: X(B) \rightarrow \R$ be the moment map of this action. Then, for all $t \in \R$, $f$ restricted to $\mu^{-1}(t)$ is the composition of the quotient map $\alpha: \mu^{-1}(t) \rightarrow \mu^{-1}(t)/S^1$ with a regular $T^2$ fibration $\bar f: \mu^{-1}(t)/S^1 \rightarrow \{t \} \times N_{\R} \subset B$. Moreover $\Crit f \subseteq \mu^{-1}(0)$. 
\item[b)] Also $X_0(B)$ has an $S^1$-action given by translations in the $dt$ direction, i.e., $e^{2 \pi i s}\cdot (b, [\eta])= (b, [\eta + s dt])$, for all $b \in B_0$ and $\eta \in T^*_bB_0$. The restriction of this action to $X_\blacklozenge(B)$ coincides with the restriction of the $S^1$-action on $X(B)$ in point (a). 
\item[c)] There exists a smooth Lagrangian section $\sigma_0: B \rightarrow X(B)$ which extends the zero section on $X_{\blacklozenge}(B)$. 
\end{itemize}

Notice that from $(b)$ it follows that $\mu$ restricted to $X_{\blacklozenge}(B)$ is just the coordinate $t$. 
Let us denote
\[ X(B)_{\rd} = \mu^{-1}(0)/S^1 \ \ \ \text{and} \ \ \ X_0(B)_{\rd} = (\mu^{-1}(0) \cap X_0(B))/S^1.\]
It follows from $(b)$ that 
\begin{equation} \label{red_smooth}
  X_0(B)_{\rd} = (N_{\R}- \Delta) \times M_{\R}/M
\end{equation}
and the fibration $\bar f: X_0(B)_{\rd} \rightarrow \{0 \} \times N_{\R}$ is just the projection. 
Inside  $X(B)_{\rd}$  define the surface
\[ S := \alpha( \Crit f). \]
Via symplectic reduction $X(B)_{\rd}$ comes with a reduced symplectic form which blows up along $S$. 
The section $\sigma_0$ induces a reduced section $\bar \sigma_0: N_{\R} \rightarrow X(B)_{\rd}$ given by $\bar \sigma_0(b) = \alpha(\sigma_0(0, b))$. 
Since $\sigma_0$  is a section, it must avoid $\Crit f$, therefore $\bar \sigma_0(N_{\R}) \cap S = \emptyset$.

\begin{rem} As mentioned, $X(B)$ is homeomorphic to $X$ as defined in \eqref{mirror}. We also believe that it is symplectomorphic to it with the standard symplectic form \eqref{stsymp}, but this has not been proved yet. Abouzaid, Auroux and Katzarkov
\cite{AAK} use a different method to construct a piecewise smooth Lagrangian fibration on $X$ with the standard symplectic form and in all dimensions. Their fibration is a perturbation of the one in \eqref{dualfib}. It is possible that one can generalize our results to all dimensions using their Lagrangian fibration, but this requires some additional work since it does not yet give a detailed enough description to obtain our strongest results. One important advantage of our method is that it can be more easily generalized to construct Lagrangian spheres and sections in compact Calabi-Yau $3$-folds, see Example \ref{spheres_quintic}. 
\end{rem}

\subsection{Examples}

Let us discuss some simple $2$-dimensional examples.

\begin{ex} Let $P = [0,1]$ and choose $\nu = 0$. Then 
\[ \check{X}= \C^2 - \{ z_1z_2 - 1 = 0 \} \]
and there is a Lagrangian torus fibration on $\check X$ given
by
\[ \check{f} (z_1, z_2) = \left( \log|z_1 z_2 - 1|, \frac{|z_1|^2 - |z_2|^2 }{2} \right). \]
This fibration is just a standard example of focus-focus fibration, with a singular fibre over $(0,0)$ having the topology of a pinched torus (see 
\cite{splagI}, 
\cite{Gross_spLagEx}, and \cite{CB-M}). The mirror is
\[ X= \{ xy = z + 1 \}.\]
The fibration on $X$ is
\[ f(x,y,z) = \left( |x|^2 - |y|^2 , \log |z| \right). \]  
\end{ex}

\begin{ex} \label{fib:d2}
Now let $P = [0,m]$ for some positive integer $m$, with the maximal subdivision where $P_j=[j,j+1]$, for all $j=0, \ldots, m-1$. Consider the unique piecewise affine strictly convex function defined by $\nu(k) = \sum_{j=0}^{k} j $ for all $k \in P \cap \Z$. Then the affine toric variety associated to the cone over $P$ is
\[ V_P = \{ xy = z^m \}, \]
which is a singular variety with a $2$-dimensional $A_m$-singularity. The toric variety $V_{\Sigma}$ associated to the subdivision of $P$ is a crepant resolution of this singularity. We have that $\check X$ is obtained by removing a principal divisor from $V_{\Sigma}$ (as in formula (\ref{Xcek})). For instance, when $m=2$, then $\check{X}$ is an open set of $\mathcal O_{\PP^1}(-2)$. Notice that 
\[ \check{\nu}(t) = \min \left\{ k t + \sum_{j=0}^{k} j, \ \ k = 0, \ldots, m \right\}. \]
%\begin{figure}[!ht] 
%\begin{center}
%\input{res_affine.pdf_t}
%\caption{The polygon $\sigma$ and the base of the Lagrangian fibration with three singular fibres and the associated eigenlines} \label{twodim_aff}
%\end{center}
%\end{figure}

%Then the fibration on $\check{f}: \check{X} \rightarrow \R^2$ has discriminant locus the set $\check{\Delta} = \{ (-1, 0), \ldots, (-k,0) \}$. Around each point in $\check{\Delta}$, the fibration is of focus-focus type. Notice that the eigenlines passing through each point of $\check{\Delta}$ are all distinct and parallel (see Example~\ref{ff-aff}).  In the case $m=3$, Figure~\ref{twodim_aff} shows the polygon $\sigma$ associated to $\check X$ and a schematic picture of the affine structure on the base of the fibration.

On the mirror side, we have 
\[ X = \left \{ xy = \sum_{j=0}^{m} t^{\nu(j)}  z^j \right \}. \]
Notice that $X$ is a smoothing of the singularity $\check{X}_P$. If $\alpha_1(t), \ldots, \alpha_m(t)$ are the complex roots of the polynomial, let $r_j(t) = \log | \alpha_j(t)|$. Then the discriminant locus of $f$ is the set of points $\Delta = \{ (0, r_j(t))\,|\, j=1, \ldots, m \}$ (see Figure~\ref{twodim_aff_mir}).
\enlargethispage{1em}
\begin{figure}[!ht] 
\begin{center}
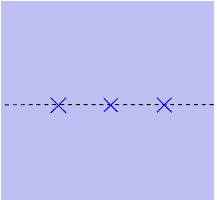
\caption{The base of the Lagrangian fibration on $X$.} \label{twodim_aff_mir}
\end{center}
\end{figure}
\end{ex}

\begin{figure}[!ht] 
\begin{center}
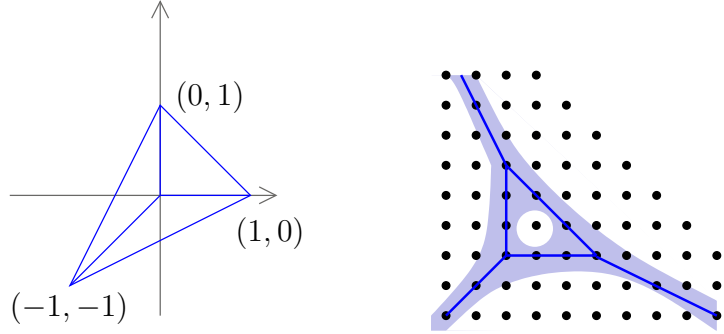
\caption{The polytope $P$ (left), the tropical curve $\Gamma$ and the amoeba (right).} \label{can_p2}
\end{center}
\end{figure}

Let us now do some $3$-dimensional examples.

\begin{ex} \label{triangle}
Let $P = \conv \{ (-1,-1), (1,0), (0,1) \}$ with its unique maximal subdivision (see Figure \ref{can_p2}). Then $\check{X}$ is an open set in the total space of the canonical bundle of $\PP^2$. Let $\nu$ be the unique piecewise affine strictly convex function such that $\nu(-1,-1) = \nu(1,0) = \nu(0,1) = 1$ and $\nu(0,0) = 0$. Then we have 
\[ \check \nu(s,t) = \min \{ 0, -s-t+1, s+1, t+1 \} \]
and the corresponding tropical curve $\Gamma$ is as in Figure~\ref{can_p2}.  The mirror is 
\[ X = \{ xy = t z_1^{-1} z_2^{-1} + tz_1 +tz_2 + 1 \} \]
and the discriminant locus of $f: X_{\Sigma} \rightarrow \R^3$ is a thickening of $\Gamma$.  
\end{ex}

\begin{figure}[!ht] 
\begin{center}
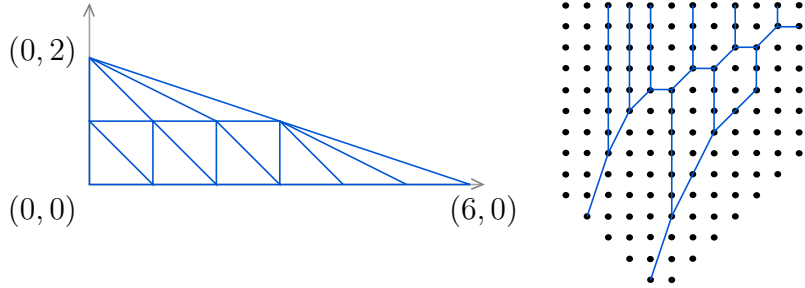
\caption{The polytope $P$ and the corresponding tropical curve.} 
\label{a2d:trop}
\end{center}
\end{figure}

\enlargethispage{1em}
\begin{ex} \label{a2d:mirror}
For some positive integer $d$, let 
\begin{equation} \label{a2d:polytope}
P = \conv \{ (0,0), (0,2), (2d,0) \} 
\end{equation}
together with some maximal subdivision and a piecewise affine strictly convex function $\nu$ (see Figure~\ref{a2d:trop} for an example when $d = 3$). We then have $\check{X}$ as an open set in the corresponding toric variety.  Its mirror $X$ has equation 
\begin{equation} \label{a2d:smoothing}
  X= \left\{ xy = \sum_{j \in P \cap \Z^2} t^{\nu(j)} z^j \right\}.
\end{equation}
Notice that $X$ is a smoothing of the $A_{2d-1}$-singularity whose equation is given by $xy=z^2+w^{2d}$. 
\end{ex}

\section{Lagrangian sections and line bundles}
In this section we give a precise and explicit classification of the space of sections, up to isotopy, of the torus fibration $f: X \rightarrow \R^n$ and we construct a Lagrangian representative for each class.  We then give a conjectural homological mirror symmetry correspondence between sections of $f$ and line bundles on $\check{X}_{\Sigma}$. In dimension $2$, Chan and Ueda \cite{chan:ueda} construct a wrapped Fukaya category generated by these sections and prove that the correspondence gives an embedding of the category into the derived category of coherent sheaves of $\check X$.

\subsection{Line bundles on $\check X$} \label{ss:line_bundles}
Line bundles on toric varieties correspond to support functions modulo linear functions.  Recall that a support function is an integral piecewise linear function $\phi: | \Sigma | \rightarrow \R$ defined on the support $|\Sigma|$ of the fan. Denote by $\Sigma(1)$ the set of $1$-dimensional cones of $\Sigma$ and for every $\rho \in \Sigma(1)$, let $u_{\rho}$ be the primitive integral generator of $\rho$ and let $D_{\rho}$ be the toric divisor corresponding to $\rho$. Then $\phi$ gives the Cartier divisor 
\begin{equation} \label{function:divisor}
   D_{\phi} = \sum_{\rho \in \Sigma(1)} \phi(u_{\rho}) D_{\rho}.
\end{equation}
Here we use a sign convention which is opposite to the one in \cite{fulton:toric} and \cite{cox:little:schenck}. We denote the line bundle corresponding to $D_{\phi}$ by $\mathcal L_{\phi}$.

Two support functions define the same line bundle if and only if their difference is a linear function.  In our case, where the fan $\Sigma$ in $\R \times N_{\R}$ is defined from a convex lattice polytope $P \subset N_{\R}$ and a smooth subdivision of $P$, support functions on $\Sigma$ are in a one-to-one correspondence with integral piecewise affine maps on $P$, whose domains of affineness are unions of the simplices $P_j$ of the subdivision. In fact, given a piecewise affine map $\phi: P \rightarrow \R$, the corresponding support function is defined as $\tilde \phi(t,tp) = t \phi(p)$ for all $p \in P$. Vice versa, given $\tilde \phi$, we have that $\phi = \tilde \phi|_{\{1 \} \times P}$ is piecewise affine. By slight abuse of notation we will continue to call a piecewise affine function $\phi$ on $P$ a support function.

\begin{defi} \label{kinks}
In the two-dimensional case, $P = [0,m]$ with $P_j = [j,j+1]$ for $j=0, \ldots, m-1$. 
Given a support function $\phi$ on $P$, let $k_{j} \in \Z$ be the difference between the slopes of $\phi|_{P_{j}}$ and $\phi|_{P_{j-1}}$ for $j = 1, \ldots, m-1$. 
Then the $m-1$-tuple $K = (k_1, \ldots, k_{m-1})$ uniquely determines the function $\phi$ up to addition of an affine function.
In dimension $3$, let $\phi$ be a support function on $P$ and let $\check e$ be an interior edge in the subdivision of $P$.  
Let $m^+_e, m^-_e \in M$ be the linear parts of $\phi|_{P^+_e}$ and $\phi|_{P^-_e}$ respectively (see \S\ref{trop_hyp} for notation). Then 
\[ \inn{m^+_e-m^-_e}{n_{\check e}} = 0, \]
i.e., we must have that 
\begin{equation} \label{bundles:cond}
 m^+_e - m^-_e = k_e n_e 
\end{equation}
for some $k_e \in \Z$. Notice also that $k_e$ does not depend on the sign of $n_e$. We call $k_e$ the {\it kink} of $\phi$ along $\check e$.
If we let $\mathscr E$ be the collection of all bounded edges of $\Gamma$, we then have a collection of integer numbers $K = (k_{e})_{e \in \mathscr E}$. We call a collection $K$ of integer numbers associated to a support function on $P$ the {\it kinks} of $\phi$.
\end{defi}

In dimension $2$, every interior point $j=1, \ldots, m-1$ of $P$ corresponds to toric divisors $D_j \cong \PP^1$ of $\check X$. 
Let $\mathcal L_{K}$ be the line bundle on $\check X$ corresponding to $K = (k_1, \ldots, k_{m-1})$.
Then we have
\begin{equation} \label{dim2_inters}
\mathcal L_{K}|_{D_j} = \mathcal{O}_{\PP^1}(k_j). 
\end{equation}
Similarly, in dimension $3$, every vertex in the subdivision of $P$ corresponds to a toric divisor of $\check X$ and every interior edge $\check e$ corresponds to a compact $1$-dimensional toric stratum, which we denote by $\PP^1_{\check e}$. If $D$ and $D'$ are the two toric divisors corresponding to the vertices of $\check e$, then $\PP^1_{\check e} = D \cap D'$.  If $K = (k_{e})_{e \in \mathscr E}$ is the set of kinks of a support function $\phi$ and $\mathcal L_{K}$ is the corresponding line bundle, then
\[ \mathcal L_{K}|_{\PP^1_{\check e}} = \mathcal{O}_{\PP^1}(k_e). \]
So the kinks $k_e$ are the intersection numbers of the line bundle 
$\mathcal L_{K}$ with one-dimensional strata.

Let $\mathscr C$ be the set of bounded components of $M_{\R} - \Gamma$. Clearly $C \in \mathscr C$ if and only if $v_C$ is an interior vertex of the subdivision of $P$. Notice also that $e$ is an edge of $\partial\bar C$ if and only if $v_C$ is a vertex of  $\check e$. 
Fix an orientation of $C$. If $(k_e)_{e \in \mathscr E}$ is a set of kinks of a support function $\phi$, then we must have
\begin{equation} \label{rel_inters}
 \sum_{e \subseteq \partial\bar C} \epsilon_e \, k_e \, n_e = 0, 
\end{equation}
where $\epsilon_e =1$ if $n_e$ agrees with the orientation of $\partial\bar C$, otherwise $\epsilon_e = -1$. 
For every $C \in \mathscr C$ define the map $\phi_C: \Z^{\mathscr E} \rightarrow M$ by 
\[ \phi_{C}( (k_e)_{e \in \mathscr E}) = \sum_{e \subseteq \partial\bar C} \epsilon_e \, k_e \, n_e. \]
 Now define $\Phi: \Z^{\mathscr E} \rightarrow M^{\mathscr C}$ by 
\begin{equation} \label{sec:fi}
 \Phi = (\phi_{C})_{C \in \mathscr C}.
\end{equation}

\begin{lem} \label{sect-bundles}
The group $\operatorname{Pic}(\check X)$ is isomorphic to $\ker \Phi$. 
\end{lem}

\begin{proof} Given a support function $\phi$, from (\ref{rel_inters}) it follows that its kinks $(k_e)_{e \in \mathscr E}$ belong to $\ker \Phi$. Moreover if we add to $\phi$ an affine function, these numbers do not change.  On the other hand, given a collection $(k_e)_{e \in  \mathscr E}$ in $\ker \Phi$, fix a simplex $P_0$ in the subdivision and set $\phi|_{P_0}= 0$. Then equations (\ref{bundles:cond}) uniquely determine a support function $\phi$.
\end{proof}

\subsection{The space of sections on $X$} We have given three slightly different constructions of the space $X$ and the fibration $f$: through the formulas (\ref{mirror}) and (\ref{dualfib}), the complex tropical model of \S\ref{topological} or the Lagrangian model $X(B)$ in Example~\ref{bneg}. We will freely switch from one construction to the other. 

Let us consider the Lagrangian model $f: X(B) \rightarrow B$  from Example~\ref{bneg}. Denote by $\iota: B_0 \rightarrow B$ the inclusion. It can be shown that the space of sections of $f$, up to isotopy, is classified by $H^1(B, \iota_{\ast} \Lambda^*)$. This can be seen as follows.  Let $\{ U_i \}$ be a good cover of $B$. If $\sigma$ is a section, we can locally lift $\sigma|_{U_i}$ to a section $\sigma_i$ of $T^*B$. Then, on overlaps, $s_{ij} = \sigma_i - \sigma_j$ defines a cocycle $[\sigma] \in H^1(B, \iota_{\ast} \Lambda^*)$ (see \cite{splagI}, \S 3). Moreover $[\sigma] = 0$ if and only if $\sigma$ is isotopic to the zero section.  Indeed, let $\sigma_t$ be an isotopy such that $\sigma_0$ is the zero section. Locally we can lift to an isotopy $\sigma_{i,t}$ of sections of $T^*B$ so that $\sigma_{i,0} = 0$. Since $\iota_{\ast} \Lambda$ is discrete, $\sigma_{i,t} - \sigma_{j,t}$ is constant in $t$ and thus equal to zero for all $t$.  Hence $[\sigma_1]=0$. On the other hand if $[\sigma] = 0$, then $\sigma_i-\sigma_j = \lambda_i-\lambda_j$ for local sections $\lambda_i$ and $\lambda_j$ of $\iota_{\ast}\Lambda$. Then $\sigma$ has a global lift $\tilde \sigma: B \rightarrow T^*B$ defined by $\tilde \sigma|_{U_i} = \sigma_i - \lambda_i$. An isotopy with the zero section is then defined by $\tilde \sigma_t = t \tilde \sigma$ projected to $T^*B/\Lambda$.

We now explicitly compute $H^1(B, \iota_{\ast} \Lambda^*)$ in the cases $\dim M_{\R} = 1$ and $2$. Observe that $\iota_{\ast} \Lambda^*$ has a globally defined section, namely $dt$, corresponding to the fact that $X(B)$ has an $S^1$-action. Let us take the quotient by the constant subsheaf spanned by this section ($\cong \Z$), i.e., we have a short exact sequence
\[ 0 \rightarrow \Z \rightarrow \iota_{\ast} \Lambda^* \rightarrow \mathcal G \rightarrow 0 \]
for some sheaf $\mathcal G$. Now notice that at a vertex $p$ of $\Delta$, $\mathcal G_p = 0$, since $\iota_{\ast} \Lambda^*$ has stalk $\Z$ at $p$. Therefore $H^0(B, \mathcal G) = 0$.  The long exact sequence in the cohomology of sheaves gives
\[ H^1(B, \iota_{\ast} \Lambda^*) \cong H^1(B, \mathcal G). \]

We now show how to compute $H^1(B, \mathcal G)$.  Let us construct a 
constructible sheaf $\mathcal G'$ on $B$, supported over $\Gamma$, as follows. Assume first $\dim M_{\R} = 2$. On a small neighborhood $U$ of a vertex of $\Gamma$, $\mathcal G'(U) = M$. When $U$ is a neighborhood of a point in the interior of an edge $e$, then $\mathcal G'(U) = M / \Z n_e$. Restriction functions are just projections to the quotients. We set $\mathcal G'$ to be zero outside of $\Gamma$. When $\dim M_{\R} = 1$ and $\Delta$ has $k$ points, then we let $\mathcal G'$ be $M \cong \Z$ at the points of $\Delta$ and $0$ away from them. We have the following:

\begin{lem} \label{sections} 
In both cases, $\dim M_{\R} = 1$ or $2$, there is a short exact sequence
\begin{equation} \label{sect_seq}
 0 \rightarrow M \rightarrow H^0(\mathcal G', B) \rightarrow H^1( \mathcal G, B) \rightarrow 0. 
\end{equation}
\end{lem}

\begin{proof}
It can be easily seen, from the expression of monodromy (\ref{bneg:mon}), that $\mathcal G|_{B_0}$ can be identified with the constant sheaf $M$. If $\dim M_{\R} = 2$, recall that $\Delta$ is homeomorphic the tropical curve $\Gamma \subset M_{\R}$. It is convenient here to identify $\Delta$ with $\Gamma$. 
It follows from the monodromy formula (\ref{bneg:mon}), that over a small neighborhood $U$ of a point in the interior of $e$, $\mathcal G(U) = \Z n_e$. Therefore we have a short exact sequence
\[ 0 \rightarrow \mathcal G \rightarrow M \rightarrow \mathcal G' \rightarrow 0. \]
The corresponding long exact sequence gives (\ref{sect_seq}). When $\dim M_{\R} = 1$ and $\mathcal G'$ is defined as above, then the same sequence also holds. 
\end{proof} 

It is clear that when $M \cong \Z$ and $\Delta$ has $k$ points, then $H^0(\mathcal G', B) \cong \Z^k$. So we conclude that $H^1(B, \iota_{\ast} \Lambda^*) \cong \Z^{k-1}$. 

\begin{ex} \label{p2:sections}
It is easy to compute $H^{1}(B, \iota_{\ast} \Lambda^*)$ in the case $B$ is constructed from Example~\ref{triangle}. Observe that in this case $\R^2 - \Gamma$ has only one bounded component, which we denote $C$. Let $e_1=(1,0)$, $e_2=(-1,1)$ and $e_3 = (0,-1)$ be integral primitive tangent vectors to the edges of $\Gamma$ which bound $C$. Let $p_1 = (-1,-1)$, $p_2 = (2,-1)$ and $p_3 = (-1,2)$ be the vertices of $\Gamma$. 
Elements of $H^{0}(B, \mathcal G')$ can be found as follows. For every vertex $p_j$ of $\Gamma$ choose an element $m_j\in\Z^2$: 
$p_j \mapsto m_j$. Then, these choices give an element of $H^{0}(B, \mathcal G')$ if and only if there exist 
$n_j \in \Z$ such that 
\[ m_{j+1} - m_{j} = n_j e_j, \ \ j=1,2,3 \]
and the indices are assumed cyclic. It is clear that such a system has solutions if and only if 
\[ \sum_{j=1}^3 n_j e_j = 0, \]
which holds if and only if $n_1 = n_2 = n_3 = k$ for some $k \in \Z$. Once this condition is satisfied, then we have a unique solution for each choice of $k$ and initial choice $m_1 \in \Z^2$. So
$H^{0}(B, \mathcal G') \cong \Z^3$, corresponding to a choice of $k$ and of $m_1$. We conclude, from Lemma~\ref{sections}, that for this example
\[ H^{1}(B, \iota_{\ast} \Lambda^*) = \Z. \]
\end{ex}

The argument of the previous example can be generalized as follows. Let $\Phi: \Z^{\mathscr E} \rightarrow M^{\mathscr C}$ be the map defined in (\ref{sec:fi}). Then we have 
\begin{lem} \label{sect:cond} When $(B, \Delta, \mathcal A)$ is as in Example~\ref{bneg}, then
\[ H^{1}(B, \iota_{\ast} \Lambda^*) \cong \ker \Phi. \]
This establishes a one-to-one correspondence between isomorphism classes of line bundles on $\check X$ and isotopy classes of sections of $X$.
\end{lem}
The proof uses the same argument as in Example~\ref{p2:sections}. 

\subsection{Topological construction of sections} \label{top_sec} We give an explicit construction of the sections, which will also explain the result of the last lemma. For this purpose we use the complex tropical model. Recall that the fibration $f: X \rightarrow \R \times M_{\R}$ was defined as the composition of $\alpha: X \rightarrow Y$ and the $T^{n-1}$ fibration $Y \rightarrow \R \times M_{\R}$, where $Y = \R \times M_{\R} \times (M_{\R} / M)$. The map $\alpha$ is the projection on the quotient with respect to an $S^1$-action. The fixed point set of the action is a set $S$, which is a finite set of points in dimension $n=2$ or a surface in dimension $n=3$. The strategy for the construction of a section of $f$ is as follows. 
First construct a section $\bar \sigma$ of $\Log: M_{\R} \times (M_{\R} / M) \rightarrow M_{\R}$ such that $\bar \sigma(M_{\R}) \cap S = \emptyset$, then extend it to a section $\sigma_{Y}$ of $Y \rightarrow \R \times M_{\R}$ which coincides with $\bar \sigma$ on $\{ 0 \} \times M_{\R}$. Finally lift $\sigma_Y$ to a section $\sigma_{X}$ of $f$.  Notice that the last two steps are easy. In fact, to extend $\bar \sigma$ to $\sigma_{Y}$ one can define $\sigma_{Y}(t,q) = (t, \bar \sigma(q)$). Moreover any two extensions are isotopic (outside of $\{ 0 \} \times M_{\R}$ there is no constraint on $\sigma_{Y}$). Notice that $\alpha$ restricted to the image of $\sigma_{Y}$ is a trivial $S^1$-bundle, since the image of $\sigma_{Y}$ avoids $\{ 0 \} \times S$. 
Then a lift $\sigma_{X}$ exists and any two lifts are isotopic. In the following we will denote by $\bar \sigma$ the section of $\Log$ and by $\sigma$ the final section of $f$, i.e., $\sigma = \sigma_X$.  We call $\bar \sigma$ the {\it reduced section}. So let us construct sections $\bar \sigma: M_{\R} \rightarrow M_{\R} \times (M_{\R} / M)$ such that $\bar \sigma(M_\R) \cap S = \emptyset$. 

\subsection{Sections in dimension $2$}\label{sec:d2} In dimension $2$, assume $M = \Z$. Then we have $Y = \R \times \R \times S^1$. Let $\Gamma = \{ 0,1, \dots, m \} \subset \R$ and $S =\{ (0, i), \ldots, (m,i) \} \subset \R \times S^1$ (we think of $S^1 \subset \C$). We construct reduced sections $\bar \sigma: \R \rightarrow \R \times S^1$ avoiding $S$.  For every $m$-tuple of integers $K =(k_1, \ldots, k_m)$ define 
\[ \bar \sigma_K(s) = \begin{cases} 
	(s, e^{2k_r \pi i s}) &\ r-1 \leq s \leq r, \ \ r=1, \ldots, m \\
			     (s, 1) &\ \text{otherwise}.
                  \end{cases}  \]
Then $\bar \sigma_K$ is a well-defined section such that $\bar \sigma_K(\R) \cap S = \emptyset$ (see Figure~\ref{2dsections} for a picture of a section). In the case $K=(0, \ldots, 0)$, we will consider the corresponding section as the zero section and denote it by $\bar \sigma_0$. We will call the numbers $K$ the {\it twisting numbers} of the section.  Two sections $\sigma_K$ and $\sigma_{K'}$  are isotopic if and only if $K=K'$ (see discussion below).

\begin{figure}[!ht] 
\begin{center}
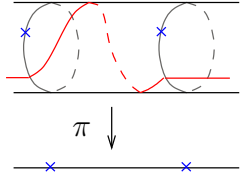
\caption{The section $\bar \sigma$ (red line) and the set $S$ (blue crosses).} \label{2dsections}
\end{center}
\end{figure}

\subsection{Sections in dimension $3$} \label{sec:d3} We consider the surface $S \subseteq M_{\R} \times (M_{\R} / M)$ as constructed in \S\ref{topological}. The zero section $\bar \sigma_0: M_{\R}  \rightarrow M_{\R} \times (M_{\R} / M)$ does not intersect $S$. In fact neither the circles $\delta_e$ nor the $2$-chains $T_p$ contain $0 \in M_{\R} / M$. To construct a general section proceed as follows.  For every bounded edge $e \subset \Gamma$, choose an integer $k_e$.  If $p^+, p^-$ are the two vertices of $e$ corresponding to $P^+_e$ and $P^-_e$ respectively, then define
\begin{equation} \label{sec:edge}
  \bar \sigma((1-t)p^+ + t p^-) = ( (1-t)p^+ + t p^-, - [k_e n_e t]), 
\end{equation}
so that $\bar \sigma|_e$ wraps $- k_e$ times around the cycle $n_e$ in the fibre $M_{\R} / M$. Clearly $\bar \sigma(e) \cap S = \emptyset$. It is also clear that $n_e$ is the only cycle around which $\bar \sigma|_e$ can wrap non-trivially if we want $\bar \sigma(e) \cap S = \emptyset$. When $e$ is an unbounded edge, simply let $\bar \sigma|_e$ be the zero section. Clearly this gives a well-defined section $\bar \sigma: \Gamma \rightarrow M_\R \times M_{\R} / M$, which we want to extend to all of $M_{\R}$. If $C$ is a bounded component of $M_{\R} - \Gamma$, then an extension of $\bar \sigma$ to $\bar \sigma: C  \rightarrow M_{\R} \times (M_{\R} / M)$ exists if and only if 
\begin{equation} \label{sec:comp}
   \sum_{e \subseteq \partial\bar C} \epsilon_e \, k_e \, n_e = 0. 
\end{equation}
where $\epsilon_e =1$ if $n_e$ agrees with a fixed orientation of $\partial\bar C$, otherwise $\epsilon_e = -1$. 
It is now clear that an extension of $\bar \sigma$ to all of $M_{\R}$ exists if and only if the set of choices $(k_e)_{e \in \mathscr E} \in \ker \Phi$ as in Lemma~\ref{sect:cond}. We call the collection $K=(k_e)_{e \in \mathscr E} \in \ker \Phi$ the {\it twisting numbers} of the section.
\begin{lem} \label{sec:iso}
Let $\bar \sigma$ and $\bar \sigma'$ be two sections of the $\Log$ map constructed as above from the twisting numbers $(k_e)_{e \in \mathscr E}$ and $(k'_e)_{e \in \mathscr E}$ respectively. Then the two corresponding sections $\sigma$ and $\sigma'$ of $f$ are isotopic if and only if 
\begin{equation} \label{isotop}
 (k_e)_{e \in \mathscr E} = (k'_e)_{e \in \mathscr E}. 
\end{equation}
\end{lem}

\begin{proof}
It is easy to show that an isotopy of sections between $\sigma$ and $\sigma'$ exists if and only if there exists an isotopy $\bar \sigma^t$, $t \in [0,1]$ between the reduced sections $\bar \sigma$ and $\bar \sigma'$ such that 
\begin{equation} \label{sigma:s}
  \bar \sigma^t(M_{\R}) \cap S = \emptyset, \ \ \text{for all} \ t \in [0,1].
\end{equation}
If (\ref{isotop}) holds, then an isotopy $\bar \sigma^t$ can be constructed as follows.  Given $\bar \sigma$, let $\tau_{\sigma}: M_{\R}  \rightarrow M_{\R} / M$ be such that $\bar \sigma(p) = (p, \tau_{\sigma}(p))$ and denote by $\tilde \tau_{\sigma}: M_{\R} \rightarrow M_{\R}$ a lift of $\tau_{\sigma}$ to the universal cover. Condition \eqref{isotop} implies that $\tau_{\sigma}|_{\Gamma} = \tau_{\sigma'}|_{\Gamma}$, in particular we can assume that also $\tilde \tau_{\sigma}|_{\Gamma} = \tilde \tau_{\sigma'}|_{\Gamma}$. Now define the isotopy $\bar \sigma^t(p) = (p, [ t \tilde \tau_{\sigma'}(p) +(1-t) \tilde \tau_{\sigma}(p)])$. We have that \eqref{sigma:s} holds since $\bar \sigma^t|_{\Gamma} = \bar \sigma|_{\Gamma} = \bar \sigma'|_{\Gamma}$ for all $t \in [0,1]$. 

On the other hand suppose $\sigma^t$ exists, satisfying (\ref{sigma:s}). Fix a bounded edge $e \subset \Gamma$ and let $\delta_e$ be the $1$-cycle defined in (\ref{cylinder_cycle}). Condition (\ref{sigma:s}) implies that $\tau_{\sigma^t}(e) \subseteq T^2 - \delta_e$ for all $t \in [0,1]$ and all edges $e \subset \Gamma$. Now assume that lifts $\tilde \tau_{\sigma^t}$, $\tilde \tau_{\sigma}$ and $\tilde \tau_{\sigma'}$ have been chosen so that $\tilde \tau_{\sigma^t}$ is an isotopy between $\tilde \tau_{\sigma}$ and $\tilde \tau_{\sigma'}$.
\begin{figure}[!ht] 
\begin{center}
\includegraphics{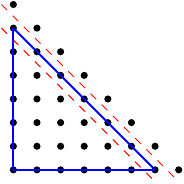}
\caption{The lift $\tilde \tau_{\sigma} (\partial\bar C)$ (continuous line) and the lifts of $\delta_e$ (dashed).} \label{sec:isotop}
\end{center}
\end{figure}
Consider a bounded component $C$ of $M_\R - \Gamma$. Then $\tilde \tau_\sigma(\partial\bar C)$ (resp. $\tilde \tau_{\sigma'}(\partial\bar C)$) is a closed polygonal curve in $M_\R$ with vertices in $M$ and with an edge parallel to $n_e$ and of length $|k_e|$ (resp. $|k'_e|$), for every edge $e \subset \partial\bar C$. 
Suppose that (\ref{isotop}) doesn't hold. Then there must be at least one bounded edge $e$ of $\Gamma$ such that $\tilde \tau_\sigma(e)$ and $\tilde \tau_{\sigma'}(e)$ lie on parallel but distinct lines in $M_\R$ with slope $n_e$ and containing integral points. Notice that lifts of $\delta_e$ are also parallel lines with slope $n_e$, but never intersecting $M$. Figure~\ref{sec:isotop} depicts this situation when $C$ is a standard simplex. It can now be seen that if $\tilde \tau_\sigma(e)$ and $\tilde \tau_{\sigma'}(e)$ lie on parallel but distinct lines, then $\tilde \tau_{\sigma^t}(e)$ must intersect a lift of $\delta_e$ at some point, for some value of $t$. This is a contradiction and thus we have proved the lemma. 
\end{proof}
Thus the isotopy class of a section is uniquely determined by the twisting numbers $K = (k_e)_{e \in \mathscr E} \in \ker \Phi$. We denote by $\bar \sigma_K$ the reduced section and by $\sigma_K$ the corresponding section of $f$. 

\subsection{Lagrangian sections} \label{lag_sections}
Let $f:X(B) \rightarrow B$ be the Lagrangian model of the fibration given in Example \ref{bneg}. Here we will construct an explicit Lagrangian representative for any class $[\sigma]$ of a section in the space $H^1(B, \iota_{\ast} \Lambda^*)$. We point out that $X(B)$, as constructed in \cite{CB-M}, does not contain the full $X_0(B)$, but the slightly smaller subset $X_{\blacklozenge}(B)$. Nevertheless our sections will be constructed so that they coincide with the fixed zero section on a neighborhood of the disks $D_{v^-}$ containing the vertices of $\Delta$, so that we will never have to worry about the technicalities of the fibration near a vertex of $\Delta$. For this reason, in what follows, we will continue to work on $X_0(B)$ instead of the smaller set $X_{\blacklozenge}(B)$. We will also work with the $S^1$-actions and the reduced spaces $X(B)_{\rd}$ and $X_0(B)_{\rd}$ described at the end of \S \ref{lag_model}.

Let $K = (k_e)_{e \in \mathscr E} \in \ker \Phi$ be a set of twisting numbers determining the isotopy class of a section in $H^1(B, \iota_{\ast} \Lambda^*)$. These numbers also determine the support function $\phi_K: P \rightarrow \R$ of a line bundle on $\check X$. Since we now fix the support function, in the following we set $\phi:=\phi_K$. 

\begin{lem} \label{subdiv_P}
The function $\phi$ can be extended, as a continuous piecewise affine map, to the whole of $N_{\R}$. 
\end{lem}

\begin{proof} First we define a subdivision of $N_{\R}$. Fix some euclidean metric on $N_{\R}$ and orientation of $P$. For every edge $E$ of $P$, let $n'_{E} \in N_{\R}$ be a non-zero vector which is orthogonal to $E$ and points outside of $P$ .  If $F$ is another edge of $P$ and $v \in E \cap F$ is a vertex of $P$, let $Q_{v}$ be the infinite polyhedron bounded by the two rays emanating from $v$ in the direction of $n'_{E}$ and $n'_{F}$.  If $\check e$ is an edge in the subdivision of $P$ contained in an edge $E$ of $P$, define the infinite polyhedron
\[ Q_e = \{ p + t n'_{E},  \ \ p \in \check e, \ \ t \geq 0 \}. \]
Clearly the polyhedra $Q_v$ and $Q_e$, together with the polyhedra of the subdivision of $P$, define a subdivision of $N_{\R}$. Notice that $Q_v$ and $Q_e$ are not necessarily integral polyhedra, but this will not matter. Let us extend $\phi$ as follows. On the polyhedron $Q_v$ define $\phi$ to be the constant function equal to $\phi(v)$. On $Q_e$ define it to be 
\[ \phi(p+t n'_{E}) = \phi(p). \]
This defines an extension of $\phi$. 
\end{proof}

\begin{rem} \label{extend:pe} Given the above subdivision of $N_{\R}$ and the extension of $\phi$, we can extend the definitions of $P^{\pm}_e$ (see the end of \S \ref{trop_hyp}), and the definitions of $m^{\pm}_e$ and $k_e$ (see Definition \ref{kinks}) also to the boundary edges in the subdivision of $P$. The only difference is that in the case $P^{\pm}_e$ is one of the unbounded polyhedra, then $m^{\pm}_e$ and $k_e$ are not necessarily integral. In fact, if $\check e$ is an edge in the subdivision of $P$ contained in the boundary edge $E$ and  $P^+_e=Q_e$, then the construction above implies that $m^+_e$ maps $n_{\check e}$ to some integer and $n'_E$ to zero. Since $n_{\check e}$ and $n'_E$ do not necessarily form a basis of $N$, $m^+_e$ is not necessarily integral (e.g. we could have $n_{\check e} =(1,1)$ and $n'_E =(1,-1)$, if $n_{\check e}$ is mapped to $1$ then $m^+_e$ is not integral). This, however, will not be a problem in the following construction.
\end{rem}

We now want to define suitable smooth approximations $\tilde \phi_{\epsilon}$ of $\phi$ depending on $\epsilon >0$ and then define the reduced section as the graph of $d \tilde \phi_{\epsilon}$. The main idea is to modify $\phi$ only near the $1$-skeleton of the subdivision of $N_{\R}$, i.e. we only smooth out the corners of $\phi$, so that sufficiently away from the $1$-skeleton it remains affine. This implies that the reduced section will coincide with the zero section on a neighborhood of the negative vertices of the discriminant locus $\Delta$. This is convenient since it automatically implies that it extends to all singular fibres in this neighborhood, without further work. The most technical part will be to show that the section extends to singular fibres over the edges of $\Delta$, this is proved in Lemma \ref{red_sect}.

A standard way to smooth a function is to use the convolution product with a so-called ``mollifier''. Choose on $N_{\R}$ an auxiliary inner product, giving a norm $\| \cdot \| $. Define the following smooth, compactly supported function on $N_{\R}$: 
\begin{equation} \label{mollifier}
\mu_{\epsilon}(x) = \begin{cases} 
	A_{\epsilon}^{-1} \exp \left( \frac{1}{\| x \|^2- \epsilon^2} \right), &\ \text{if} \ \|x\| < \epsilon; \\
				 0, &\ \text{otherwise}.
                               \end{cases}
\end{equation}
where $A_{\epsilon}$ is some constant chosen so that $\int_{N_{\R}} \mu_{\epsilon} = 1$. This is an example of a mollifier.

Now recall that the convolution product of $\mu_{\epsilon}$ with a continuous function $h: N_{\R} \rightarrow \R$ is defined by
\begin{equation} \label{convolution}
 h  \ast \mu_{\epsilon} (x) := \int_{N_{\R}} h(y)\mu_{\epsilon}(x-y) dy = \int_{N_{\R}} h(x-y)\mu_{\epsilon}(y) dy.
\end{equation}
It is well known that $h \ast \mu_{\epsilon}$ is smooth and converges to $h$ in the $C^0$-topology on every compact subset of $N_{\R}$ as $\epsilon \rightarrow 0$. We also have the following easy fact whose proof we leave to the reader:
\begin{lem} \label{conv_aff}
If $h$ is affine, then $h \ast \mu_{\epsilon} = h$. 
\end{lem}

We now define 
\[ \tilde \phi_{\epsilon} = \phi \ast \mu_{\epsilon}. \]
These functions define smooth approximations to $\phi$. 

\begin{defi} \label{skel}
We will denote by $\Sk_{1}$ the $1$-skeleton of the subdivision of $N_{\R}$ described Lemma \ref{subdiv_P} and by 
$$ V_{\epsilon} =  \epsilon\text{-neighborhood of} \  \Sk_1. $$ 
\end{defi}

Then we have 
\begin{lem} \label{outside_epsilon}
The function $\tilde \phi_{\epsilon}$ coincides with $\phi$ outside $V_{\epsilon}$.
\end{lem}
\begin{proof} It is clear from the definition of $\mu_{\epsilon}$ and of the convolution product that $\tilde \phi_{\epsilon}(x)$ only depends on the values of $\phi$ in an $\epsilon$-ball centered in $x$. If the distance between $x$ and $\Sk_1$ is greater than $\epsilon$ then $\phi$ is affine on an $\epsilon$-ball centered at $x$. Therefore $\tilde \phi_{\epsilon}(x) = \phi(x)$ by Lemma \ref{conv_aff}. 
\end{proof}
 
For $x \in N_{\R}$ and $\epsilon > 0$ denote by $B_{\epsilon}(x)$ the open ball of radius $\epsilon$ centered at $x$. Given an edge $\check e$ of the subdivision of $P$, let
\begin{equation} \label{ue}
 U_e =\{ x \in N_{\R} \, | \, B_{\epsilon}(x) \ \text{is contained in the interior of} \ P^+_e \cup P^-_e \},  
\end{equation}
where $P^+_e$ and $P^-_e$ are as in \S \ref{trop_hyp} (see also Remark \ref{extend:pe}). 

\begin{lem} \label{phi_bend} If $\epsilon$ is such that $U_e$ is not empty, then 
\begin{equation} \label{bend}
d \tilde \phi_{\epsilon}|_{U_e}(n_{\check e}) = \inn{m^+_e}{n_{\check e}}= \inn{m^-_e}{n_{\check e}},
\end{equation}
where $m^{\pm}_e$ are as in Definition \ref{kinks} (see also Remark \ref{extend:pe}). 
In particular, for every $x \in U_e$, $(d \tilde \phi_{\epsilon})_x$ is a point in the affine line
\[ \{ m \in M_{\R}  \, | \, \inn{m-m^+_e}{n_{\check e}}=0 \}. \]
\end{lem}
\begin{proof}
Since $\inn{m^+_e-m^-_e}{n_{\check e}} = 0$, we have that $\phi - m^+_e$ is constant along the edge $\check e$. Moreover, the directional derivative of $\phi$  in the direction $n_{\check e}$ is well-defined on the subset $P^+_e \cup P^-_e$ and it is equal to $ \inn{m^+_e}{n_{\check e}}$. In particular it is well-defined and equal to $\inn{m^+_e}{n_{\check e}}$ on $B_{\epsilon}(x)$ for every $x \in U_e$. Therefore for every $x \in U_e$:
\begin{align*}
 (d \tilde \phi_{\epsilon})_x(n_{\check e}) & =  \frac{d}{dt}|_{t=0} \int_{N_{\R}} \phi(x-y+tn_{\check e} )\mu_{\epsilon}(y) dy  \\
  & =   \int_{N_{\R}}\inn{m^+_e}{n_{\check e}}\mu_{\epsilon}(y) dy = \inn{m^+_e}{n_{\check e}}, 
\end{align*}
where the second equality follows since we can assume $x-y \in B_{\epsilon}(x)$.
\end{proof}

At the end of Example \ref{bneg} we observed that the one-forms in a fibre $T_x^*B / \Lambda^*_x$ which are monodromy invariant with respect to monodromy around $e$ are of the form $[a dt + m]$, where $m \in \ker n_{\check e}$. As explained in point (b) at the end of \S \ref{lag_model}, translations in the $dt$ direction generate the $S^1$-action, so we can view $\ker n_{\check e}$ as the quotient by the $S^1$-action of the space of monodromy invariant one forms around $e$. Notice also that  
$[d(\tilde \phi_{\epsilon} - m^+_e)_x]$ naturally lives in the quotient by the $S^1$-action (i.e. $\bmod\,\, dt$). So we have:

\begin{cor} \label{bend_cor} Let $U_e$ be the open set defined in \eqref{ue}. Then for every $x \in U_e$, $[d(\tilde \phi_{\epsilon} - m^+_e)_x]$ is contained in $\ker n_{\check e}$ ,  i.e. in the quotient by the $S^1$-action of the space of monodromy invariant one-forms around $e$. 
\end{cor}

We now fix $\epsilon$ so that: 
\begin{itemize}
\item[i)] for every edge $e$ of $\Delta$,  $e \subset U_e$; 
\item[ii)] the closure $\bar V_{2\epsilon}$ of $V_{2\epsilon}$ does not contain any vertex of $\Delta$. For every vertex $p \in \Delta$, we let $W_p$ be the connected component of $N_{\R} - \bar V_{\epsilon}$ containing~$p$. 
\end{itemize}
We can also assume the following
\begin{itemize}
	\item[iii)] the disks $D_{v^-}$ of Theorem \ref{lagfib_thm}, where $f$ fails to be smooth, are contained in $N_{\R} - V_{2\epsilon}$, which is a closed set inside the union of the $W_p$'s (see also \cite{CB-M}).
\end{itemize}
From (\ref{red_smooth}) it follows that the map 
\[ \begin{split}
        \bar \sigma_{\phi}:  N_{\R}-\Delta & \rightarrow X_0(B)_{\rd} \\
           \                                  x   & \mapsto (x, [(d\tilde \phi_{\epsilon})_x])
    \end{split} \]
defines a Lagrangian section, which we call the graph of $d \tilde \phi_{\epsilon}$. Notice also that the same section is defined by the graph of $d(\tilde \phi_{\epsilon}-m^+_e)$. 
\begin{lem} \label{red_sect} With the assumptions (i)-(iii) above, the Lagrangian section $\bar \sigma_{\phi}: N_{\R}-\Delta \rightarrow X_0(B)_{\rd}$ extends to a smooth Lagrangian section $\bar \sigma_{\phi}: N_{\R} \rightarrow X(B)_{\rd}$ such that 
\[ \bar \sigma_{\phi}(N_\R) \cap S = \emptyset. \]
\end{lem}
\begin{proof} 
On every connected component of $N_{\R} - V_{\epsilon}$, we have that the map $x \mapsto (d \tilde \phi_{\epsilon})_x$ is constant and it maps $x$ to the slope of $\phi$. This follows from Lemma \ref{outside_epsilon}, since $\phi$ is affine on every connected component of $N_{\R}- V_{\epsilon}$. Moreover, if $x$ is in a bounded component of $N_{\R}-V_{\epsilon}$, then $(d \tilde \phi_{\epsilon})_x \in M$. This implies that on bounded components of $N_{\R} - V_{\epsilon}$, $\bar \sigma$ coincides with the reduced zero section $\bar \sigma_0$. Therefore, on bounded components of $N_{\R}- V_{\epsilon}$, i.e., those which contain the vertices of $\Delta$, $\bar \sigma_{\phi}$ extends as required.  We need to show that $\bar \sigma_{\phi}$ extends also over $\Delta \cap V_{\epsilon}$ and on unbounded components of $N_{\R} - V_{\epsilon}$. 
  
Consider an edge $e$ of $\Delta$. It follows from (i)-(iii) above that, inside $\R \times N_{\R}$, there is an open neighborhood $U$ of the closure of $e \cap V_{\epsilon}$, over which the Lagrangian fibration $f$ is smooth and such that $U \cap ( \{ 0 \} \times N_{\R})$ is contained in $U_e$. Let us restrict our attention to the fibration $f_{U} := f|_{f^{-1}(U)}$. The singular fibres of $f_U$ are of the type which are called generic-singular in \cite{CB-M}. This means we can describe $f_U$ as is done in Section 3 of \cite{RCB1}. Namely, let $D \subseteq \C$ be the unit disc. Then, by taking a smaller $U$ if necessary, we can assume $U \cong D \times (0,1)$ and $\Delta \cap U \cong \{ 0 \} \times (0,1)$.  Let $b = b_1 + i b_2$ be a coordinate on $D$ and $r \in (0,1)$, then $(b, r)$ are coordinates on $U$. The periods of the fibration $f_{U}$ are given by the multivalued one-forms
\begin{equation}
\begin{aligned} \label{periods}
\lambda_1 & =  - \log |b| db_1 + \Arg(b) db_2 + dH, \\
\lambda_2 & =  2 \pi db_2, \\
\lambda_3 & =  dr,
\end{aligned}
\end{equation}
where $H$ is a smooth function on $U$. Now let $\Lambda^* = \spn_{\Z}(\lambda_1, \lambda_2, \lambda_3)$. Notice that $\lambda_1$ blows up over $\Delta$. Therefore, $\Lambda^*$ has rank $2$ over $\Delta$. There is a fibre preserving symplectomorphism  between $T^*U/ \Lambda^*$ and $f^{-1}(U)- \Crit f_U$. The $S^1$-action corresponds to translation along $\lambda_2$. Therefore, in these coordinates the moment map is given by $b_2$. The space of monodromy invariant one-forms is spanned by $\lambda_2$ and $\lambda_3$. Now let us consider the function $\tilde \phi_{\epsilon}$ in these coordinates. First of all it is defined only on the slice $\{ b_2 = 0 \}$, so it depends on $b_1$ and $r$. Then Corollary \ref{bend_cor} implies that 
\[ d(\tilde \phi_{\epsilon}-m^+_e) = h(r) dr \]
for some function $h(r)$.  This section clearly extends over $\Delta$ to a section which avoids $\Crit f_U$.  Since unbounded components of $N_{\R} - V_{\epsilon}$ only contain parts of unbounded edges of $\Delta$ over which the fibration is smooth, the same argument shows that $\bar \sigma_{\phi}$ extends also here. 
\end{proof}

Finally we can prove the following

\begin{thm} \label{sections_thm}
Let $K = (k_{e})_{e \in \mathscr E}$ be a collection of integers representing a class $[\sigma_K]$ in the space of smooth sections $H^1(B, \iota_{\ast} \Lambda^*)$. Then this class has a smooth Lagrangian representative. 
\end{thm}

\begin{proof} Using a smooth approximation $\tilde \phi_{\epsilon}$ of the support function $\phi$ associated to $K$, we have constructed in Lemma \ref{red_sect} a section $\bar \sigma_{\phi}: N_{\R} \rightarrow X(B)_{\rd}$ of the reduced fibration $\bar f$. Any lift $\sigma_{\phi}: N_{\R} \rightarrow \mu^{-1}(0)$ of $\bar \sigma_{\phi}$ to the $S^1$-fibres of the $S^1$-action defines a $2$-dimensional isotropic submanifold of $X(B)$. Using the isotropic neighborhood theorem it is easy to show that this lift can be extended to a genuine section $\sigma_{\phi}: \R \times N_{\R} \rightarrow X(B)$. 

We need to show that $\sigma_{\phi}$ is in the class represented by $K$. Let $C$ be some bounded component of $N_{\R} - \Delta$ and let $e$ be an edge in the boundary of $C$, whose vertices are $p^+ \in P^+_e$ and $p^- \in P^-_e$. By construction $m^+_e$, $m^-_e$, $k_e$ and $n_e$ satisfy (\ref{bundles:cond}). Lemmas \ref{outside_epsilon} and \ref{phi_bend} imply that as $x$ moves along a curve in $U_e$ from a point in $W_{p^-}$ to a point in $W_{p^+}$, $(d\tilde \phi_{\epsilon})_x$   moves on a straight line with slope $n_e$ from the point $m^-_e$ to the point $m^+_e$. In other words, along this curve, the section $\bar \sigma_{\phi}$ winds $k_e$ times around the cycle $n_e$. 
\end{proof}
\subsection{The mirror correspondence}
We have seen that line bundles on $\check X$ and Lagrangian sections on $X$ are both characterized by an $(m-1)$-tuple of numbers $K = (k_1, \ldots , k_{m-1})$ in the $2$-dimensional case or by a collection $K= (k_e)_{e \in \mathscr E}$ in $\ker \Phi$ in the $3$-dimensional case. We conjecture the following:
\begin{con} \label{hms:dim2} The Lagrangian sections generate a suitable derived Fukaya category such that the correspondence which maps the sections $\sigma_K$ to the line bundle $\mathscr L_{-K}$ induces an embedding of this category in $\DC(\check X)$.
\end{con}
In the two-dimensional case, this conjecture has been proved by Chan and Ueda \cite{chan:ueda}, where the morphisms in the Fukaya category are defined using a wrapped Floer homology. 

Here are some examples. 

\begin{ex} In the two dimensional case (Examples \ref{fib:d2} and \S \ref{sec:d2}), let $m=2$, then $\check X$ is the total space of $\mathcal{O}_{\PP^1}(-2)$ with only one compact divisor $D$. Then the Picard group of $\check X$ is $\Z$. The zero section corresponds to structure sheaf $\mathcal{O}_{\check X}$. Line bundles of the type $\mathcal{O}_{\check X}(l D)$ correspond to sections~$\sigma_{2l}$ . 
\end{ex}

\begin{ex} \label{a2d:bundls}
Consider $P$ as in Example~\ref{triangle}, where $\check X$ is an open set in the total space of $\mathcal{O}_{\PP^2}(-3)$. Then $\operatorname{Pic}(\check X)$ is $\Z$ as well as the group of sections (cfr. Example~\ref{p2:sections}). Let $e_1$, $e_2$ and $e_3$ be the integral tangent vectors to the bounded edges of $\Gamma$ chosen as in Example~\ref{p2:sections}. A support function $\phi$ on $P$, up to the sum of a linear function, is uniquely determined by an integer $k$. Let $\mathcal L_{k}$ be the corresponding line bundle on $\check X$. The corresponding section of $X$ is determined by twisting numbers $k_{e_1} = k_{e_2} = k_{e_3} = - k$. Letting $D \cong \PP^2$ be the only compact toric divisor of $\check X$, the line bundles on $\check X$ of the type $\mathcal{O}_{\check X}(l D)$ correspond to $k = -3 l$. 
\end{ex}

\begin{ex} \label{a2d:lineb}
Let us consider the case of the smoothing of $A_{2d}$-sin\-gu\-la\-ri\-ties of Example~\ref{a2d:mirror}. Here the polytope $P$ is given by (\ref{a2d:polytope}) and the smoothing $X$ is given by equation (\ref{a2d:smoothing}). In this case a simplicial subdivision of $P$ has $d-1$ interior points, which we denote $c_j$. Hence $\check X$ has $d-1$ compact toric divisors $D_1, \ldots, D_{d-1}$. Notice that the subdivision can be chosen so that 
$D_j$ is isomorphic to the one point blowup of the Hirzebruch surface $\Sigma_j$, e.g., as in Figure \ref{a2d:trop} for the case $d=3$. Using Lemmas~\ref{sect-bundles} and \ref{sect:cond} one can show that the group of sections on $X$ and $\operatorname{Pic}(\check X)$ are isomorphic to $\Z^{3d}$. Let $C_1, \ldots, C_{d-1}$ be the bounded components of $M_{\R}- \Gamma$. Label the edges of $C_j$ by $e_{j1}, \ldots, e_{j5}$ as in Figure~\ref{sections_a2d} and orient them in anticlockwise order. Clearly we have $e_{j1} = e_{(j-1)4}$ for $j=2, \ldots d-1$. 
\begin{figure}[!ht] 
\begin{center}
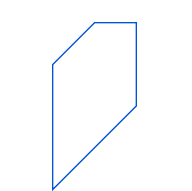\\
\vspace{1ex}\refstepcounter{figure}Figure \arabic{figure}.
\label{sections_a2d}
\end{center}
\end{figure}
Bundles on $\check X$ of the type $\mathcal{O}_{\check X}(lD_j)$ correspond to the support function $\phi$ on $P$ such that $\phi(c_j) = l$ and $\phi$ is zero on all other vertices of the subdivision. A calculation shows that these line bundles have the following intersection numbers:
\begin{gather*} 
k_{j1} = (j-2) l, \ \ k_{j2} = k_{j3} = - l, \ \ k_{j4} = -(j+1)l, \ \ k_{j5} = -2l \\
 k_{(j-1)3} = k_{(j-1)5} = l, \ \ \ k_{(j+1)2} = k_{(j+1)5} = l. 
\end{gather*}
All other intersection numbers are zero.
\end{ex}

\subsection{Translation by a section}  \label{ts}
Notice that the sections constructed in Theorem \ref{sections_thm} have the property that they all coincide with the zero section in a neighborhood of every negative vertex. In Theorem 1.2 of \cite{CBMS} it was proved that a Lagrangian section $\sigma$ with this property induces a unique fibre-preserving symplectomorphism $T_{\sigma}: X \rightarrow X$ which maps the zero section to $\sigma$. We call this symplectomorphism ``translation by a section'', since it behaves as a translation on smooth fibres, in particular we have that 
\[ T_{\sigma_K} \sigma_{K'} = \sigma_{K+K'} \]
for all pairs of twisting numbers $K$ and  $K'$.  

It has been conjectured that translation by a section should be mirror to tensoring by the corresponding line bundle. 
Theorem 1.6 of \cite{CBMS} gives some evidence of this by showing that there is only one autoequivalence of the bounded derived category of coherent sheaves which preserves skyscraper sheaves (mirror to fibres) and which sends the structure sheaf to a line bundle. 

\section{Lagrangian spheres}

In this section we will construct families of Lagrangian spheres in $X$. 
As in the case of sections, we first give a topological construction. In this case it is convenient to use the complex tropical model  (\S \ref{topological}) of the fibration $f: X \rightarrow \R \times  M_{\R}$.  Let $C$ 
be a bounded connected component of $M_\R - \Gamma$ and let $\lambda: \bar C \rightarrow M_{\R} \times (M_{\R} / M)$ be a section of $\Log$ such that $\lambda(\partial\bar C) \subseteq S$. Then we define
\begin{equation} \label{lift_sphere}
 L_{\lambda} = \alpha^{-1}(\lambda(\bar C)) 
\end{equation}
where $\alpha: X \rightarrow Y$ is the projection over the $S^1$ quotient and we think of $M_{\R} \times (M_{\R} / M)$ as  $\{ 0 \} \times M_{\R} \times (M_{\R} / M) \subset \R \times M_{\R} \times (M_{\R} / M) = Y$. Since $\bar C$ is homeomorphic to an $(n-1)$-ball and the fibres of $\alpha|_{\lambda(\bar C)}$ are circles collapsing to points at the boundary of $\bar C$, we have that $L_{\lambda}$ is homeomorphic to $S^{n}$. This construction was also suggested by Chan in \cite{chan:an}. 

\subsection{Spheres in dimension $2$} \label{spheres_d2} We look at the case where $\check X$ is an open set in the total space of $\mathcal{O}_{\PP^1}(-2)$ and we use the notation of \S\ref{sec:d2}. Let $\bar C = [0,1]$ and for some $k \in \Z$ define
\[ \lambda_k(s) =  \left( s, e^{\left(2k s + \frac{1}{2} \right) \pi i} \right) \]
Then we have that $\lambda(\partial\bar C) \subseteq S$ and $L_{\lambda_k} \cong S^2$. We simplify notation by renaming $L_{\lambda_k}$ by $L_{k}$ and we call $k$ the {\it twisting number} of the sphere. If we view $X$ as in equation (\ref{mirror}) and deform the polynomial $h_t$ so that it acquires a double root, then $X$ becomes singular with an ordinary double point. As a consequence, the singular points of the $S^1$-action come to coincide and the sphere $L_0$ vanishes, i.e., it represents a vanishing cycle. The other spheres $L_k$ are twists of the vanishing cycle. It is not difficult to show that if $T_{\sigma_\ell}$ denotes translation by the a section $\sigma_{\ell}$ (see \S \ref{ts}), then we have
\begin{equation} \label{transl_d2} 
T_{\sigma_{\ell}}(L_k)=L_{k+\ell}. 
\end{equation}
We also point out that this construction automatically gives Lagrangian spheres. 

\subsection{Spheres in dimension $3$} \label{vc:d3} We use the same notation as in \S\ref{topological}. We fix a bounded component $C$ of $M_{\R}- \Gamma$.  Let $v_C$ be the unique vertex in the decomposition of $P$ corresponding to $C$. The tangent wedges to the simplices in the decomposition of $P$ which contain $v_C$ form a complete fan $\Sigma_C$ in $N_{\R}$. An edge $e \subset \partial\bar C$ corresponds to an edge $\check e$ which emanates from $v_C$.

\begin{defi} \label{semi:integral} For every edge $\check e$ emanating from $v_C$, let $\epsilon_{\check e} \in \{1, -1 \}$ be such that $\epsilon_{\check e} \, n_{\check e}$ points outward from $v_C$. A continuous function $\vartheta: |\Sigma_C| \rightarrow \R$ is said to be a {\it semi-integral support function} if its restriction to every maximal cone of $\Sigma_C$ is linear and for every edge $\check e$ emanating from $v_C$ we have 
\begin{equation} \label{semi-int}
 \vartheta(\epsilon_{\check e} \, n_{\check e})= \frac{1}{2} \ \mod \Z. 
\end{equation}
\end{defi}

Recall that every bounded edge $e$ of $\Gamma$ bounding a connected component $C$ of $M_\R - \Gamma$ represents a $1$-dimensional toric stratum $\PP^1_{e}$ of $\check X$ which is also a subvariety of the divisor $D_C$ corresponding to $C$. 

\begin{defi} \label{twist:spheres}
A set of {\it twisting numbers} for the bounded component $C$ is a collection $\ell := (\ell_e)_{e \subset \partial\bar C}$ given by assigning to every edge $e$ bounding $C$ an integer $\ell_e$ such that

\begin{itemize}
\item[a)] $\ell_e$ has the same parity as the self-intersection number of $\PP^1_{e}$ inside $D_C$;
\item[b)] the numbers $\ell_e$ satisfy the following balancing condition
\begin{equation} \label{vc:balance}
\sum_{e \subset \partial\bar C} \epsilon_e \, \ell_e \, n_e = 0,
\end{equation}
 where $\epsilon_e=1$ if $n_e$ coincides with a fixed orientation of $\partial\bar C$, otherwise $\epsilon_e =-1$.
\end{itemize}
\end{defi}

To any semi-integral support function $\vartheta$, we can associate its set of {\it kinks} as follows. For every edge $\check e$ emanating from $v_C$, denote by $\tilde \theta^+_e$ and $\tilde \theta^-_e$ the restriction of $\vartheta$ to the tangent wedges of $P^+_e$ and $P^-_e$ respectively. Then there exists an integer $\ell_e$, the kink of $\vartheta$ at $\check e$, such that 
\begin{equation} \label{spheres_kinks}
     \tilde \theta^+_e -\tilde \theta^-_e = \frac{\ell_e}{2} n_e.
\end{equation}
We have the following:
\begin{prop} \label{kinks_to_fun} The kinks $(\ell_e)_{e \subset \partial\bar C}$ of a semi-integral support function on $\Sigma_C$ define a set of twisting numbers for $C$. Moreover any set of twisting numbers $(\ell_e)_{e \subset \partial\bar C}$ determines a semi-integral support function on $\Sigma_C$ which is unique up to adding an integral linear function. 
\end{prop}

\begin{proof} Assume for simplicity that for all edges $\check e$ emanating from $v_C$, $n_{\check e}$ points outward from $v_C$. Fix an edge $\check e$ emanating from $v_C$. Let $\check e^+$ and $\check e^-$ be the edges of $P^+_e$ and $P^-_e$ respectively, different from $\check e$, which emanate from $v_C$.  Denote by $b_e$ the self-intersection number of $\PP^1_e$  inside $D_C$. Then, from toric geometry, it follows that 
\begin{equation} \label{self:int}
   n_{\check e^-} = - n_{\check e^+} - b_e  n_{\check e}. 
\end{equation}
Moreover, since $\{n_{\check e}, n_{\check e^+} \}$ form a basis of $N$ and $n_e$ is primitive, we have
\begin{equation} \label{ne:+}
  \inn{n_e}{n_{\check e^+}} = \pm 1. 
\end{equation}
Now let $\tilde \theta^+_e$ and $\tilde \theta^-_e$ be linear functions such that (\ref{spheres_kinks}) holds for some integer $\ell_e$. Assume also that $\tilde \theta^+_e$ satisfies 
\begin{equation} \label{halfmodz}
    \inn{\tilde \theta^+_e}{n_{\check e^+}} =  \frac{1}{2} \mod \Z \ \ \ \text{and} \ \ \  \inn{\tilde \theta^+_e}{n_{\check e}} =  \frac{1}{2} \mod \Z.
\end{equation}
Then we claim that $\ell_e$ and $b_e$ have the same parity if and only if $\tilde \theta_e^-$ satisfies 
 \begin{equation} \label{halfmodz-}
    \inn{\tilde \theta^-_e}{n_{\check e^-}} =  \frac{1}{2} \mod \Z \ \ \ \text{and} \ \ \  \inn{\tilde \theta^-_e}{n_{\check e}} =  \frac{1}{2} \mod \Z.
\end{equation}
The second equality of (\ref{halfmodz-}) follows directly from the second equality of (\ref{halfmodz}) and from (\ref{spheres_kinks}). From (\ref{spheres_kinks}), (\ref{self:int}) and (\ref{ne:+}) it follows that 
\begin{align} \label{hmz}
    \inn{\tilde \theta^-_e}{n_{\check e^-}} & = \inn{\tilde \theta^+_e - \frac{\ell_e}{2} n_e}{- n_{\check e^+}- b_e n_{\check e}}  \\
   \notag & = - \inn{\tilde \theta^+_e}{n_{\check e^+}} - b_e \inn{\tilde \theta^+_e}{n_{\check e}} \pm \frac{\ell_e}{2}.
   \end{align} 
If $b_e$ and $\ell_e$ have the same parity, then this equality and equation (\ref{halfmodz}) imply (\ref{halfmodz-}). On the other hand, if (\ref{halfmodz-}) holds, then this equality and equation (\ref{halfmodz}) imply that $\ell_e$ and $b_e$ have the same parity.

In particular we have proved that the kinks of a semi-integral support function must satisfy (a) of Definition \ref{twist:spheres}. Point (b) follows from (\ref{spheres_kinks}). 

\begin{figure}[!ht] 
\begin{center}
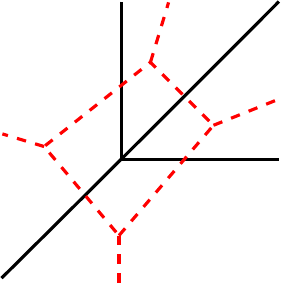
\caption{The cyclic indexing of vertices and edges of $\partial\bar C$.} \label{vc_fan}
\end{center}
\end{figure}

Given a set of twisting numbers $(\ell_e)_{e \subset C}$, we construct a semi-integral support function as follows. 
Let us fix a cyclic, anti-clockwise ordering of the maximal cones of $\Sigma_C$. Let us label the $j$-th maximal cone by $\nu_j$, with $j \in {1, \ldots , r}$ and by $\check e_j$ the edge in the common intersection of $\nu_j$ and $\nu_{j+1}$ (see Figure \ref{vc_fan}).  Let $\tilde \theta_1 \in M_{\R}$ be a linear function such that 
 \begin{equation} \label{kinks:fun}
    \inn{\tilde \theta_1}{n_{\check e_r}} =  \frac{1}{2} \mod \Z \ \ \ \text{and} \ \ \  \inn{\tilde \theta_1}{n_{\check e_1}} =  \frac{1}{2} \mod \Z.
 \end{equation}
Now define inductively
\begin{equation} \label{ell_to_psi}
 \tilde \theta_{j+1} = \frac{1}{2} \epsilon_{e_j} \ell_{e_j} n_{e_{j}} + \tilde \theta_j,
\end{equation}
where $\epsilon_{e_j} = 1$ if $\nu_{j+1}$ is the tangent wedge to $P^+_{e_j}$ and $\nu_{j}$ is the tangent wedge to $P^-_{e_j}$, otherwise $\epsilon_{e_j} = -1$. By the balancing condition (\ref{vc:balance}), $\tilde \theta_{r+1} = \tilde \theta_1$.   Define $\vartheta: |\Sigma_C| \rightarrow \R$ by
\[ \vartheta|_{ \nu_{j}} = \tilde \theta_j.\]
By the argument above, $\vartheta$ is a semi-integral support function and it is unique up to adding a linear function. 
\end{proof}

Finally, given a set of twisting numbers $(\ell_e)_{e \subset \partial\bar C}$ for $C$, we explain how to construct a sphere. Here we use the complex tropical model of the fibration (\S \ref{topological}). First let us define $\lambda: \partial\bar C \rightarrow M_{\R} \times (M_{\R} / M)$ so that $\lambda(\partial\bar C) \subset S$. Let $\check e$, $\check e^+$ and $\check e^-$ be edges emanating from $v_C$ defined as in the proof of the previous proposition and let $e$, $e^+$ and $e^-$ be the dual edges in $\partial\bar C$. 
Let $p^+$ and $p^-$ be the vertices of $\partial\bar C$ which are dual to $P^+_e$ and $P^-_e$ respectively. Clearly $p^+ = e \cap e^+$ and $p^- = e \cap e^-$. See Figure \ref{p+and-}.

\begin{figure}[!ht] 
\begin{center}
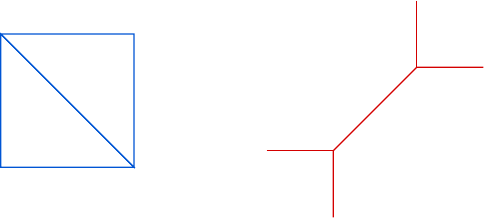\\
\vspace{1ex}\refstepcounter{figure}Figure \arabic{figure}.\label{p+and-}
\end{center}
\end{figure}

Recall equation (\ref{cylinder_cycle}) defining the cycles $\delta_e$ used to construct $S$. Consider the unique points in $M_{\R}/M$ given by
\begin{equation} \label{thetapi}
 \theta_{p^+} = \delta_{e} \cap \delta_{e^+} \ \ \text{and} \ \ \theta_{p^-} = \delta_{e} \cap \delta_{e^-}.
\end{equation}
Now define $\lambda|_{e}$ as follows  
\begin{equation} \label{vc:bound}
  \lambda|_{e} ((1-t)p^+ + tp^-) = \left( (1-t)p^+ + tp^-, -\left [\frac{1}{2} \ell_e n_e t \right] + \theta_{p^+} \right). 
\end{equation} 
Notice that $\lambda|_{e}(e) \subseteq S_e$ (cfr. (\ref{cylinder:e})). Moreover, it follows from \eqref{self:int} and the fact that $\ell_e$ and $b_e$
have the same parity that
\[ \lambda|_{e}(p^-) = (p^-, \theta_{p^-}). \]
In particular, we have that $\lambda(\partial\bar C) \subset S$.

The balancing condition (\ref{vc:balance}) ensures that $\lambda$ can be extended to a section $\lambda: \bar C \rightarrow M_{\R} \times (M_{\R} / M)$. 
Hence for every set of twisting numbers $\ell = (\ell_e)_{e \subset \partial\bar C}$ we have 
an embedded $3$-sphere constructed over $C$, which we denote $L_{\ell}$. It is not difficult to see that $L_{\ell}$ and $L_{\ell'}$ are isotopic if and only if $\ell = \ell'$. 

It is clear from the construction that the spheres $L_{\ell}$ constructed over the same component $C$ are related via a translation by a section (see \S \ref{ts}). In fact suppose $\sigma_{\kappa}$ is a section with twisting numbers $\kappa = (k_e)_{e \in \mathscr E}$, then we 
have the following formula
\begin{equation} \label{trans_sphere} 
T_{\sigma_{\kappa}} (L_{\ell}) = L_{\ell + 2 \kappa},
\end{equation}
where $L_{\ell + 2 \kappa}$ is a sphere with twisting numbers $\ell_e + 2k_e$ for all $e \subseteq \partial\bar C$. 

\begin{ex}
\label{localP2exampleredux}
Let us look at the case when $P$ is as in Example~\ref{triangle}. In this case there is just one bounded connected component $C$ whose corresponding divisor in $\check X$ is $\PP^2$. Thus for every edge $e$ bounding $C$ the self-intersection number of $\PP^1_{e}$ is $1$. We thus have a family of Lagrangian spheres $L_{\ell}$, with $\ell$ an odd integer. Figure~\ref{p2:vc} shows the $2$-chains $T_p$ at all vertices of $\Gamma$. From their shape it is clear why $\lambda|_{e}$ has to wind an odd multiple of $n_e/ 2$ around $S_e$. In fact suppose $\lambda|_{\partial\bar C}$ starts at point $1$ (black dot), then continues to wind along the cylinder over the horizontal edge. When it reaches the next vertex, it has to go to point $2$, so that it can start winding along the diagonal edge. In order to move from point $1$ to point $2$ it has to wind an odd multiple of~$n_e/ 2$.

\begin{figure}[!ht] 
\begin{center}
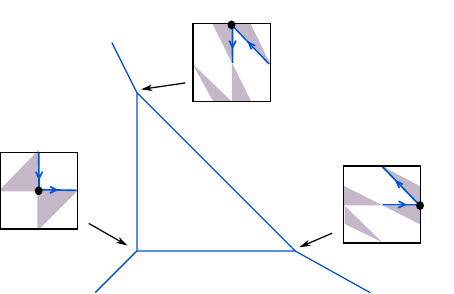 
\caption{The $2$-chains $T_p$ at the vertices of $\Gamma$ and the construction of $L_{\ell}$.} \label{p2:vc}
\end{center}
\end{figure}
\end{ex}

\subsection{Lagrangian spheres} \label{vc:lag}
For every semi-integral support function $\vartheta$ of $\Sigma_C$, with kinks $\ell = (\ell_e)_{e \subset \partial\bar C}$, we construct a Lagrangian sphere in $X(B)$, representing the isotopy class of the sphere $L_{\ell}$. In this case we need to use the Lagrangian model of the fibration in \S \ref{lag_model}. In particular here the discriminant locus $\Delta$ is locally codimension $1$ near negative vertices (see Figure~\ref{negative_thick}). In the following we use the same notations as in \S \ref{lag_model} and \S \ref{lag_sections}. Notice that in this case $C$ will be a bounded component of $N_{\R} - \Delta$ and $e$ will be an edge of $\Delta$. The main idea is to construct a Lagrangian section $\lambda: \bar C \rightarrow X(B)_{\rd}$ such that $\lambda (\partial\bar C) \subset S$. Then $L_{\lambda} = \alpha^{-1}(\lambda(\bar C))$ is a Lagrangian $3$-sphere.

We define smooth approximations to $\vartheta$ using convolution, as in \S \ref{lag_sections}. For $\epsilon > 0$, let
\[ \tilde \vartheta_{\epsilon} = \vartheta \ast \mu_{\epsilon}, \]
where $\mu_{\epsilon}$ and $\ast$ are defined in (\ref{mollifier}) and (\ref{convolution}). Since $C \subset |\Sigma_C|$, we regard $\tilde \vartheta_{\epsilon}$ as defined on $C$. We can now define a Lagrangian section as 
\begin{equation} \label{lag_sph_sec}
\begin{split}
\lambda_{\epsilon} :  \ C & \rightarrow X_0(B)_{\rd} \\
    \ b & \mapsto (b, [(d \tilde \vartheta_{\epsilon})_b]).
\end{split}
\end{equation}
The goal of this section is to prove:

\begin{thm} \label{lag_spheres}
For $\epsilon>0$ sufficiently small, the section $\lambda_{\epsilon}: C \rightarrow X_0(B)_{\rd}$ constructed above extends by continuity to a map $\lambda_{\epsilon}: \bar C \rightarrow X(B)_{\rd}$ such that 
\begin{itemize}
\item[a)] $\lambda_{\epsilon}(\partial\bar C) \subseteq S$;
\item[b)] the smooth Lagrangian sphere $L_{\lambda_{\epsilon}}= \alpha^{-1}(\lambda_{\epsilon}(\bar C))$ is in the same isotopy class of the sphere $L_{\ell}$ constructed in \S \ref{vc:d3}.
\end{itemize}
\end{thm}

The difficulty here is to analyze the behavior of $\lambda_{\epsilon}$ near the discriminant locus. Near the (negative) vertices, the main observation is that the image of $\lambda_{\epsilon}$ coincides with some canonical subset, i.e., the fixed point set of the anti\-symplectic involution constructed in \cite{CBMS}. So, even though the discriminant is fattened around these points, we do not have to worry about the behavior of $\lambda_{\epsilon}$, since it is controlled by this canonical set. Near edges of $\Delta$ one needs to work with the local model and check that $\lambda_{\epsilon}$ behaves as desired.  

Recall that $\Sk_1$ is the $1$-skeleton of the subdivision of $N_{\R}$ constructed in Lemma \ref{subdiv_P} and that $V_{\epsilon}$ is its $\epsilon$-neighborhood (see Definition \ref{skel}). We choose $\epsilon$ so that conditions (i)-(iii) listed after Corollary \ref{bend_cor} hold.   We also let
\[ \tilde W_p = (-\delta, \delta) \times W_p \]
be a thickening of $W_p$ inside $B$, for some $\delta >0$.

It is convenient here to use the fact, proved in \cite{CBMS}, that on $X(B)$ there exists a unique antisymplectic involution $\iota: X(B) \rightarrow X(B)$ such that $f \circ \iota = f$ and which leaves the zero section fixed, i.e., $\iota \circ \sigma_0 = \sigma_0$. Here, antisymplectic means that $\iota^* \omega = - \omega$.  On $X_0(B)$, $\iota$ is simply the map
\[ (b, [\eta]) \mapsto (b, - [\eta]) \]
for all $b \in B_0$ and $\eta \in T^*_bB_0$. In particular this implies that every smooth fibre $f^{-1}(b)$ contains eight fixed points, one of which is the intersection between the fibre and the zero section. Let $J$ be the fixed point locus of $\iota$, i.e., 
\[ J = \{ q \in X(B) \, | \, \iota(q) = q \}. \]
We have that $J$ is a Lagrangian submanifold of $X(B)$ such that $f|_J: J \rightarrow B$ is a degree $8$ branched covering of $B$ which branches over $\Delta$. The branching locus is $J \cap \Crit f$.  It also follows from the construction in \cite{CBMS} that $\iota$ is compatible with the $S^1$-action, in the sense that 
\[ \iota( e^{is} \cdot q) = e^{-is} \cdot \iota(q) \]
for all $e^{is} \in S^1$ and $q \in X(B)$.

 We have the following 

\begin{lem} \label{local_J} For every vertex $p \in \Delta$, the set $f^{-1}(\tilde W_p) \cap J$ has five connected components. Two of these are mapped one-to-one onto $\tilde W_p$ by $f$ and three of them are mapped two-to-one. For every edge $e$ of $\Delta$ emanating from $p$, there is a unique connected component of $f^{-1}(\tilde W_p) \cap J$, which we denote by $J_e$, such that 
\begin{itemize}
\item[a)] for every small simple loop $\gamma_e \subset \tilde W_p$ around $e$, $f^{-1}(\gamma_e) \cap J_e$ consists of two disjoint loops;
\item[b)] if $e'$ is another edge of $\Delta$ emanating from $p$ and $\gamma_{e'} \subset \tilde W_p$ is a small simple loop around $e'$, then $f^{-1}(\gamma_{e'}) \cap J_e$ is a circle which is mapped two-to-one to $\gamma_{e'}$ by $f$.
\end{itemize}
\end{lem}

\begin{proof} The fibration $f$ around the vertex $p$ is of negative type. The first two assertions are proved in \cite{CBMS} (in particular in Example 3.10 the five connected components are computed explicitly). 
Let $e, e'$ and $e''$ be the three edges of $\Delta$ emanating from $p$. The monodromy formula (\ref{bneg:mon}) implies that, given $b \in \tilde W_p$, there exists a basis of $\Lambda^*_b$ with respect to which the monodromy transformations associated to simple loops $\gamma_e$, $\gamma_{e'}$ and $\gamma_{e''}$ around the three edges are given respectively by 
\[ \left( \begin{array}{ccc}
             1  & 1 & 0 \\
             0  & 1 & 0 \\
             0  & 0 & 1 
           \end{array} \right),  \  \left( \begin{array}{ccc}
             1  & 0 & 1 \\
             0  & 1 & 0 \\
             0  & 0 & 1 
           \end{array} \right), \  \left( \begin{array}{ccc}
             1  & -1 & -1 \\
             0  & 1 & 0 \\
             0  & 0 & 1 
           \end{array} \right). \]
The eight points of $f^{-1}(b) = T^*_bB/ \Lambda^*$ which are fixed by $\iota$ are given, with respect to this basis, by
  \[ \begin{split} & (0,0,0), \, \left( \half , 0,0 \right),  \, \left(0, 0, \half \right), \, \left(\half, 0, \half \right), \\
                        & \left(0, \half ,0 \right), \, \left(\half, \half, 0\right), \, \left(0, \half, \half \right), \, \left(\half, \half, \half \right). 
     \end{split} \]
The first two points are invariant with respect to all three matrices, and therefore they belong to the two connected components of $f^{-1}(\tilde W_p) \cap J$ which are mapped one-to-one onto $\tilde W_p$. The two points in the second pair are fixed by the first matrix and are exchanged by the second and third one. This implies that they belong to the same connected component of $f^{-1}(\tilde W_p) \cap J$, which we denote by $J_{e}$. Notice that $f^{-1}(\gamma_{e}) \cap J_{e}$ must be a pair of disjoint circles.  On the other hand, $f^{-1}(\gamma_{e'}) \cap J_{e}$ is a single circle which covers $\gamma_{e'}$ twice, since the pair of points is exchanged by monodromy as we move around $\gamma_{e'}$. The same happens for $f^{-1}(\gamma_{e''}) \cap J_{e}$. Therefore $J_e$ satisfies $(a)$ and $(b)$ and clearly it is the unique connected component satisfying them. 
\end{proof}

Given a cyclic, anti-clockwise ordering of the maximal cones of $\Sigma_C$, let $p_j$ be the vertex of $\partial\bar C$ which is contained in the $j$-th cone (see also Figure~\ref{vc_fan}). Let $e'_j$ denote the edge of $\Delta$ emanating from $p_j$ which is not contained in $\partial\bar C$. Then
\begin{lem} The set $\lambda_{\epsilon}(W_{p_j} \cap C) \subset X_0(B)_{\rd}$ is contained  $\alpha( \mu^{-1}(0) \cap J_{e'_j})$, where $J_{e'_j}$ is the connected component  of $f^{-1}(\tilde W_p) \cap J$ described in\linebreak Lemma~\ref{local_J}.
\end{lem}

\begin{proof} Lemma \ref{outside_epsilon}, applied to $\tilde \vartheta_{\epsilon}$ instead of $\tilde \phi_{\epsilon}$, implies that $\tilde \vartheta_{\epsilon}$ coincides with $\vartheta$ on $W_{p_j} \cap C$, i.e., with $\tilde \theta_j$, the restriction of $\vartheta$ to the $j$-th cone.  Therefore, for all $b \in W_{p_j} \cap C$
\[ \lambda_{\epsilon}(b) = (b, [ (d \tilde \vartheta_{\epsilon})_b] )= (b, [ \tilde \theta_j]) = (b, \theta_{p_j}), \]
where $\theta_{p_j}$ is defined as in equation (\ref{thetapi}).  A point $(b, [m]) \in X_0(B)_{\rd}$ is in $\alpha( \mu^{-1}(0) \cap J)$ if and only if $m \in \half M$. 
It follows from (\ref{semi-int}) that $\tilde \theta_{j} \in \half M$, so that $(b, \theta_{p_j}) \in \alpha( \mu^{-1}(0) \cap J)$. Moreover, (\ref{bneg:mon}) and (\ref{semi-int}) imply that $\theta_{p_j}$, seen as an element in $T^*_bB / \Lambda^*_b$, is not invariant with respect to monodromy around $e_{j-1}$ and $e_{j}$. From the characterization of $J_{e'_j}$ given in Lemma \ref{local_J} this implies that $(b, \theta_{p_j}) \in  \alpha( \mu^{-1}(0) \cap J_{e'_j})$ and therefore the lemma. 
\end{proof}

\begin{cor} We have that $\lambda_{\epsilon}|_{W_{p_j} \cap C}: W_{p_j} \cap C \rightarrow X_0(B)_{\rd}$ can be ex\-ten\-ded by continuity to $W_{p_j} \cap \bar C$ so that $\lambda_{\epsilon}( W_{p_j} \cap \partial\bar C) \subseteq S$. 
\end{cor}

\begin{proof}
The branching locus of $J_{e'_j}$ is $J_{e'_j} \cap \Crit f$ and is mapped by $f$ homeomorphically onto $W_{p_j} \cap \partial\bar C$. This can be checked explicitly in the local models  (e.g., in Example 3.10 of \cite{CBMS}). Therefore, since $\lambda_{\epsilon}({W_{p_j} \cap C})$ is contained in $\alpha( \mu^{-1}(0) \cap J_{e'_j})$, as $b \in W_{p_j} \cap C$ approaches a point in $W_{p_j} \cap \partial\bar C$, $\lambda_{\epsilon}(b)$ approaches the image by $\alpha$ of a branching point of  $J_{e'_j}$, i.e., a point on $S$.  This proves the corollary. 
\end{proof}

\begin{proof}[Proof of Theorem \ref{lag_spheres}] To extend $\lambda_{\epsilon}$ it remains to show that $\lambda_{\epsilon}$ extends also to points of $\partial\bar C \cap \bar V_{\epsilon}$. Given an edge $e_j$ of $\partial\bar C$, let us consider an open neighborhood $U$ of $e_j \cap \bar V_{\epsilon}$. We can assume that $U \cap (\{ 0 \} \times N_{\R}) \subseteq  U_{e}$ and that $f_U := f|_{f^{-1}(U)}$ is smooth. Just as in Lemma \ref{phi_bend} we have that 
\begin{equation} \label{bend2}
d \tilde \vartheta_{\epsilon}|_{U_{e_j}}(n_{\check e_{j}}) = \inn{\tilde \theta_j}{n_{\check e_j}}.
\end{equation}
We proceed as in Lemma \ref{red_sect}. We let $U \cong D \times (0,1)$ and $\Delta \cap U \cong \{ 0 \} \times (0,1)$ and define $\Lambda^* \subseteq T^*U$ as the $\Z$-linear span of the periods (\ref{periods}). Then $T^*U/ \Lambda^*$ and $f^{-1}(U)- \Crit f_U$ are fibrewise symplectomorphic.  Assume also that the coordinates have been chosen so that  
\[ C \cap U = \{ b_2 = 0, b_1 > 0 \}. \]
In Sections 3.1 and 3.2 of \cite{CBMS} we gave an explicit description of an antisymplectic involution over a neighborhood of an edge. We review it here. Without loss of generality we can assume that in these coordinates the zero section $\sigma_0$ is given by $\sigma_0(u) = (u, \frac{1}{2}[dH])$, where $H$ is the smooth function appearing in equation \eqref{periods} for the periods.
Then the antisymplectic involution that fixes $\sigma_0$ is
\[ \iota: (u, [\eta]) \mapsto (u, [dH- \eta]) \]
for all $u \in U$ and $\eta \in T^*_uU$. The fixed points of $\iota$ which are not monodromy invariant around $\Delta$ are given by the graphs of the following $1$-forms:
\[ \begin{split}
    \ &\half \left(dH + \lambda_1\right), \ \half \left( dH + \lambda_1 + \lambda_2 \right) \\
    \ &\half \left(dH + \lambda_1\ + \lambda_3 \right), \ \half \left( dH + \lambda_1 + \lambda_2  + \lambda_3 \right)
   \end{split} \]  
The first two forms are interchanged by monodromy around $\Delta$, similarly the last two. Therefore the first and second pair define two connected components which map $2$ to $1$ over $U$.  Since the $S^1$-action corresponds to translations by multiples of $\lambda_2$, the first two forms (resp. the last two) have the same image in the quotient by the $S^1$-action. 
Therefore, in these coordinates, we have that $\tilde \theta_j$ is either equal to the first form or to the third one, modulo $\lambda_2$. Let us assume, without loss of generality, that
\[ \tilde \theta_j = \half \left(dH + \lambda_1\right) \mod \lambda_2. \]
Then equation (\ref{bend2}) and Corollary \ref{bend_cor} imply that over $C \cap U$ we have
\[ d \tilde \vartheta_{\epsilon} = \half \left(dH + \lambda_1\right) + h(r) dr \]
for some smooth function $h$. The above expression is defined for points $u=(b_1, r) \in U$, with $b_1 > 0$ (since it is defined on $C \cap U$). When $b_1 = 0$, $\lambda_1$ blows up. This suggests that as $b_1 \rightarrow 0$ the points $(u, [(d \tilde \vartheta_{\epsilon})_u])$ might converge to points in $\Crit f_{U}$. This is precisely what happens:
\begin{lem} Viewing the points $(u, [(d \tilde \vartheta_{\epsilon})_u])$, for $u = (b_1, r) \in C \cap U$, as points in $f^{-1}(U) - \Crit f_{U}$ we have that 
\[ \lim_{b_1\rightarrow 0} (u, [(d \tilde \vartheta_{\epsilon})_u]) \in \Crit f_{U} \]
\end{lem}
\begin{proof}
To prove this we need to describe how $T^*U/ \Lambda^*$ is compactified by gluing in $\Crit f_U$. This was done in \cite{RCB1}. We follow the exposition given in Sections 3.1 and 3.2 of \cite{CBMS}. Let 
\[ Y = \{ (z_1, z_2, w) \in \C^2 \times \C^* \, | \, |z_1z_2| < 1 \ \ \text{and} \ \ \log |w| \in (0,1) \} \]
with symplectic form induced from the standard one on $\C^2 \times \C^*$. Then the map 
\[ 
\begin{split}
q:  Y &\rightarrow U \\
q: (z_1, z_2, w) &\mapsto ( z_1 \bar z_2, \log |w| )
\end{split} 
\]
is a Lagrangian fibration whose general fibres are isomorphic to $T^2 \times \R$. We have that 
\[ \Crit q = \{ z_1 = z_2 = 0 \} \cong S^1 \times (0,1) \]
and $q( \Crit q) = \Delta = \{ 0 \} \times (0,1)$. 
Define the following two sections of $q$:
\[ 
   \begin{split}
     \Sigma_1(b,r) &= (1, \bar b, e^r) \\
     \Sigma_2(b,r) &= (b, 1, e^r)
   \end{split}
\]
where $b= b_1+ib_2 \in D$ and $r \in (0,1)$. We also have a $\C^* \times S^1$-action on $Y$ given by
\[ (\tau, e^{is}) \cdot (z_1, z_2, w) = ( \tau z_1, \bar \tau^{-1} z_2, e^{is} w). \]
We use the sections and the action to define the following maps
\[ 
\begin{split}
   \xi_j : (\C^* \times S^1) \times U & \rightarrow Y \\
            ((\tau, e^{is}), u) &  \mapsto (\tau, e^{is}) \cdot \Sigma_j(u) 
\end{split}
\]
Now consider the following open subsets of $(\C^* \times S^1) \times U$
\[ 
\begin{split}
  V_1 & = \{ ((\tau, e^{is}), (b,r) ) \in  (\C^* \times S^1) \times U \, | \, |b| < |\tau| < 1 \} \\
  V_2 & = \{ ((\tau, e^{is}), (b,r) ) \in  (\C^* \times S^1) \times U \, | \, 1 < |\tau| < |b|^{-1} \}
\end{split}
\]
and define 
\[ Z_j = \xi_j(V_j). \]
These are open subsets of $Y$ and 
\[ Z = Z_1 \cup Z_2 \cup \Crit q \]
 is a neighborhood of $\Crit q$. 

We define a similar structure on $T^*U/\Lambda^*$. Let $L_1$ and $L_2$ be the Lagrangian sections which are defined  as the graphs  respectively of the one-form $dH$ and of the zero $1$-form. Define the following $\C^* \times S^1$-action on $T^*U/\Lambda^*$:
\[ (\tau, e^{is}) \cdot (u, [ \eta]) = (u, [ \eta - \log |\tau| db_1 + \Arg \tau \, db_2 + s dr]). \]
Then the sections and the action are used to define maps
\[ 
\begin{split}
   \zeta_j : (\C^* \times S^1) \times U & \rightarrow T^*U/\Lambda^* \\
            ((\tau, e^{is}), u) &  \mapsto (\tau, e^{is}) \cdot L_j(u).
\end{split}
\]
We then define the open subsets
\[ Z'_j = \zeta_j (V_j). \]
It is not difficult to show that the map $g: Z'_1 \cup Z'_2 \rightarrow Z_1 \cup Z_2$ given by 
\[ g|_{Z'_j}  = \xi_j \circ \zeta_j^{-1} \]
defines a fibre-preserving symplectomorphism. It is proved in \cite{RCB1} that $f^{-1}(U)$ is symplectomorphic to the symplectic manifold obtained by gluing $T^*U / \Lambda^*$ to $Z$ via $g$. Under this identification we have that $\Crit f_U = \Crit q$.

We now want to compute the points  $g(u, [ (d \tilde \vartheta_{\epsilon})_u])$ for all $u= (b_1, r) \in C \cap U$. Notice that
\[ \begin{split}
        (u,  [ (d \tilde \vartheta_{\epsilon})_u]) & = \left( u, \left[ \half ( \lambda_1 + dH ) + h(r) dr \right] \right)  \\
                                                   \ & =\left(u,  \left[ - \half \log|b_1| db_1 + dH + h(r) dr \right] \right)  \\
                                                   \ & = \zeta_1((\sqrt{b_1}, e^{i h(r)}), (b_1, r)). 
   \end{split}
\]
Therefore 
\[ \begin{split}
              g(u, [ (d \tilde \vartheta_{\epsilon})_u]) & = \xi_1((\sqrt{b_1}, e^{i h(r)}), (b_1, r) ) \\
                                                               \   & = (\sqrt{b_1}, e^{i h(r)}) \cdot \Sigma_1(b_1, r)  \\
                                                               \   & = ( \sqrt{b_1}, \sqrt{b_1}, e^{r+i h(r)}) 
   \end{split} 
\]
It is now clear that 
\[
\lim_{b_1 \rightarrow 0} g(u, [ (d \tilde \vartheta_{\epsilon})_u]) \in \Crit q 
\]
and this proves the lemma. 
\end{proof}
This lemma completes the proof of point $(a)$ of Theorem \ref{lag_spheres}. 

To prove part $(b)$, we observe that equation (\ref{bend2}) implies that for all $u \in U_{e_j} \cap C$, $(d \tilde \vartheta_{\epsilon})_u$ is a point on the affine line
\[ \{ m \in M_{\R} \, | \, \inn{m-\tilde \theta_j}{n_{\check e_j}} = 0 \}. \]
Moreover, if $u \in W_{p_j} \cap C$ then $(d \tilde \vartheta_{\epsilon})_u = \tilde \theta_j$ and if $u \in W_{p_{j+1}} \cap C$ then\linebreak $(d \tilde \vartheta_{\epsilon})_u = \tilde \theta_{j+1}$. It follows that, as $u$ moves from $W_{p_j} \cap C$ to $W_{p_{j+1}} \cap C$ inside $U_{e_j}$, $(d \tilde \vartheta_{\epsilon})_u $ moves along a line of slope $n_{e_j}$ from $\tilde \theta_j$ to $\tilde \theta_{j+1}$. Then (\ref{ell_to_psi}) implies that it  winds $\half \ell_{e_j}$ times around the cycle $n_{e_j}$. This proves point $(b)$ of the theorem. 
\end{proof} 

\begin{ex} \label{spheres_quintic} The method for constructing Lagrangian spheres described in this section is quite general and also applies to examples of compact Calabi-Yau's with Lagrangian fibrations. Consider for example the quintic hypersurface in $\PP^4$. In Section 19.3 of \cite{GHJ}, the first author described an integral affine manifold with singularities $B$ whose associated manifold $X(B)$ is homeomorphic to the quintic. Applying \cite{CB-M}, we also have a symplectic form on $X(B)$ and a Lagrangian fibration $f: X(B) \rightarrow B$. Here $B$ is the boundary of a (suitably rescaled) standard $4$-simplex. The discriminant locus intersected with a $2$-face of $B$ is represented by the honeycomb pattern in Figure \ref{quintic}.
\begin{figure}[!ht] 
\begin{center}
\includegraphics{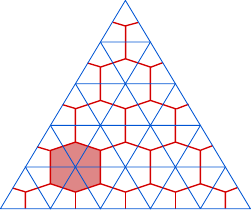}
\caption{Lagrangian spheres in the quintic are constructed over the shaded hexagon.} \label{quintic}
\end{center}
\end{figure}
The (blue) triangles come from a polyhedral subdivision giving $B$ the structure of a simple affine manifold with singularities. The fibration restricted to a neighborhood of the interior of a 2-face has the properties (a)-(c) listed at the end of \S \ref{lag_model}. Here $N_{\R}$ is the affine space spanned by the $2$-face. The $S^1$ action is given by translations in the direction of the $1$-dimensional subspace of $T^*B_0$ which is monodromy invariant around the discriminant locus inside the two face. We can construct Lagrangian spheres in $X(B)$ which map over the shaded hexagon in Figure \ref{quintic} as follows. The polyhedral subdivision (restricted to the $2$-face) defines a $2$-dimensional fan at the vertex contained in the hexagon. We may thus consider a semi-integral support function $\vartheta$ with respect to this fan. The results of this section then show that the graph of the differential of a smoothing of $\vartheta$ extends and lifts to a Lagrangian sphere in $X(B)$. 
\end{ex}

\subsection{Spheres over edges} \label{vc:edges}
There is another type of construction of vanishing cycles. Let $$P = \conv \{ (0,0), (1,0), (0,1), (1,1) \}$$ and subdivide it by adding the diagonal from $(1,0)$ to $(0,1)$. If we let $\nu$ be the unique piecewise affine function such that $\nu(1,1) = 1$ and zero at all other vertices, then $\Gamma$ looks like the
tropical curves shown in Figure~\ref{vc:conifold}. 

\begin{figure}[!ht] 
\begin{center}
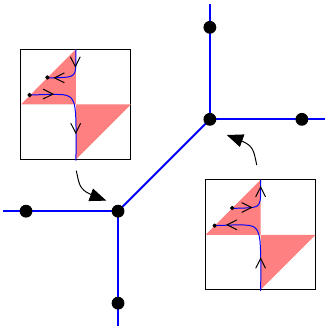
\caption{$\Gamma$ and the curves $\gamma_1$ and $\gamma_3$ inside $\Log^{-1}(p)$ and $\Log^{-1}(p')$.} \label{vc:conifold}
\end{center}
\end{figure}

Here $\check X$ is a small resolution of an ordinary double point (with equation $xy-wz = 0$), while $X$ is a smoothing of this singularity (hence given by equation $xy-wz = t$). The transformation from $\check X$ to $X$ is also known as a conifold transition and it is of independent interest (see \cite{STY} and \cite{CB-M-Con} for more information on this vanishing cycle). For every integer $k$, we construct a curve $\lambda_k: S^1 \rightarrow S$ which can be extended to an embedding $\lambda_k: D^2 \rightarrow M_{\R} \times M_{\R} / M$, where $D^2$ is the unit disc, such that $\lambda_k(D^2) \cap S = \lambda_k(S^1)$. Then, as before, we define 
\[ L_k = \alpha^{-1}(\lambda_k(D^2)). \]
The curve $\lambda_k$ is defined as follows. Let $e$ be the edge joining the two vertices $p$ and $p'$ of $\Gamma$. Let $\theta$ and $\psi$ be two distinct (but close) points on $T^2$ lying on the $1$-cycle $\delta_e$. Let $\gamma_1$ be an oriented curve inside $\Log^{-1}(p)$ going from $(p, \psi)$ to $(p, \theta)$ drawn and oriented as in Figure~\ref{vc:conifold}. In particular $\gamma_1$ is contained in $T_p$. Let $\gamma_2: [0,1] \rightarrow S_e$ be defined as
\[ \gamma_2(s) = ((1-s)p + sp', [kn_es] + \theta). \]
Clearly $\Log(\gamma_2([0,1])) = e$. Define $\gamma_3$ to be an oriented curve inside $\Log^{-1}(p')$ going from $(p', \theta)$ to $(p', \psi)$ drawn and oriented as in Figure~\ref{vc:conifold}. Finally let $\gamma_4: [0,1] \rightarrow S_e$ be defined by
\[ \gamma_4(s) = ((1-s)p' + sp, [-kn_e s] + \psi). \]
\begin{figure}[!ht] 
\begin{center}
\includegraphics{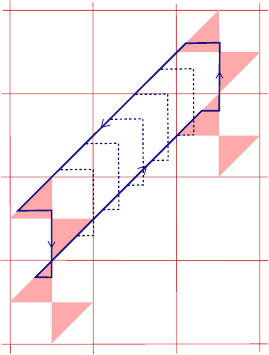}\\
\vspace{1ex}\refstepcounter{figure}Figure \arabic{figure}.\label{conif_cycle}
\end{center}
\end{figure}
We define $\lambda_k: S^1 \rightarrow S$ as the concatenation of the curves $\gamma_1$, $\gamma_2$, $\gamma_3$ and $\gamma_4$. Clearly  $\lambda_k$ can be extended to a map $\lambda_k: D^2 \rightarrow  M_{\R} \times M_{\R} / M$, this follows from the fact that the concatenation of  $\gamma_1, \ldots, \gamma_4$ is trivial in homology. Figure \ref{conif_cycle} depicts $\lambda_2: S^1 \rightarrow S$ projected to the torus and lifted to its universal cover. It also suggests a way to extend $\lambda_k$ to an embedding of $D^2$: the dashed lines are the fibres of $\Log \circ \lambda_k$.  A Lagrangian version of a sphere of this type will be given elsewhere. We will use this sphere in Section \ref{mirroran}.

\section{Compactly supported sections} \label{compact_support}
In this section we will determine the isotopy classes of pairs of sections $\sigma$ and $\sigma'$ of $f: X \rightarrow \R^n$ which coincide outside some compact set. In this case we say that $\sigma$ is compactly supported with respect to $\sigma'$. When $\sigma'$ is the zero section we just say that $\sigma$ is compactly supported. The closure of the set where the two sections differ is called the support of $\sigma$ (relative to $\sigma'$). 

\begin{defi} Suppose that $\sigma$ coincides with $\sigma'$ outside of a compact set $K \subseteq \R^n$ homeomorphic to an $n$-ball. Denote by $K^+$ a copy of $K$ with an orientation induced from a fixed orientation of $\R^n$ and by $K^-$ a copy of $K$ with the opposite orientation. Then we can glue $K^+$ and $K^-$ along their boundary to form an oriented $n$-sphere. Define a map from this sphere to $X$ by defining it to be equal to $\sigma$ on $K^+$ and equal to $\sigma'$ on $K^-$. Denote the image of this map by $\sigma \sigma'$. It defines a class in $H_n(X, \Z)$ which we denote by $[\sigma \sigma']$.
\end{defi}

We will determine when $[\sigma \sigma']$ coincides with the class of one of the spheres $L_{\ell}$ constructed in the previous section.

\subsection{The two-dimensional case} \label{cs:d2} Assume we are as in \S \ref{sec:d2}, with only two singular fibres corresponding to $S =\{ (0, i), (1, i) \}$.  Figure~\ref{compact_sup_d2} depicts two reduced sections (the continuous (blue) line and the dotted (red) line) which can be extended and then lifted to a pair of sections of $f: X \rightarrow \R^2$ which coincide outside a compact set. We prove this in the following. 

\begin{figure}[!ht] 
\begin{center}
\includegraphics{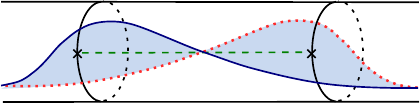}\\
\vspace{1ex}\refstepcounter{figure}Figure \arabic{figure}. \label{compact_sup_d2}
\end{center}
\end{figure}

\begin{prop} \label{cs2:prop} In dimension $2$, let $f: X \rightarrow \R^2$ be as in \S \ref{topological} with only two singular fibres. Then, for every integer $k$, the sections $\sigma_{k+2}$ and $\sigma_k$ defined in \S  \ref{sec:d2} are isotopic to a pair of sections which coincide outside some compact set. 
\end{prop} 

\begin{proof} We consider the case $k=-1$, the other cases are obtained after translation by a section. It can be easily seen that the reduced sections $\bar \sigma_{-1}$ and $\bar \sigma_1$ defined in \S  \ref{sec:d2} are isotopic (in $(\R \times S^1)- S$) to a pair of sections which can be depicted as in Figure~\ref{compact_sup_d2}, the continuous (blue) line representing $\bar \sigma_{-1}$ and the dotted (red) line representing $\bar \sigma_1$.  The two sections coincide outside a compact set containing $\Gamma = \{ 0, 1 \}$. As in \S \ref{topological}, let $Y = \R \times \R \times S^1$, which contains $\R \times S^1$ as the slice $\{ 0 \} \times \R \times S^1$. We must show that the sections can be extended and then lifted to a pair of sections of $f$ which coincide outside a compact set. The two sections can be extended to sections of $Y \rightarrow \R^2$ in such a way that they coincide outside a compact set $K$ homeomorphic to a $2$-disk. This can be done by unwinding both sections as we move away from the slice $\{ 0 \} \times \R \times S^1$ so that they both become flat (coinciding with the zero section). Denote by $\sigma_{-1}$ and $\sigma_1$ the two sections after the extension. We now show that there are lifts of both sections to the circle fibration $\alpha: X \rightarrow Y$ which coincide outside of $K$. 
As explained at the beginning of this section, we can construct a map $\xi_{\sigma_{-1}{\sigma_1}}: S^2 \to Y$, where $S^2$ is formed by gluing two copies of $K$ along the boundary and $\xi_{\sigma_{-1}{\sigma_1}} = \sigma_{-1}$ on one copy of $K$ and $\xi_{\sigma_{-1}{\sigma_1}} = \sigma_{1}$ on the other copy. Pull back the $S^1$-fibration $\alpha: X \rightarrow Y$ to $S^2$ via $\xi_{\sigma_{-1}{\sigma_1}}$. We claim that the pull back fibration is a trivial $S^1$ bundle over $S^2$. Recall that $\alpha$ satisfies the following property.  There exists an orientation on $Y$ such that the first Chern class of $\alpha$, evaluated on the boundary of a small ball centered at any one of the two points of $S$ (and not containing the other one), is equal to $1$. Notice that the complement of $\sigma_{-1} \cup \sigma_1$ inside $Y$ consists of three connected components. One of them is unbounded and the other two are two open balls, each one containing one of the points of $S$.  The intersections of these latter components with $\{ 0 \} \times \R \times S^1$ are the shaded regions in Figure~\ref{compact_sup_d2}. The union of the boundaries of these two regions coincides with the image $\xi_{\sigma_{-1}{\sigma_1}}(S^2)$. One region induces an orientation agreeing with the orientation on $S^2$, but the other one induces the opposite orientation. This implies that the first Chern class of $\alpha$ evaluated on $\xi_{\sigma_{-1}{\sigma_1}}(S^2)$ is zero and therefore the pullback bundle on $S^2$ is trivial. This shows that we can lift $\sigma_{-1}|_{K}$ and $\sigma_{1}|_{K}$ to the fibration $\alpha$ in such a way that the lifts coincide on the boundary of $K$. Then they can be extended outside of $K$ so that they coincide. \end{proof}

Let $L_k$ be the sphere constructed in \S \ref{spheres_d2} with twisting number $k$.

\begin{thm} \label{dim2_diff} In $H_2(X, \Z)$, up to some choice of orientation of $L_{k+1}$, we have
\[ [\sigma_{k+2} \sigma_k  ] = [L_{k+1}] \]
\end{thm}
\begin{proof} We do the case $k=-1$, the other cases follow after translation by a section. Let us consider $\sigma_{-1}$ and $\sigma_{1}$ as sections of $Y \rightarrow \R^2$ (i.e., before they have been lifted to sections of $X$). Then let $Z$ be the closure of the union of the two bounded components of $Y - (\sigma_{-1} \cup \sigma_{1})$. The intersection of $Z$ with $\{ 0 \} \times \R \times S^1$ is the shaded area in Figure \ref{compact_sup_d2}. 
We can write $Z = Z_1 \cup Z_2$, where $Z_1$ and $Z_2$ are the closures of the two bounded components. We can assume that $Z_1 \cap Z_2$ is homeomorphic to the interval $[0,1] \subset \R$ and that $Z_1$ and $Z_2$ are both homeomorphic to a $3$-ball. Now let $W = \alpha^{-1}(Z)$ and $W_j = \alpha^{-1}(Z_j)$. We have that $W_j$ is homeomorphic to a $4$-ball. We can view $[\sigma_{1} \sigma_{-1}]$ as an element of $H_2(W, \Z)$.  Notice that also $[L_0]$ can be viewed as an element of $H_2(W, \Z)$. In fact let $I$ be the segment joining the two points of $S$ depicted by the (green) dashed line in Figure \ref{compact_sup_d2}. Notice that $I \subseteq Z$. Then $L_0 = \alpha^{-1}(I) \subset W$. 

The Mayer-Vietoris sequence for the homology of $W$ with decomposition $W = W_1 \cup W_2$ gives the isomorphism
\[ H_2(W; \Z) \to H_1 (W_1 \cap W_2; \Z). \]
Moreover $W_1 \cap W_2 \cong [0,1] \times S^1$ so that $H_1 (W_1 \cap W_2; \Z) \cong \Z$. To prove the statement of the theorem it is enough to show that $[\sigma_{1} \sigma_{-1}]$ and $[L_0]$ are both mapped to a generator of $H_1(W_1 \cap W_2; \Z)$. This is obvious for $[L_0]$, since $L_0 \cap (W_1 \cap W_2)$ is a fibre of $\alpha$ whose class generates $H_1(W_1 \cap W_2; \Z)$. We now show that this is true also for $[\sigma_{1} \sigma_{-1}]$. 

Let $B_4$ and $B_3$ be the balls of radius $1$ centered at the origin respectively in $\C^2$ and in $\R \times \C$. Let $B_4^*$ and $B_3^*$ be the two balls minus their origin. Then the map 
\begin{equation} \label{local:alpha} 
   \begin{split}
     \alpha': B_4  & \to  B_3 \\
     (x,y) & \mapsto ( |x|^2 - |y|^2, 2 xy)   
   \end{split} 
\end{equation} 
is surjective and $\alpha'_{|B_4^*}: B_4^* \to B_3^*$ is a principal $S^1$ bundle with first Chern class $1 \in \Z = H^2(B_3^*, \Z)$. The fibration $\alpha': B_4 \rightarrow B_3$ is fibrewise isomorphic to $\alpha_{|W_j}: W_j \rightarrow Z_j$.  We have $S^3 = \partial B_4$ and $S^2 = \partial B_3$.  Given a point $\ast \in S^2$, let $S^1 = \alpha'^{-1}(\ast)$. Then we have isomorphisms 
\begin{equation*} \begin{CD}
 H_2 (S^3, S^1 ; \Z) @>\partial>> H_1(S^1; \Z)  \cong \Z \\
@V \alpha'_*VV \\
H_2(S^2, \ast; \Z)
\end{CD}
\end{equation*}
where $\partial$ is the boundary map. The map of pairs $\alpha': (S^3, S^1) \rightarrow (S^2, \ast)$ is clearly homotopy equivalent to the map $\alpha: (\partial W_j, W_1 \cap W_2) \rightarrow (\partial Z_j, Z_1 \cap Z_2)$. So we also have the diagram of isomorphisms:
\begin{equation} \label{hopf} \begin{CD}
 H_2 (\partial W_j, W_1 \cap W_2; \Z) @>\partial>> H_1(W_1 \cap W_2; \Z)  \cong \Z \\
@V \alpha_*VV \\
H_2(\partial Z_j, Z_1 \cap Z_2; \Z).
\end{CD}
\end{equation}
Notice that $[\sigma_{1} \sigma_{-1}]$ can be viewed as an element in $H_2(\partial W; \Z)$ and $\sigma_{1}\sigma_{-1} \cap \partial W_j$ gives an element of $H_2 (\partial W_j, W_1 \cap W_2; \Z)$. Since $\sigma_{1}$ and $\sigma_{-1}$ are sections, $\alpha_{\ast}[\sigma_{1}\sigma_{-1} \cap \partial W_j]$ is a generator of $H_2(\partial Z_j, Z_1 \cap Z_2; \Z)$. This implies that also $\partial [\sigma_{1}\sigma_{-1} \cap \partial W_j] = [\sigma_{1}\sigma_{-1} \cap (W_1 \cap W_2)]$ is a generator of  $H_1(W_1 \cap W_2; \Z)$. This concludes the proof.
\end{proof}

\subsection{The three-dimensional case} \label{cs:d3}
Let $n=3$. We consider the complex tropical fibration $f: X \rightarrow \R \times M_{\R}$ and the construction of sections given in \S \ref{top_sec}. Let $C$ be a bounded connected component of $M_{\R} - \Gamma$. Recall that $C$ corresponds to a toric divisor $D_C$ in $\check X$. Moreover, every edge $e$ in $\partial\bar C$ corresponds to a toric boundary divisor of $D_C$ which we denote $\PP^1_e$.  
\begin{defi} \label{defi:Kc} Denote by $\phi_{K_C}: P \rightarrow \R$ a support function for the line bundle $\mathcal O_{\check X}(D_C)$ on $\check X$. We fix a choice of $\phi_{K_C}$ such that it satisfies $\phi_{K_C}(v_C) = 0$ and $\phi_{K_C}(v) = -1$ for all other vertices $v$ in the subdivision of $P$.  The corresponding kinks $K_{C} = (k_e)_{e \in \mathscr E}$ of $\phi_{K_C}$ satisfy the following. If $e$ is an edge of $\partial\bar C$, then $k_e = -b_e - 2$, where $b_e$ is the self-intersection number of $\PP^1_{e}$ inside $D_C \subset \check X$. If $e$ is an edge which emanates from a vertex of $C$, but is not in $\partial\bar C$, then $k_e=1$. In all other cases $k_e = 0$. 
\end{defi}
The numbers $K_C$ also identify an isotopy class of sections of $f$. 

Now let $R \subset M_{\R}$ be a convex compact set whose interior contains $C$ and which does not intersect other edges of $\Gamma$ except the ones emanating from the vertices of $C$. Let 
\[ \tilde R := [-\epsilon, \epsilon] \times R, \]
for some $\epsilon \in (0, 1)$. Let $\sigma_0$ be the zero section.  

\begin{prop}  \label{com_sup:dim3} An isotopy class $K \in \ker \Phi$ of sections of $f$ has a representative $\sigma$ whose support is contained in $\tilde R$ if and only if $K = a K_C$, for some $a \in \Z$.
\end{prop}

\begin{proof} Suppose that a section $\sigma$ in the isotopy class $K = (k_e)_{e \in \mathscr E}$ has support contained in $\tilde R$. Let $\bar \sigma: M_{\R} \rightarrow M_{\R} \times M_{\R} / M$ be the reduced section. Then $\bar \sigma$ coincides with $\bar \sigma_0$ outside $R$. If $e$ is an edge which does not emanate from a vertex of $C$, then $k_e = 0$, since $R$ does not intersect $e$.

Now let $e$ be an edge of $\partial\bar C$. Let $p^+$ and $p^-$ be the vertices of $e$ corresponding respectively to the simplex $P^+_e$ and $P^-_e$. Let $e^+$ and $e^-$ be the edges $\partial\bar C$, adjacent to $e$, which emanate respectively from $p^+$ and $p^-$. Let $d^+$ and $d^-$ be the remaining edges emanating respectively from $p^+$ and $p^-$ (see Figure \ref{compact_sup}). 
By choosing appropriate integral affine coordinates, we may assume that 
\begin{equation} \label{ne:ne+} \begin{split}
     n_{e^+} & =  (0,-1), \\
     n_e & =  (1, 0), \\
     n_{e^-}  & =  (-b_e, 1). 
     \end{split}
\end{equation} 
The balancing condition implies
\[ n_{d^+} = n_{e^+}  -  n_e = (-1,-1), \qquad n_{d^-} =n_e - n_{e^-} = (1+b_e, -1). \] 
Let $C'$ be the other component of $M_{\R} - \Gamma$ adjacent to $e$. In the case $C'$ is a bounded component then condition (\ref{sec:comp}) holds and becomes
\begin{equation} \label{kc_cond} 
k_{d^+} (n_{e}  -  n_{e^+}) +  k_e n_e + k_{d^-} (n_e - n_{e^-}) = 0. 
\end{equation}
This imposes the following necessary conditions:
\begin{equation} \label{kc_cond2}
  \begin{split}
    k_{d^+} & =  k_{d^-} \\
   k_e & =  - k_{d^+} (b_e+2).
   \end{split}
\end{equation}
If we prove that these conditions hold also when $C'$ is not bounded, then the first equation of (\ref{kc_cond2}) implies that all numbers attached to edges emanating from a vertex of $C$ but not contained in $\partial\bar C$ coincide with a fixed number $a \in \Z$, which implies that $K = a K_C$.  Thus, assuming $C'$ is not bounded, we interpret $k_{d^+}$ and $k_{d^-}$ as the number of times $\bar \sigma$ winds along $d^+$ and $d^-$. If we choose two points $q^+$ and $q^-$, respectively on  $d^+$ and $d^-$, and sufficiently far from $C$, then $\bar \sigma$ is constant along the segment joining $q^+$ and $q^-$, since it coincides with the zero section. Now consider the convex hull  of $p^+$, $p^-$, $q^+$ and $q^-$ (i.e. the region labeled by $Q_e^-$ in Figure \ref{compact_sup}), then the fact that its boundary must be mapped by $\bar \sigma$ to a homologically trivial loop implies that (\ref{kc_cond}) must hold and hence also \eqref{kc_cond2}.

On the other hand, let us show that there exists a section representing the class $a K_C$ whose support is contained in $\tilde R$. For later use, we will construct a piecewise linear section. First let us construct a map $\tilde \sigma: M_{\R} \rightarrow M_{\R}$ with support inside $R$ such that the reduced section $\bar \sigma: M_{\R} \rightarrow  M_{\R} \times  M_{\R} / M$ can be defined as
\begin{equation} \label{reduced_sigma}
   \bar \sigma(b) = (b, [\tilde \sigma(b)]).
\end{equation}
  Fix some point $p_C$ in the interior of $C$. Let $e$ be an edge of $\partial\bar C$ and let $p^+$, $p^-$, $e^+$, $e^-$, $d^+$ and $d^-$ be as above. Choose points $q^+$ and $q^-$ on $d^+$ and $d^-$ respectively which are also in $R$. Denote by $Q_e^+$ the convex hull of $p_C$, $p^+$ and $p^-$ and by $Q_e^-$ the convex hull  of $p^+$, $p^-$, $q^+$ and $q^-$ (see Figure \ref{compact_sup}).  Let $Q_e = Q_e^+ \cup Q_e^-$. Since $R$ is convex, $Q_e$ is contained in $R$. 

Let $\phi_{K_C}$ be the support function of Definition \ref{defi:Kc}. Denote by $m^+_e$ (resp. $m^-_e$) the linear part of $\phi_{K_C}$ restricted to $P^+_e$ (resp. to $P^-_e$). Assuming that $n_{e^+}$, $n_e$ and $n_{e^-}$ are as in \eqref{ne:ne+}, one computes that $m^+_e = n_{e^+} - n_e$ and $m^-_e= n_e - n_{e^-}$. 
Given a point $b \in Q_e^+$, there are unique positive real numbers $t_1$, $t_2$ and $t_3$, satisfying $\sum t_j = 1$, such that we can write $b = t_1 p_C + t_2 p^+ + t_3 p^-$. Define
\[ \tilde \sigma(b) = a(t_2 m_e^+ + t_3m_e^-). \]
In particular $\tilde \sigma(p_C) = 0$,  $\tilde \sigma(p^+) =  a m_e^+$ and $\tilde \sigma(p^-) =  a m_e^-$.
Let us now define $\tilde \sigma$ on $Q_e^-$. On the segment from $p^+$ to $q^+$ define $\tilde \sigma$ by 
\begin{equation} \label{di+}
    \tilde \sigma ((1-s)p^+ + sq^+) =  a(1-s) m^+_e.
\end{equation}
Similarly, on the segment from $p^-$ to $q^-$ 
\[  \tilde \sigma ((1-s)p^- + sq^-) =  a(1-s)m^-_{e}. \]
Define $\tilde \sigma$ on the segment from $q^+$ to $q^-$ to be zero. By construction $\tilde \sigma$ extends to the interior of $Q_e^-$ (see Figure \ref{compact_sup} for a picture of the image of $Q_e$ under $\tilde \sigma$). We do this for every edge $e \subset \partial\bar C$ and we define $\tilde \sigma$ to be zero on the rest of $M_{\R}$. Define the reduced section $\bar \sigma$ by \eqref{reduced_sigma}. By construction we have $\bar \sigma(M_{\R}) \cap S = \emptyset$. 

We now show that we can extend $\bar \sigma$ to a section $\sigma: \R \times M_{\R} \rightarrow X$ with support inside $\tilde R$.  First extend it to a section $\sigma_Y: \R \times M_{\R} \rightarrow Y$, where $Y = \R \times M_{\R} \times M_{\R}/M$ by defining $\sigma_Y(t,p) = ((t,p), [\rho(t) \tilde{\sigma}(p)])$, where $\rho$ is a suitable bump function with values in the interval [0,1] and vanishing outside $[-\epsilon, \epsilon]$.  In fact we can also choose $\rho$ with values in the smaller interval $[0, 1- \epsilon]$ for small $\epsilon$; this choice will be more convenient later.   Clearly $\sigma_Y$ coincides with the zero section outside $\tilde R$.
We have to show that we can lift $\sigma_Y$ to the $S^1$ fibration $\alpha: X \rightarrow Y$, so that it coincides with the lift of the zero section outside $\tilde R$. Notice that $\tilde R$ is homeomorphic to a $3$-ball. If we consider a $3$-sphere as two copies of $\tilde R$ glued along its boundary, then $\sigma_Y$ on one copy of $\tilde R$ and the zero section on the other copy give a continuous map from $S^3$ to $Y$. The pull-back of $\alpha$ via this map is trivial, since all $S^1$-bundles are trivial on $S^3$.  Therefore the map can be lifted to the $S^1$-bundle. This gives lifts of $\sigma_Y$ and of the zero section which coincide on the boundary of $\tilde R$. This ends the construction of sections with compact support. 
\end{proof}

Observe that the line bundle on $\check X$ corresponding to a Lagrangian section of the type $\sigma_{aK_C}$ via the correspondence of Conjecture~\ref{hms:dim2} is $\mathcal{O}_{\check X}(-aD_C)$. 

\begin{figure}[!ht] 
\begin{center}
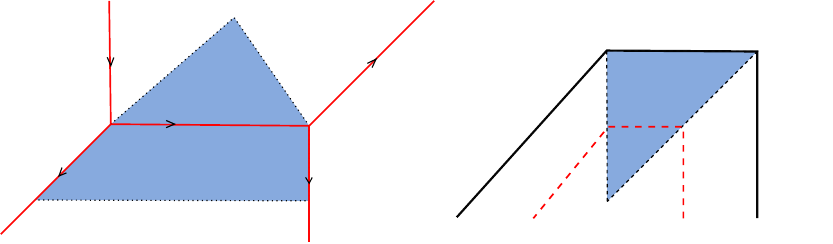
\caption{The picture on the right shows the image of $Q_e$ under $\tilde \sigma$ in the case $b_e=-1$ and $a=-1$ (shaded region). The dark line is $\tilde \sigma(\partial\bar C)$. The dashed (red) line is $\tilde \lambda(\partial\bar C)$.}
\label{compact_sup}
\end{center}
\end{figure}

A compactly supported section $\sigma$ and the zero section, as explained in the beginning of this section,  define a homology class in $H_{3}( X, \Z)$ denoted by $[\sigma\sigma_0]$. Now, given a bounded connected component $C$ of $M_{\R} - \Gamma$, consider the sphere $L_{\ell}$ defined by the numbers $\ell_e = b_e + 2$ for every edge $e$ bounding $C$ and denote it by $L_{-K_C}$. We then have

\begin{thm} \label{diff_sec} Up to some choice of orientation on $L_{-K_C}$ we have
\[ [\sigma_{-K_C}\sigma_0] = [L_{-K_C}]. \]
\end{thm}
\begin{proof} We use the map $\tilde \sigma: M_{\R} \rightarrow M_{\R}$ constructed in the proof of Proposition~\ref{com_sup:dim3} with $a=-1$. Define $\tilde \lambda: \bar C \rightarrow M_{\R}$ by 
\begin{equation} \label{lambda:sigma}
    \tilde \lambda = \frac{1}{2} \tilde \sigma|_{\bar C}
\end{equation}
and $\lambda(b) = (b, [ \tilde \lambda(b)])$ for all $b \in \bar C$ (see also Figure \ref{compact_sup} for a picture of $\tilde \lambda(\partial\bar C)$). It can then be checked that $\lambda(\partial\bar C) \subset S$ and
\[ L_{-K_C} = \alpha^{-1}(\lambda(\bar C)). \]

\medskip

{\it Step 1: definition of a set $Z$ containing $\lambda(\bar C)$.} We construct a subset $Z$ containing $\lambda(\bar C)$ which generalizes the set $Z$ constructed in the proof of Theorem \ref{dim2_diff}. Intuitively $Z$ looks like the shaded area depicted in Figure \ref{compact_sup_d2} after it is spun around its center and then thickened in the extra dimension. The center corresponds to the point $p_C$ chosen in the proof of Proposition~\ref{com_sup:dim3}. The singular points (i.e., the crosses in Figure \ref{compact_sup_d2}) span $\lambda(\partial\bar C)$. 

In the following we assume that the function $\rho$ used in the construction of $\sigma_Y$ in Proposition \ref{com_sup:dim3} has values in the interval $[0, 1- \epsilon]$ and that $\rho(0) = 1- \epsilon$. 
Given an edge $e$ of $\partial\bar C$, let $Q_e$ be the neighborhood of $e$ defined in the proof of Proposition \ref{com_sup:dim3} and let
\[ \begin{split}
               Q  & = \bigcup_{e \subset \partial\bar C} Q_e \\
    \tilde Q & = [- \epsilon, \epsilon] \times Q \subset \R \times M_{\R}. 
   \end{split} \]
Now define the set 
\[ Z = \{ (t, b, [r \, \rho(t) \tilde \sigma(b)] ) \, | \, (t,b) \in \tilde Q \ \text{and} \ r \in [0,1] \} \subset Y, \]
where, as usual, $Y= \R \times M_{\R} \times M_{\R}/ M$. Essentially, $Z$ fills the space between the  zero section and the section $\sigma_Y$. In particular the two sections, restricted to $\tilde Q$, form the boundary of $Z$. Notice also that $\lambda( \bar C) \subseteq Z$. 

For any $p \in \partial\bar C$, consider the ray $\rho_p$ in $M_{\R}$ which starts at $p_C$ and passes through $p$  and let $\tilde \rho_p = \rho_p \times [-\epsilon, \epsilon]$. Now let 
\[ Z_p = Z \cap \Log^{-1}(\tilde \rho_p). \]
We have that $Z_p$ is homeomorphic to a closed $3$-ball. Moreover
\[ Z = \bigcup_{p \in \partial\bar C} Z_p. \]

\medskip

{\it Step 2: properties of $Z_p$.} We now prove that 
\begin{equation} \label{z_int_s} 
 Z \cap S = \lambda(\partial\bar C), 
\end{equation}
and in particular that, for all $p \in \partial\bar C$,
\begin{equation} \label{zp_int_s}
   Z_p \cap S = \lambda(p).
\end{equation}
Since $\Log(Z) = \tilde Q$, we only have to study $Z \cap S \cap \Log^{-1}(\tilde Q)$.  Suppose first that $p$ is a point in the interior of an edge $e$ of $\partial\bar C$. Recall that the part of the surface $S$ that is mapped to the interior of $e$ is the cylinder $S_e$ defined in equation (\ref{cylinder:e}), where $\delta_e$ is the cycle defined in (\ref{cylinder_cycle}). 
We can write $p$ as $p =(1-s)p^+ + s p^-$, with $s \in (0,1)$. Then the intersection of $Z$ with the fibre $\Log^{-1}(p) \cong M_{\R} / M$ is
\[ Z \cap \Log^{-1}(p) = \{ [- r(1- \epsilon)((1-s)m_e^+ + s \, m^-_e)] \, | \, r \in [0,1] \} \]
The points in $Z \cap \Log^{-1}(p)$ which lie on $S_e$ correspond to the values of $r \in [0,1]$ such that
\[ \inn{- r(1- \epsilon)((1-s)m_e^+ + s \, m^-_e)}{ n_{\check e}} = \frac{1}{2} \mod \Z. \]
By definition we have that $\inn{m_e^+}{n_{\check e}} = \inn{m^-_e}{n_{\check e}} = -1$. So the above equation becomes
\[ r(1- \epsilon)= \frac{1}{2}, \]
since $r(1-\epsilon) \in [0,1]$. We conclude that for any $p$ in the interior of $e$,
\[ Z \cap S_e \cap \Log^{-1}(p) = \left[ - \frac{1}{2} ((1-s)m_e^+ + s \, m^-_e) \right] = \lambda(p). \]
Suppose now that $p$ is a vertex of $e$, e.g. $p = p^+$. Then, in appropriate affine coordinates, we can assume that 
equations (\ref{ne:ne+}) hold. In these coordinates $m^+_e = (-1,-1)$, therefore
\[ Z \cap \Log^{-1}(p) = \{ [- r(1- \epsilon)(-1,-1)] \, | \, r \in [0,1] \}.\]
Moreover, $S \cap \Log^{-1}(p)$ is  the set $T_p$ described in Remark \ref{algae:rem}. In these coordinates it has the shape depicted in Figure \ref{algae}. Then we see that 
\[ Z \cap S \cap \Log^{-1}(p) =  (Z \cap \Log^{-1}(p)) \cap T_p = [ (1/2, 1/2) ] = \lambda(p). \] 
Finally we have to check that $Z$ does not intersect $S$ over the interior of the edges $d^+$ and $d^-$. Let $q^+ \in d^+$ be the point defined in the proof of Proposition~\ref{com_sup:dim3}. We just need to check what happens over a point $p = (1-s)p^++sq^+$ with $s \in (0,1]$. Using formula \eqref{di+} we see that in this case 
\[ Z \cap \Log^{-1}(p) = \{ [ -r(1-\epsilon)(1-s)m_e^+] \, | \, r \in [0,1] \}. \]
On the other hand, points of this type are also on $S_{d^+}$ if they satisfy
\[ \inn{-r(1-\epsilon)(1-s)m_e^+}{n_{\check{d}^+}} =  \frac{1}{2} \mod \Z. \]
Using the fact that $n_{\check{d}^+} = \pm(n_{\check{e}}- n_{\check{e}^+})$ and $\inn{m_e^+}{n_{\check e}} = \inn{m^+_e}{n_{\check e^+}} = -1$, we can see that the above equation is never satisfied. Thus we have
\[ Z \cap S_{d^+} \cap \Log^{-1}(p) = \emptyset. \]
This proves \eqref{z_int_s} and \eqref{zp_int_s}. 

\medskip

{\it Step 3: conclusion.}  For every $p \in \partial\bar C$ define 
\[ W_p = \alpha^{-1}(Z_p) \]
and let 
\[ \quad W = \alpha^{-1} (Z). \]
Notice that since $Z_p$ is homeomorphic to a $3$-ball and is transversal to $S$ at $\lambda(p)$, it can be viewed as a $3$-ball in the normal bundle to $S$ at $\lambda(p)$. This implies that the fibration $\alpha|_{W_p}: W_p \rightarrow Z_p$ is always fibrewise isomorphic to the map $\alpha'$ given in \eqref{local:alpha}. In particular $W_p$ is homeomorphic to a $4$-ball. 

The classes $[\sigma_{-K_C}\sigma_0]$ and $[L_{-K_C}]$ can be viewed as classes in $H_3(W, \Z)$. We must show that they define the same class. 
Consider the set
\[ I = \bigcap_{p \subset \partial\bar C} Z_p = [-\epsilon, \epsilon] \times \{ p_C \} \times \{0 \} \subset Y. \]

Now define
\[ \widehat W = \{ (p, q) \in \partial\bar C \times X \, | \, q \in W_p \}, \quad \widehat Z = \{ (p, q) \in \partial\bar C \times Y \, | \, q \in Z_p \}. \]
These spaces are sort of (real) blow ups  along the set $I$ of $W$ and $Z$ respectively. Clearly there is an $S^1$-action on $\widehat W$ and a quotient map $\widehat \alpha: \widehat W \rightarrow \widehat Z$  inducing the commutative diagram
\begin{equation*} \begin{CD}
 \widehat W @>\widehat \alpha>> \widehat Z \\
@VVV  @VVV \\
 W @>\alpha>>  Z
\end{CD}
\end{equation*}
where the vertical arrows are the projections. We also have the other projection maps $\widehat W \rightarrow \partial\bar C$ and $\widehat Z \rightarrow \partial\bar C$. These are respectively a $4$-ball and a $3$-ball bundle over $\partial\bar C$. We claim that they are trivial bundles. In fact, the latter can be viewed as the unit $3$-ball bundle inside the normal bundle of $S$ in $Y$, restricted to the curve $\lambda(\partial\bar C)$. Since $S$ and $Y$ are oriented, the normal bundle to $S$ is orientable, therefore it becomes trivial when restricted to $\lambda(\partial\bar C)$. Similarly $\widehat W$ can be viewed as the unit $4$-ball bundle of the normal bundle of $S$ inside $X$, so it is also trivial. We denote by $\partial \widehat W$ and $\partial \widehat Z$ the corresponding $3$-sphere and $2$-sphere bundles respectively. In particular we have 
\begin{equation} \label{homologies}
    H_3(\widehat W, \Z) = H_2( \widehat W, \Z) = 0, \quad H_3(\partial \widehat W, \Z) \cong \Z, \quad H_2(\partial \widehat Z, \Z) \cong \Z.
\end{equation} 
In the last two groups, a generator is given by the fibre. 

Now let $\widehat I \subset \widehat Z$ be the preimage of $I \subset Z$. We have $\widehat I \cong  \partial\bar C \times [-\epsilon, \epsilon] $. Notice that the projections $\widehat Z - \widehat I \longrightarrow Z-I$ and $\widehat W - \widehat \alpha^{-1}(\widehat I) \longrightarrow W - \alpha^{-1}(I)$ are homeomorphisms. Thus, excision implies that 
\[ H_3(W, \alpha^{-1}(I); \Z) \cong H_3(\widehat W, \widehat \alpha^{-1}(\widehat I); \Z ).\]
Notice that $\alpha^{-1}(I) \cong S^1 \times [-\epsilon, \epsilon]$, so that its second and third homology vanish. This gives
\[ H_3(W; \Z) \cong H_3(W, \alpha^{-1}(I); \Z). \] 
Therefore we can view the classes  $[\sigma_{-K_C}\sigma_0]$ and $[L_{-K_C}]$ inside $H_3(\widehat W, \widehat \alpha^{-1}(\widehat I) )$. Notice that $ \widehat \alpha^{-1}(\widehat I) \cong \partial\bar C \times S^1 \times [-\epsilon, \epsilon]$. The long exact sequence in relative homology and \eqref{homologies} gives an isomorphism
\[  H_3(\widehat W, \widehat \alpha^{-1}(\widehat I); \Z ) \stackrel{\partial}{\longrightarrow} H_2(\widehat \alpha^{-1}(\widehat I);  \Z) \cong \Z. \]

The goal is to show that both classes  $[\sigma_{-K_C}\sigma_0]$ and $[L_{-K_C}]$ are mapped to a generator by the above isomorphism. We proceed in a similar way to Theorem \ref{dim2_diff}. First of all notice that by construction $[L_{-K_C}]$ is mapped to a generator by $\partial$. Moreover $[\sigma_{-K_C}\sigma_0]$ can be viewed as a class inside $H_3(\partial \widehat W, \widehat \alpha^{-1}(\widehat I); \Z )$. We have the diagram
\begin{equation*} \resizebox{\textwidth}{!}{$\begin{CD}
0 @>>> H_3(\partial \widehat W, \Z) @>>>  H_3(\partial \widehat W, \widehat \alpha^{-1}(\widehat I); \Z ) @>\partial>> H_2(\widehat \alpha^{-1}(\widehat I); \Z) @>>> 0 \\
@. @.   @VV\widehat \alpha_*V  \\
@. @. H_3(\partial \widehat Z, \widehat I; \Z)
\end{CD}$}
\end{equation*}
which is analogous to diagram \eqref{hopf}. Notice that $H_3(\partial \widehat Z, \widehat I; \Z) \cong H_3(\partial \widehat Z; \Z) \cong \Z$. Since $[\sigma_{-K_C}\sigma_0]$ is constructed by lifting sections, we have that $\widehat \alpha_* [\sigma_{-K_C}\sigma_0]$ is a generator of $H_3(\partial \widehat Z; \Z)$. 
Moreover, since the generator of  $H_3(\partial \widehat W, \Z)$ is the fibre of the sphere bundle
and $\widehat \alpha$ on the fibres is just the Hopf map, we have that $\widehat \alpha_*$ restricted to  $H_3(\partial \widehat W, \Z)$ is zero. Thus $\widehat \alpha_*$ must descend to an isomorphism of the quotient:
\[ \widehat \alpha_*: H_3(\partial \widehat W, \widehat \alpha^{-1}(\widehat I); \Z )/ H_3(\partial \widehat W, \Z)   \rightarrow  H_3(\partial \widehat Z, \widehat I; \Z). \]
Then  $[\sigma_{-K_C}\sigma_0]$ must descend to a generator of 
$$ H_3(\partial \widehat W, \widehat \alpha^{-1}(\widehat I); \Z )/ H_3(\partial \widehat W, \Z),$$ i.e. of  $H_2(\widehat \alpha^{-1}(\widehat I); \Z) $. This concludes the proof.
\end{proof}

% Fix some edge $e$ of $\partial C$ and define the closed sets
%\[ Z_1 = \bigcup_{p \in e} Z_p \ \ \text{and} \ \ Z_2 = \overline{Z - Z_1}. \]
%If $p^+$ and $p^-$ are the vertices of $e$ then
%\[ Z_1 \cap Z_2 = Z_{p^+} \cup Z_{p^-}. \]
%Let $W = \alpha^{-1}(Z)$, $W_j = \alpha^{-1}(Z_j)$ and $W_{p^{\pm}}= \alpha^{-1}(Z_{p^{\pm}})$. By construction, the classes $[\sigma_{-K_C}\sigma_0]$ and  $[L_{-K_C}]$ can be both viewed as elements of $H_3(W, \Z)$. Let us show that they coincide. Observe that $W_j$ is contractible for $j=1,2$. This can be seen as follows. Consider the set
%\[ \bar W_j = \{ (p, q) \in \partial C \times X \, | \, q \in W_p \} \]
%Then $\bar W_j$ is $4$-ball fibre bundle over a segment (it is a sort of blow-up of $W_j$ at the point $p_C$). Therefore $\bar W_j \cong [0,1] \times B_4$. In particular $\bar W_j$ is contractible. We also have the map $\bar W_j \rightarrow B$ given by $(p,q) \mapsto f(q)$. Then $W_j$ is obtained from $\bar W_j$ by contracting to a point the preimage of $p_C$ by this map.  
% The Mayer-Vietoris sequence for the decomposition $W= W_1 \cup W_2$ gives an isomorphism
%\begin{equation} \label{mv_dim3}
% H_3(W; \Z) \to H_2 (W_1 \cap W_2; \Z) = H_2 (W_{p^+} \cup W_{p^-}; \Z). 
%\end{equation}

Clearly by translating with respect to an arbitrary section $\sigma_{\kappa}$ we can describe other spheres as difference of sections. Applying formula (\ref{trans_sphere}) we have
\begin{cor} \label{comp_trans} We have
$$[\sigma_{-K_C + \kappa} \sigma_{\kappa}] = [L_{-K_C + 2\kappa}],$$ 
where $L_{-K_C + 2 \kappa}$ denotes the sphere whose twisting numbers are given by $b_e + 2 + 2 k_e$ for all edges $e \subset \partial C$. 
\end{cor}

\section{Intersections}

Recall, from formula (\ref{lift_sphere}), that a sphere $L_{\ell}$ is constructed as the union of the $S^1$-fibers over all points of a section $\lambda$ of the $\Log$ map, defined over $\bar C$. Clearly, if $\sigma$ is a section, the intersection points between $L_{\ell}$ and $\sigma$ are in one-to-one correspondence with the intersection points between $\lambda(\bar C)$ and $\bar \sigma$, where $\bar \sigma: N_{\R} \rightarrow X(B)_{\rd}$ is the reduced section. Therefore $\sigma \cap L_{\ell}$ is in one-to-one correspondence with the set $\{ b \in \bar C \, | \, \lambda(b) = \bar \sigma(b) \}$. Since $\lambda (\partial\bar C) \subset S$, this set is all contained in $C$, i.e., there are no intersection points on the boundary of $\bar C$. Up to translation by a section, we can assume that $\sigma$ is the zero section $\sigma_0$. 

We now consider $\lambda$ constructed as in \S \ref{vc:lag}, i.e., as the graph of the differential of the smooth function $\tilde \vartheta_{\epsilon}$ (see formula (\ref{lag_sph_sec})). Recall that $\tilde \vartheta_{\epsilon}$ is a smoothing of a semi-integral support function $\vartheta: | \Sigma_C| \rightarrow \R$ whose kinks are the twisting numbers $\ell$. 

For convenience let us define the following map 
\begin{equation} \label{Dtheta}
   \begin{split}
       D \vartheta: & \ \bar C \longrightarrow M_{\R}  \\
                \             & \        b    \longmapsto  (d \tilde \vartheta_{\epsilon})_b.
   \end{split} 
\end{equation}
We then have the one-to-one correspondence
\[ \sigma_0 \cap L_{\ell} \longleftrightarrow  \bigcup_{m \in M} (D \vartheta)^{-1}(m). \]
We will call a point $b$ in the right-hand set an intersection point. We wish to study transversality of intersections and to count intersection points. It is easy to see that an intersection point $b \in C$ corresponds to a transversal intersection point if and only if the Hessian $(\hes \tilde \vartheta_{\epsilon})_b$ is non-degenerate. 
Now define the curve $\gamma_{\ell}: \partial\bar C \rightarrow M_{\R}$ by setting
\begin{equation}
\label{gammaelleq}
\gamma_{\ell} = D \vartheta |_{\partial\bar C}.
\end{equation}
It follows from the results of \S \ref{lag_sections} and \S \ref{vc:lag} that $\gamma_{\ell}$  is a closed polygonal line joining the points $\tilde \theta_j$ defined in (\ref{ell_to_psi}). Fix an orientation on $N_{\R}$, and the induced orientation on $M_{\R}$. Then we have an anticlockwise cyclic indexing of two-dimensional cones of $\Sigma_C$ (see Figure \ref{vc_fan} and the proof of Proposition~\ref{kinks_to_fun} for our conventions). 

\begin{defi} \label{w_num} Recall that the winding number of the curve $\gamma_{\ell}$ around a point $m \in M_{\R}$, with $m \notin \gamma_{\ell}( \partial\bar C)$, is the degree of the map $\partial\bar C \rightarrow S^1$ given by $b \mapsto \frac{\gamma_{\ell}(b) - m}{|\gamma_{\ell}(b) - m|}$. Since for all $m \in M$, $m \notin \gamma_{\ell}(\partial\bar C)$, the winding number of $\gamma_{\ell}$ around $m$ is well defined for every $m \in M$. Denote it by $w_{\ell}(m)$.
\end{defi}
If $m \in M_{\R}$ is a regular value of $D \vartheta$ and $m \notin \gamma_{\ell}( \partial\bar C)$, then the degree of $D \vartheta$ at $m$ equals the winding number of $\gamma_{\ell}$ around $m$. In particular, if the winding number is not zero, $m$ is in the image of $D\vartheta$.

We have the following:

\begin{thm} \label{nondeg_hess} If the semi-integral support function $\vartheta: |\Sigma_C| \rightarrow \R$, (resp. its opposite $-\vartheta$), is strictly convex and $\epsilon > 0$ is sufficiently small, then for any $b \in C$, either $(d \tilde \vartheta_{\epsilon})_b \in \gamma_{\ell}$ or $(\hes \tilde \vartheta_{\epsilon})_b$ is positive (resp. negative) definite. In particular the image of $D\vartheta$ is the closed convex polygon whose boundary is~$\gamma_{\ell}$. 
\end{thm}

Before proving the theorem let us study a one-dimensional version of the problem. Let $h: \R \rightarrow \R$ be a continuous function which is piecewise affine, i.e., it is locally affine except at a finite number of points in $\R$. We call a point $t \in \R$ a non-smooth point if $h$ is not locally affine at $t$. At a non-smooth point $t$, let $m^+_t$ and $m^-_t$ be the slopes of $h$ respectively on the right and on the left of $t$. The function $h$ is convex if at every non-smooth point $t$ we have 
\[ m^+_t - m^-_t > 0. \]

Now let $\eta: \R \rightarrow \R$ be an even $C^{\infty}$ function, whose support is $[-\delta, \delta]$ and such that $\eta(t) > 0$ for all $t \in (-\delta, \delta)$.  Define
\[ \tilde h = h \ast \eta. \]
\begin{lem} \label{conv_convx1} In the above situation, we have $\tilde h''(t) \geq 0$ and equality holds if and only if the interval $(t- \delta, t+ \delta)$ does not contain non-smooth points. 
\end{lem}
\begin{proof}
By the definition of the convolution product we have
\[ \tilde h''(t) = \int_{\R} h(s) \eta''(t-s)ds. \]
Assume, without loss of generality, that $t=0$. Then
\[ \tilde h''(0) = \int_{\R} h(s) \eta''(-s)ds = \int_{-\delta}^{\delta} h(s) \eta''(-s)ds. \]
Now suppose that $(-\delta, \delta)$ contains $r$ non-smooth points $\{t_1, \ldots, t_r \}$. Define $t_0= -\delta$ and $t_{r+1} = \delta$. Suppose, moreover, that on the interval $[t_{j-1}, t_j]$, $h$ coincides with the affine function $m_j(t) = q_jt+c_j$. By convexity
\begin{equation} \label{convx1}
 q_{j+1} - q_j >0. 
 \end{equation}
We have
\begin{equation} \label{der_conv}
    \tilde h''(0) = \sum_{j=1}^{r+1} \int_{t_{j-1}}^{t_j} m_j(s) \eta''(-s)ds. 
\end{equation}
 Integration by parts gives
\begin{align*} 
\int_{t_{j-1}}^{t_j} m_j(s) \eta''(-s)ds &= m_j(t_{j-1})\eta'(-t_{j-1}) - m_j(t_j) \eta'(-t_{j})\\
&\quad -q_j( \eta(-t_{j}) - \eta(-t_{j-1}) ). 
\end{align*}
Using the fact that $\eta$ is even we obtain
	\begin{align*} \int_{t_{j-1}}^{t_j} m_j(s) \eta''(-s)ds &=  m_j(t_j) \eta'(t_{j})- m_j(t_{j-1})\eta'(t_{j-1})\\
		&\quad	-q_j( \eta(t_j) - \eta(t_{j-1}) ). 
	\end{align*}
Continuity of $h$ implies that 
\[ m_{j}(t_j) \eta'(t_{j}) = m_{j+1}(t_j) \eta'(t_{j}). \]
Moreover we have 
\[ \eta(t_0)=\eta'(t_0)= \eta(t_{r+1})= \eta'(t_{r+1})=0. \]
Using these facts, equation (\ref{der_conv}) becomes
\[
    \tilde h''(0) = \sum_{j=1}^{r} (q_{j+1}-q_{j}) \eta(t_j) > 0.
\]
The last inequality follows from (\ref{convx1}) and the fact that $t_j \in (- \delta, \delta)$, where $\eta$ is strictly positive. Clearly, if $(-\delta, \delta)$ does not contain non-smooth points, integration by parts shows that 
$\tilde h''(0) =0$
\end{proof}

\begin{proof}[Proof of Theorem \ref{nondeg_hess}] Let $\epsilon$ satisfy conditions (i)-(iii) listed after Corollary~\ref{bend_cor}. Let $e_j$ be the edge of $\partial\bar C$ whose vertices are $p_j$ and $p_{j+1}$. As we have already discussed (see the end of the proof of Theorem \ref{lag_spheres}), as $b$ moves inside $U_{e_j}$ from a point in $W_{p_j}$ to a point in $W_{p_{j+1}}$, then $(d \tilde \vartheta_{\epsilon})_b$ moves inside the line passing from $\tilde \theta_j$ and $\tilde \theta_{j+1}$. Using convexity of $\vartheta$ and Lemma~\ref{conv_convx1} one can refine the argument to show that for all $b \in \bar U_{e_j}$, $(d \tilde \vartheta_{\epsilon})_b$ belongs to the segment from $\tilde \theta_j$ to $\tilde \theta_{j+1}$.   Now let 
\[ U_{C} = \bigcup_{e \subset \partial\bar C} U_e. \]
The above argument shows that for any $b \in \bar U_C$, $(d \tilde \vartheta_\epsilon)_b \in \gamma_{\ell}$. 
We now prove that if $b \in C - \bar U_C$ then $(\hes \tilde \vartheta_{\epsilon})_b$ is positive definite. Thus we need to show that for any $v \in N_{\R}$, 
\begin{equation} \label{posdef}
 \inn{(\hes \tilde \vartheta_{\epsilon})_b v}{v}  > 0. 
\end{equation}
We can choose linear coordinates $(x_1,x_2)$ on $N_{\R}$ so that $v= \frac{\partial}{\partial {x_1}}$. In this case
\[ \inn{(\hes \tilde \vartheta_{\epsilon})_b v}{v} = \int_{\R} \left( \int_{\R} \vartheta(x_1,x_2) \frac{\partial^2 \mu_{\epsilon}}{\partial x_{1}^2} (b_1-x_1, b_2-x_2) dx_1\right) dx_2 \]
We claim that
\begin{equation} \label{posdef_int}
 \int_{\R} \vartheta(x_1,x_2) \frac{\partial^2 \mu_{\epsilon}}{\partial x_{1}^2} (b_1-x_1, b_2-x_2) \, dx_1 \geq 0
\end{equation}
and that if $b=(b_1, b_2) \in C - \bar U_C$, then the inequality is strict for some values of $x_2$. This would imply (\ref{posdef}). For fixed $x_2$, define the functions
\[ h(s) = \vartheta(s,x_2), \ \ \ \eta(s) = \mu_{\epsilon}(s, b_2-x_2). \] 
Notice that $h$ is a continuous, convex, piecewise affine function on $\R$. On the other hand, $\eta$ is identically zero if  $|x_2-b_2| \geq \epsilon$, but when $|x_2-b_2| < \epsilon$ it is smooth, even and its support is an interval $(-\delta_{x_2}, \delta_{x_2})$ with $\delta_{x_2} > 0$.  Moreover $\eta(s) > 0$ for all $s \in (-\delta_{x_2}, \delta_{x_2})$. 

The integral on the left hand side of (\ref{posdef_int}) coincides with $(h \ast \eta)''(b_1)$ and therefore the inequality follows from Lemma \ref{conv_convx1}.  This lemma also implies that the integral in (\ref{posdef_int}) is strictly positive for $x_2 \in (b_2-\epsilon, b_2+\epsilon)$,  if in the interval $(b_1-\delta_{x_2}, b_1+\delta_{x_2})$ there exists a non-smooth point of $h$.
We prove that this is true for some $x_2 \in (b_2-\epsilon, b_2+\epsilon)$. The condition  $b \in C - \bar U_C$ implies that the $\epsilon$-ball around $b$ intersects at least two non parallel edges of $\Sk_1$ emanating from the vertex $v_C$ of the subdivision. 
In particular, we can find a point $(\bar x_1, \bar x_2)$ in the $\epsilon$-ball around $b$ which lies on an edge of $\Sk_1$ not parallel to the $x_1$-axis. This implies that the line in $N_\R$ given by $\{ (s,\bar x_2) \, | \, s \in \R \}$ intersects this edge transversally in the point $(\bar x_1, \bar x_2)$. Thus $\bar x_1$ is a non-smooth point for $h$ defined by setting $x_2= \bar x_2$; moreover $\bar x_2 \in (b_2-\epsilon, b_2+\epsilon)$ and $\bar x_1 \in (b_1-\delta_{\bar x_2}, b_1 + \delta_{\bar x_2})$ since $(\bar x_1, \bar x_2)$ lies in the $\epsilon$-ball around $b$. 

The last statement follows easily. In fact, by the previous argument, if $m \notin \gamma_{\ell}( \partial\bar C)$ and $m = D \vartheta(b)$ for some $b \in C$, then $m$ is a regular value of $D\vartheta$ and its degree must be positive. Moreover, strict convexity of $\vartheta$ implies that $\gamma_{\ell}(\partial\bar C)$ is an embedded simple curve which encloses a convex bounded polygon. So $w_{\ell}(m)$ is $1$ if $m$ is inside this polygon and zero if $m$ is outside. This implies the last claim of the theorem.
\end{proof}

% $D_{\ell}$ is the image of the smooth map 
%\begin{equation} \label{sph:map}
   %                \begin{split}
      %                       \bar C & \longrightarrow  M_{\R} \\
         %                        b      & \mapsto                 (d \tilde \vartheta_{\epsilon})_b.
            %       \end{split}
%\end{equation}

We have the following 

\begin{cor} \label{convex:intersection}
If  $\vartheta: |\Sigma_C| \rightarrow \R$, or its opposite $-\vartheta$, is strictly convex and $\epsilon > 0$ is sufficiently small, the corresponding Lagrangian sphere $L_{\ell}$ intersects the zero section transversally and 
\[ \#( \sigma_0 \cap L_{\ell})= \#( D \vartheta(C) \cap M) \]
\end{cor}
\begin{proof} The previous theorem implies that all points of $m \in D\vartheta(C) \cap M$ are regular values of $D\vartheta$ and that every point in the preimage of $m$ contributes $+1$ to the degree. Since $m$ has degree $1$, it has only one preimage. 
\end{proof}

More generally, suppose that we have an arbitrary section $\sigma_{\kappa}$. Then the number of intersection points between $\sigma_{\kappa}$ and $L_{\ell}$ is the same as between the zero section and the translate of $L_{\ell}$ by $\sigma_{- \kappa}$, i.e., using formula (\ref{trans_sphere}):
\[ \# ( \sigma_{\kappa} \cap L_{\ell}) = \# ( \sigma_0 \cap L_{\ell-2 \kappa}) \]

\begin{ex}  \label{inters:p2} In the case of Example~\ref{triangle}, the spheres $L_{n}$ are labeled  by odd numbers $n$. In Figure~\ref{triangle_intersection} we picture some of the spheres and their intersection points with the zero section. Clearly in this case we 
have the formula
\begin{equation} \label{cap_p2}
  \# ( \sigma_0 \cap L_{2k+1}) = \left | \frac{k(k+1)}{2} \right | 
\end{equation}
Notice that $L_1$ and $L_{-1}$ are the only spheres which do not intersect the zero section. 

\begin{figure}[!ht] 
\begin{center}
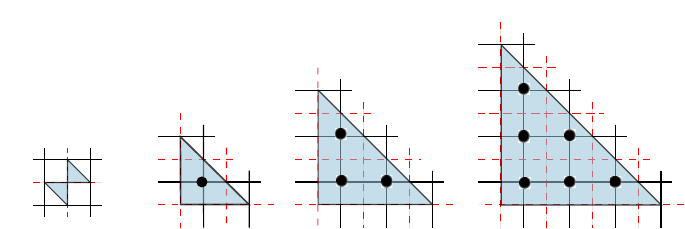\\
\vspace{1ex}\refstepcounter{figure}Figure \arabic{figure}.\label{triangle_intersection}
\end{center}
\end{figure}
\end{ex}

\begin{ex}
\label{a2d:mirrorredux}
Let us look at the case of Example~\ref{a2d:mirror}.   Assume that a subdivision of $P$ has been chosen so that every compact toric divisor in $\check X$ is a one point blow up of a Hirzebruch surface (e.g., as in Figure \ref{a2d:trop} for the case $d=3$). Then we can orient and number the bounding edges of each component $C_j$ as in Example~\ref{a2d:lineb}. Let us then construct the spheres of type $L_{\ell}$ over each component. Notice that when $j$ is odd, then $\ell$ must be of type $( \text{odd}, \text{odd},\text{odd}, \text{even}, \text{even})$ and when $j$ is even then it must be of type $( \text{even}, \text{odd}, \text{odd},\text{odd}, \text{even})$. Let us assume $j$ is odd, as the even case is similar. There are two spheres with minimal twisting numbers, these are given by numbers $(-1,1,-1,0,0)$ or $(1,-1,1,0,0)$. Let us denote them respectively by $L^+$ and $L^-$. 
These spheres are vanishing cycles with respect to the degeneration to the $A_{2d-1}$ singularity. Notice that $L^+$ and $L^-$ do not intersect the zero section. Another interesting sphere over $C_j$ is given by numbers $(-j+2, 1, 1, j+1, 2)$. In fact these correspond to the spheres $L_{-K_{C_{j}}}$ described in Theorem~\ref{diff_sec} as a difference of sections.
\end{ex}

Let us now discuss the general case where neither $\vartheta$ nor $- \vartheta$ are convex. In this case the polyhedral curve $\gamma_{\ell}$ can have a more complicated behavior. Let us define the following numbers
\begin{equation} \label{winding}
            \begin{split}
             h^{\text{odd}}(L_{\ell}) & = - \sum_{m \in M, \, w_{\ell}(m) < 0} w_{\ell}(m), \\
             h^{\text{even}}(L_{\ell}) & = \sum_{m \in M, \, w_{\ell}(m) > 0 } w_{\ell}(m).
            \end{split}
\end{equation}
\begin{con} We conjecture that for $\epsilon > 0$ sufficiently small, $m \in\linebreak D \vartheta(C) \cap M$ if and only if $w_{\ell}(m) \neq 0$. Moreover, given $b \in (D \vartheta)^{-1}(m)$, then $(\hes \tilde \vartheta_{\epsilon})_b$ is non-degenerate and the following holds
\begin{itemize}
 \item[i)] $w_{\ell}(m) < 0$ implies that $(\hes \tilde \vartheta_{\epsilon})_b$ has signature $(1,1)$;
 \item[ii)] $w_{\ell}(m) > 0$, implies that $(\hes \tilde \vartheta_{\epsilon})_b$ has signature $(2, 0)$ or $(0,2)$. 
\end{itemize}
In particular we have that $L_{\ell}$ intersects $\sigma_0$ transversally and
\[ \#( \sigma_0 \cap L_{\ell}) = h^{\text{odd}}(L_{\ell}) + h^{\text{even}}(L_{\ell}). \]
\end{con}

\begin{rem} \label{floer}
If the above conjecture is true, then it should not be difficult to prove the following isomorphism of $\C$-vector spaces   
\begin{equation} \label{hf:sec_sph}
   \begin{split}
      HF_1( \sigma_0, L_{\ell}) & \cong \C^{h^{\text{odd}}(L_{\ell})}, \\
      HF_{\text{even}}( \sigma_0, L_{\ell}) & \cong \C^{h^{\text{even}}(L_{\ell})}.
   \end{split}
\end{equation}
Let us sketch how one could prove this. Every intersection point $b \in \sigma_0 \cap L_{\ell}$ corresponds to an intersection point $\bar b$ between the section $\lambda$ (i.e. the graph of $d \tilde \vartheta_{\epsilon}$) and the zero section $\bar \sigma_0$ in the reduced space.  Moreover we have
\begin{lem} If $X$ admits an $\omega$-compatible almost complex structure $J$ and an $S^1$ invariant $(3,0)$-form $\Omega$ with $\Omega \wedge \bar \Omega > 0$, then there are gradings $\bar \theta$, $\theta$, $\bar \psi$ and $\psi$ of $\bar \sigma_0$, $\sigma_0$, $\lambda$ and $L_{\ell}$ respectively such that for every $b \in \sigma_0 \cap L_{\ell}$ the Maslov degree computed with respect to these gradings satisfy
\[ \deg_{\sigma_0, L_{\ell}}(b) = \deg_{\bar \sigma_0, \lambda}(\bar b).\]
\end{lem}
\begin{proof} Here we compute the Maslov degree of a point with the method suggested in \cite{TY}. 
 Let $\eta$ be the vector field generating the $S^1$ action, normalized so that $|\eta| =1$ with respect to the metric given by $\omega$ and $J$.  We can define a reduced $(2,0)$ form $\Omega_{\rd}$ on $X_{\rd} = \mu^{-1}(0)/S^1$ as follows. If $\bar v_1$ and $\bar v_2$ are two tangent vectors on $X_{\rd}$ and $v_1, v_2$ are a choice of lifts to $\mu^{-1}(0)$ let 
\[ \Omega_{\rd}(\bar v_1, \bar v_2) = \Omega(\eta, v_1, v_2). \]
We have that $\Omega_{\rd}$ induces an almost complex structure which is compatible with the reduced symplectic form on $X_{\rd}$ (see \cite{Goldstein1} or \cite{Gross_spLagEx}).  We also have that if $H$ denotes the complex vector subspace generated by $\eta$ and $J \eta$, then its orthogonal complement $H^{\perp}$ is contained in the tangent space of $\mu^{-1}(0)$ and is mapped to the tangent space of $X_{\rd}$ via a complex isometry.

Now let $\bar \theta$ and $\bar \psi$ be gradings of $\bar \sigma_0$ and $\lambda$ defined using $\Omega_{\rd}$, i.e. real functions such that 
\[ \Omega_{\rd}|_{\bar \sigma_0} = r e^{i\pi \bar \theta} \vl_{\bar \sigma_0} \quad \text{and} \quad \Omega_{\rd}|_{\lambda} = q e^{i\pi \bar \psi} \vl_{\lambda} \]
for positive real functions $r$ and $q$ on $\bar \sigma_0$ and $\lambda$ respectively. Now we can choose an orthonormal basis $\{ \bar e_1, \bar e_2 \}$ of $T_{\bar b} \bar \sigma_0$ and $\{ \bar f_1, \bar f_2 \}$ of $T_{\bar b} \lambda$ such that 
\[ \bar f_j = e^{i\pi \alpha_j} \bar e_j \]
for numbers $\alpha_j \in (0,1)$. Then we have
\[ \deg_{\bar \sigma_0, \lambda}(b) =  \alpha_1+ \alpha_2 + \bar \theta(b) - \bar \psi(b) . \]
Let us define the grading on $L_{\ell}$. Let $\{f_1, f_2 \}$ be a lift of $\{ \bar f_1, \bar f_2 \}$ to $T_b \mu^{-1}(0)$ which is orthogonal to $\eta$. Then $\{\eta, f_1, f_2 \}$ is an orthonormal basis of $T_bL_{\ell}$. We have
\[ \Omega(\eta, f_1, f_2) = \Omega_{\rd}(\bar f_1, \bar f_2) = q e^{i \pi \bar \psi}.\]
Thus we can define a grading of $L_{\ell}$ to be $\psi = \bar \psi$. 

Let us now define the grading on $\sigma_0$. Define $S$ to be the tangent bundle to $\sigma_0 \cap \mu^{-1}(0)$. It can be shown that since $\sigma_0$ is Lagrangian and transversal to $\eta$ we have
\[ TX = H \oplus S \oplus JS. \]
Now let $e_0$ be a unitary vector field along $\sigma_0 \cap \mu^{-1}(0)$, which is tangent to $\sigma_0$ and orthogonal to $\sigma_0 \cap \mu^{-1}(0)$. We can decompose $e_0 = e_0'+e_0''$ where $e_0' \in H$ and $e_0'' \in S \oplus JS$. 
In particular, since $e_0$ is transversal to $\mu^{-1}(0)$, there is a function $\alpha_0: \mu^{-1}(0) \cap \sigma_0 \rightarrow (0,1)$ such that $e_0' = \ell e^{-i \pi \alpha_0} \eta$ for some positive real function $\ell$. For any frame $\{e_1, e_2 \}$ tangent to $\sigma_0 \cap \mu^{-1}(0)$ and lifting $\{ \bar e_1, \bar e_2 \}$ we have
\[ \begin{split}
           \Omega(e_0 , e_1, e_2) & = \Omega(e_0', e_1, e_2) =  \ell e^{-i \pi \alpha_0}\Omega(\eta, e_1, e_2)  \\
            & = \ell e^{-i \pi \alpha_0}\Omega_{\rd}( \bar e_1, \bar e_2) = \ell r e^{i \pi (\bar \theta- \alpha_0)}.
  \end{split} 
\]
Thus we can define a grading of $\sigma_0$ to be $\theta = \bar \theta - \alpha_0$. 

Let us now compute the Maslov degree at an intersection point $b$. We claim that up to a hamiltonian deformation of $\sigma_0$, we can assume that at each intersection point $b$, the tangent space to $\sigma_0 \cap \mu^{-1}(0)$ is orthogonal to $\eta$, i.e. $S_b  \subset H^{\perp}$. Indeed, given any smooth function $h$ on $X$, the hamiltonian vector field of $h \cdot\mu$ restricted to $\mu^{-1}(0)$ is $h \eta$. Thus we can locally choose $h$ near an intersection so that the hamiltonian flow of $h \mu$ will twist the tangent of space of $\sigma_0 \cap \mu^{-1}(0)$ so that it becomes orthogonal to $\eta$. 

In particular, if $e_1, e_2 \in S_b$ are the lifts of $\bar e_1, \bar e_2$ and $f_1, f_2$ are lifts of $\bar f_1, \bar f_2$  so that $\{\eta, f_1, f_2 \}$ is an orthonormal basis of $T_bL_{\ell}$, we have that $f_j = e^{i\pi \alpha_j} e_j$, since all these vectors are in $H^{\perp}$. Moreover we must have $e_0(b) \in H$, i.e. $e_0(b) = e^{-i \pi \alpha_0} \eta$. Therefore 
\[ \begin{split}
    \deg_{\sigma_0, L_{\ell}}(b) & = \alpha_0 + \alpha_1 + \alpha_2 + \theta(b) - \psi(b)  \\ 
                                 & = \alpha_1 + \alpha_2 + \bar \theta(b) - \bar \psi(b) =  \deg_{\bar \sigma_0, \lambda}(b).
  \end{split} \]
\end{proof}
In order to use this lemma in our situation one should prove the existence of the $S^1$-invariant form $\Omega$.  Now, recall that in the case of the cotangent bundle of a smooth $n$-dimensional manifold, the Maslov degree at a transversal intersection point $b$ between the zero section and the graph of an exact one form $d \phi$ is $n-\operatorname{Ind}(\phi, b)$, where $\operatorname{Ind}(\phi, b)$ denotes the Morse index of $\phi$ at $b$. Since the reduced space is a cotangent bundle (modulo a lattice), the above formula holds in the reduced space. Hence, the above lemma would give 
\[ \deg_{\sigma_0, L_{\ell}}(b) = 2 - \operatorname{Ind}(\tilde \vartheta_{\epsilon}, \bar b). \]

Now, observe that $J$-holomorphic strips connecting two intersection\linebreak points $b$ and $b'$ with boundary on $\sigma_0$ and $L_{\ell}$ may exist only if $D \vartheta(b) = D \vartheta(b')$. This is true for topological reasons. Indeed, a strip is a topological disc with half of the boundary on $\sigma_0$ and the other half on $L_{\ell}$.  Assume $D \vartheta(b) = m \in M$. Then the first half of the boundary gives a path $(\beta(t), m)$, with $\beta(0) = b$ and $\beta(1) = b'$ and the second half gives a path from $(b, D \vartheta(b))$ to $(b', D\vartheta(b'))$ in $N_{\R} \times M_{\R}$ along the graph of $D \vartheta$. If the disc exists with this boundary in the quotient $N_{\R} \times M_{\R}/ M$, then the latter path must close up, i.e., we must have  $D \vartheta(b)=D\vartheta(b')$. 

We then have that in computing the Floer differential, cancellations may happen only between points in $(D \vartheta)^{-1}(m)$.
Point (i) of the conjecture implies that if $w_{\ell}(m) < 0$, then all preimages of $m$ have Maslov degree $1$, therefore there can be no cancellations, i.e., every point in the preimage of $m$ contributes $1$ to the dimension of $HF_{1}( \sigma_0, L_{\ell})$. In the case $w_{\ell}(m) > 0$, then point (ii) implies that every point in the preimage of $m$ contributes $1$ to the degree of $m$, i.e., to  $w_{\ell}(m)$. We will see in the next section that if $w_{\ell}(m) > 0$, then necessarily $w_{\ell}(m) = 1$. In particular every point $m$ with $w_{\ell}(m) > 0$ contributes one either to the dimension of $HF_0(\sigma_0, L_{\ell})$ or to the dimension of $HF_2(\sigma_0, L_{\ell})$. In particular, Theorem \ref{nondeg_hess} and Corollary \ref{convex:intersection} imply that the conjecture is true in the case $\vartheta$ is strictly convex, so we have that 
\begin{equation} \label{strconv:hf}
   \vartheta \ \ \text{strictly convex} \ \ \ \Longrightarrow \ \ \ 
                                \begin{cases} 
                                             HF_1( \sigma_0, L_{\ell}) = HF_0( \sigma_0, L_{\ell}) = 0 \\ 
                                             HF_2( \sigma_0, L_{\ell}) \cong \C^ {\#( D\vartheta(C) \cap M)}
                                \end{cases}  
\end{equation}
and similarily
\begin{equation} \label{-strconv:hf}
   - \vartheta \ \ \text{strictly convex} \ \ \ \Longrightarrow \ \ \ 
                                \begin{cases} 
                                             HF_1( \sigma_0, L_{\ell}) = HF_2( \sigma_0, L_{\ell}) = 0 \\ 
                                             HF_0( \sigma_0, L_{\ell}) \cong \C^ {\#( D\vartheta(C) \cap M)},
                                \end{cases}  
\end{equation}
where the above isomorphisms are as vector spaces over $\C$. Even if the above conjecture is not true, we still believe that equalities \eqref{hf:sec_sph} are true, but it may be harder to prove them. We believe that these can be proved with methods similar to those of Abouzaid \cite{abo_coord_ring}, \cite{Ab_HMS}. 
\end{rem}

\section{Homological mirror symmetry correspondence} \label{hms}

If $\sigma$ is a compactly supported section and $L$ is one of the spheres constructed above such that 
\[ [\sigma\sigma_0] = [L], \]
then we can guess the sheaf $\mathscr{E}_{L}$ on $\check X$ associated to $L$ as follows.  Letting $\mathscr{L}_{\sigma}$ be the line bundle on $\check X$ associated to $\sigma$, then we should have a short exact sequence
\[ 0  \rightarrow \mathcal{O}_{\check X} \rightarrow \mathscr{L}_{\sigma} \rightarrow \mathscr{E}_{L} \rightarrow 0. \]
In particular, Conjecture \ref{hms:dim2} and Proposition \ref{com_sup:dim3} (respectively \ref{cs2:prop}) imply that $\sigma_{-K_C}$  (resp.\ $\sigma_2$) is a compactly supported section and its mirror line bundle is $\mathcal{O}_{\check X}(D_C)$ where $D_C$ is the divisor corresponding to $C$ (respectively $D_C = \PP^1 \subset \check X$). We have a short exact sequence
\[ 0  \rightarrow \mathcal{O}_{\check X} \rightarrow \mathcal{O}_{\check X}(D_C) \rightarrow \mathcal{O}_{\check X}(D_C)|_{D_C} \rightarrow 0, \]
where the second arrow is given by tensoring with a section vanishing on $D_C$. Using Theorem \ref{diff_sec} (resp. Theorem \ref{dim2_diff}) we conclude that we should have
\begin{equation} \label{hms:vanishing1}
  \mathscr{E}_{L_{- K_C}} = \mathcal{O}_{\check X}(D_C)|_{D_C},  \ \ \ (\text{resp.} \ \ \mathscr{E}_{L_{1}} = \mathcal{O}_{\check X}(\PP^1)|_{\PP^1}).
\end{equation}
More generally, from formula (\ref{trans_sphere}) (resp. (\ref{transl_d2})) and the fact that translation by a section should be mirror to tensoring with the corresponding line bundle, we should have
\begin{equation} \label{hms:vanishing2} 
\mathscr{E}_{L_{-K_C+2\kappa}} = \mathscr L_{-\kappa} \otimes \mathcal{O}_{\check X}(D_C)|_{D_C} \ \ \ (\text{resp.} \ \ \mathscr{E}_{L_{1+k}} = \mathscr L_{-k} \otimes \mathcal{O}_{\check X}(\PP^1)|_{\PP^1}).
\end{equation}
See also Corollary \ref{comp_trans}.

\medskip

In the three-dimensional case, let us formulate this as a conjecture 
\begin{con} \label{hms:spheres} The Lagrangian sections of $f: X \rightarrow \R^n$ and the La\-gran\-gian spheres of type $L_{\ell}$ defined on each bounded connected component $C$ of $N_{\R} - \Delta$ generate the derived partially wrapped Fukaya category of $X$ and the correspondence which maps a section to a line bundle over $\check X$ as in Conjecture \ref{hms:dim2} and a sphere $L_{\ell}$ to the sheaf $\mathscr E_{L_{\ell}}$ supported on the divisor $D_C \subset \check X$ given by 
\begin{equation} \label{hms:vanishing3} 
 \mathscr E_{L_{\ell}} = \mathscr L_{\frac{K_C - \ell}{2}}|_{D_C}
\end{equation}
induces an embedding of the derived partially wrapped Fukaya category in $\DC(\check X)$.
\end{con}

In dimension $2$, in \cite{chan:an} Chan proves this conjecture for the part concerning the spheres and in \cite{chan:ueda} Chan and Ueda prove this conjecture for the sections. They do not seem to discuss whether the two correspondences, i.e., the one concerning the spheres and the one concerning the sections, are compatible with each other. We now look at examples in more detail and discuss some first simple evidence supporting the above conjecture. 

\begin{ex} When $\check X$ is the total space of $ \mathcal{O}_{\PP^1}(-2)$, formulas (\ref{hms:vanishing1}) and (\ref{hms:vanishing2}) give 
\[ \mathscr E_{L_1} = \mathcal{O}_{\PP^1}(-2) \ \ \ \text{and} \ \ \ \mathscr E_{L_{k}} = \mathscr L_{1-k} \otimes \mathcal{O}_{\PP^1}(-2) = \mathcal{O}_{\PP^1}(-k-1).\] 
Notice that $\sigma_0$ intersects $L_1$ at only one point, so we must have:
\[ HF^*(\sigma_0, L_1) \otimes_{\R} \C  \cong \C. \]
This corresponds well with the  fact that 
\[ \Hom^j_{\DC(\check X)}(\mathcal{O}_{\check X}, \mathcal{O}_{\PP^1}(-2)) \cong H^j(\PP^1, \mathcal{O}_{\PP^1}(-2)) 
                                   = \begin{cases}
                                     \C \quad \text{for} \ j=1, \\
                                     0 \quad \text{otherwise}.
                                    \end{cases}                                      \]
\end{ex}
More generally,  $\# (\sigma_0 \cap L_k) = |k|$ and all intersections have the same Maslov degree, hence
\[ HF^*(\sigma_0, L_k) \otimes_{\R} \C  \cong \C^{|k|}. \]
This corresponds to the fact that 
\[ \Hom^*_{\DC(\check X)}(\mathcal{O}_{\check X}, \mathcal{O}_{\PP^1}(-k-1)) \cong H^*(\check X, \mathcal{O}_{\PP^1}(-k-1)) \cong \C^{|k|} \]
Notice that as line bundles on $\PP^1$, $\mathscr E_{L_{-k}} = \mathscr E_{L_{k}}^{\vee} \otimes \omega_{\PP^1}$, so  $\mathscr E_{L_{-k}}$ and  $\mathscr E_{L_{k}}$ are related by Serre duality on $\PP^1$.  This matches the fact that $\#(\sigma_0 \cap L_{-k}) = \#(\sigma_0 \cap L_{k})$

\begin{ex} In dimension 3, when $\check X$ is an open set in the total space of $ \mathcal{O}_{\PP^2}(-3)$ as in Examples \ref{triangle} and \ref{localP2exampleredux}, 
we have
\[ [\sigma_3\sigma_0] = L_3. \]
Then, since $K_C = -3$, formula (\ref{hms:vanishing3}) becomes
\[ \mathscr E_{L_{2k+1}} =  \mathscr L_{-k-2}|_{\PP^2} = \mathcal{O}_{\PP^2}(-k-2). \]
Formula (\ref{cap_p2}) and the observations \eqref{strconv:hf} and \eqref{-strconv:hf} imply
\[ HF^*(\sigma_0, L_{2k+1}) \otimes_{\R} \C  \cong \C^{|\frac{k(k+1)}{2}|}. \]
This matches the fact that
\[ \Hom^*_{\DC(\check X)}(\mathcal{O}_{\check X}, \mathcal{O}_{\PP^2}(-k-2)) \cong H^*(\PP^2, \mathcal{O}_{\PP^2}(-k-2)) \cong \C^{\left| \frac{k(k+1)}{2} \right | }. \]
Notice also in this case, that $\mathscr E_{L_{-(2k+1)}} = \mathcal{O}_{\PP^2}(k-1)$  and $\mathscr E_{L_{2k+1}} = \mathcal{O}_{\PP^2}(-k-2)$ are related by Serre duality on $\PP^2$. This matches the fact that $\#(\sigma_0 \cap L_{-(2k+1)}) = \#(\sigma_0 \cap L_{2k+1})$.
\end{ex}

\subsection{Lagrangian spheres and cohomology of line bundles}
In this section we gather some more evidence supporting Conjecture \ref{hms:spheres}. In particular we will prove the following: 

\begin{thm} \label{coho_winding} For all Lagrangian spheres $L_{\ell}$ in $X$ over $C$, we have the following
\begin{equation} 
 \begin{split}
 \Hom^1_{\DC(\check X)}(\mathcal{O}_{\check X},  \mathscr L_{\frac{K_C - \ell}{2}}|_{D_C}) & \cong \C^{h^{\text{odd}}(L_{\ell})} \\
 \Hom^{\text{even}}_{\DC(\check X)}(\mathcal{O}_{\check X},  \mathscr L_{\frac{K_C - \ell}{2}}|_{D_C}) & \cong \C^{h^{\text{even}}(L_{\ell})} 
 \end{split}
\end{equation}
where $h^{\text{odd}}(L_{\ell})$ and $h^{\text{even}}(L_{\ell})$ are defined in (\ref{winding}).
\end{thm}

\begin{rem}
In the case the semi-integral support function $\vartheta$ with kinks $\ell$ (or its opposite  $-\vartheta$) is strictly convex, the above theorem, together with observations \eqref{strconv:hf} and \eqref{-strconv:hf}, implies
\begin{equation} 
 \begin{split}
 \Hom^1_{\DC(\check X)}(\mathcal{O}_{\check X},  \mathscr L_{\frac{K_C - \ell}{2}}|_{D_C}) & \cong HF^1(\sigma_0, L_{\ell}) = 0, \\
 \Hom^{\text{even}}_{\DC(\check X)}(\mathcal{O}_{\check X},  \mathscr L_{\frac{K_C - \ell}{2}}|_{D_C}) & \cong HF^{\text{even}}(\sigma_0, L_{\ell}).
 \end{split}
\end{equation}
\qed
\end{rem}

We have
\[ \Hom^*_{\DC(\check X)}(\mathcal{O}_{\check X},  \mathscr L_{\frac{K_C - \ell}{2}}|_{D_C}) = H^*(D_C,  \mathscr L_{\frac{K_C - \ell}{2}}|_{D_C}). \]
Recall that $D_C$ is a toric surface given by the fan $\Sigma_C$ whose cones are the tangent wedges to the simplices containing $v_C \in C$. Let us denote by $\Sigma_C(1)$ the one-dimensional cones of $\Sigma_C$. If $\check e$ is an edge emanating from $v_C$, its tangent wedge is a one-dimensional cone of $\Sigma_C$, so by slight abuse of notation we write $\check e \in \Sigma_C(1)$. We also assume that for every $\check e \in \Sigma_C(1)$ the tangent vector $n_{\check e}$ points outward from $v_C$, so that $n_{\check e}$ is a primitive integral generator of the one-dimensional cone $\check e$. 

We now describe a support function on $|\Sigma_C|$ corresponding to the line bundle $\mathscr L_{\frac{K_C - \ell}{2}}|_{D_C}$. Let $\psi_{K_C}$ be a support function for the canonical bundle on $D_C$. In particular we require that 
\begin{equation} \label{psi:def}
    \psi_{K_C}(n_{\check e}) = -1
\end{equation} 
for all $\check e \in \Sigma_C(1)$. It can be easily verified that a support function for the line bundle $\mathscr L_{\frac{K_C - \ell}{2}}|_{D_C}$ on $D_C$ is given by 
\begin{equation} \label{psi:teta}
     \psi = \frac{1}{2} \psi_{K_C} - \vartheta,
\end{equation}
where $\vartheta$ is the semi-integral support function whose kinks are $\ell/2$. 

Cohomology of line bundles on a toric variety can be computed in terms of support functions as explained in Section 3.5 of Fulton's book \cite{fulton:toric} (or in Chapter 9 of  the book by Cox, Little and Schenk \cite{cox:little:schenck}). Recall that our convention for support functions is opposite to the one used by Fulton (op. cit.) or by Cox, Little and Schenk (op. cit.). In our case $|\Sigma_C| = N_{\R} \cong \R^2$, so we proceed as follows. For every $m \in M$ let 
\[ Z(m) = \{ n \in N_{\R} \, | \, \inn{m}{n}+ \psi(n) \geq 0 \}. \]
For every $i$, consider the relative cohomology groups:
\[ H^i(\psi, m) := H^i(N_{\R}, N_{\R} - Z(m); \C). \]
Then the cohomology of the line bundle $\mathscr L_{\psi}$ over $D_C$, given by the support function $\psi$, is computed by
\[ H^i(D_C, \mathscr L_{\psi}) = \bigoplus_{m \in M} H^i(\psi, m). \]
Now let 
\begin{equation*}
        \begin{split}
        Q_0 & = \{ m \in M \, | \, Z(m) = N_{\R} \}, \\
        Q_1 & = \{ m \in M \, | \, \{ 0 \} \subsetneq Z(m)  \subsetneq N_{\R}\}, \\
        Q_2 & = \{ m \in M \, | \, Z(m) = \{ 0 \} \}.
        \end{split}
\end{equation*}
Furthermore, for every $m \in Q_1$, let 
\[ h_{\psi}(m) = b_0(N_{\R} - Z(m)) - 1, \]
where $b_0$ denotes the $0$-th Betti number. Then one can show that
\begin{equation} \label{coho_L}
        \begin{split}
        H^0 (D_C, \mathscr L_{\psi}) & \cong \C^{\# Q_0}, \\
        H^1 (D_C, \mathscr L_{\psi}) & \cong \bigoplus_{m \in Q_1} \C^{h_{\psi}(m)}, \\
        H^2 (D_C, \mathscr L_{\psi}) &  \cong \C^{\# Q_2}.
        \end{split}
\end{equation}
Letting $\psi$ be given by (\ref{psi:teta}), we have the following:
\begin{lem} \label{inq_psi:theta}
For every $\check e \in \Sigma_C(1)$ and every $m \in M$, we have 
	\begin{align*} \inn{m}{n_{\check e}} + \psi(n_{\check e}) \geq 0 \ \Longleftrightarrow \ \inn{m}{n_{\check e}} - \vartheta(n_{\check e}) >  0, \\
 \inn{m}{n_{\check e}} + \psi(n_{\check e}) < 0 \ \Longleftrightarrow \ \inn{m}{n_{\check e}} - \vartheta(n_{\check e}) <  0. 
\end{align*}
\end{lem} 

\begin{proof} Both equivalences follow from (\ref{psi:def}) and (\ref{psi:teta}). In fact, 
\[ \inn{m}{n_{\check e}} + \psi(n_{\check e}) \geq 0 \ \Leftrightarrow \ \inn{m}{n_{\check e}} - \vartheta(n_{\check e}) \geq \frac{1}{2}. \]
On the other hand recall that $\vartheta(n_{\check e}) = \frac{1}{2} \mod \Z$ for every $\check e \in \Sigma_C(1)$. Therefore 
\[ \inn{m}{n_{\check e}} - \vartheta(n_{\check e}) >  0 \ \Rightarrow \ \inn{m}{n_{\check e}} - \vartheta(n_{\check e}) \geq  \frac{1}{2}. \]
Hence applying (\ref{psi:def}) and (\ref{psi:teta}) gives the first equivalence. The second equivalence follows immediately since we cannot have $\inn{m}{n_{\check e}} - \vartheta(n_{\check e}) =0$. \end{proof}

Let 
\[ W(m)  = \{ n \in N_{\R} \, | \, \inn{m}{n} - \vartheta(n) \geq 0 \}. \]

\begin{lem} \label{ZW} For every $m \in M$ we have 
	\begin{align*} Z(m) = N_{\R} \ \Longleftrightarrow \ W(m) = N_{\R},  \\
 Z(m) = \{ 0 \} \ \Longleftrightarrow \ W(m) = \{ 0 \}.
	\end{align*}
Moreover, for all $m \in Q_1$, 
\[ b_0(N_{\R} - Z(m)) =   b_0(N_{\R} - W(m)). \]
\end{lem}

\begin{proof} We have that $Z(m) = N_{\R}$ if and only if $\inn{m}{n_{\check e}} + \psi(n_{\check e}) \geq 0$ for all $\check e \in \Sigma_C(1)$. Similarly $Z(m) = \{ 0 \}$ if and only if $\inn{m}{n_{\check e}} + \psi(n_{\check e}) <  0$ for all $\check e \in \Sigma_C(1)$. Therefore the first two equivalences follow from Lemma \ref{inq_psi:theta}. 

Label a cone of $\Sigma_C$, of any dimension, with a ``$+$'' if it is entirely contained in $Z(m)$ and denote by $Z^+(m)$ the union of the cones of $\Sigma_C$ which are labeled with a ``$+$''. Observe that a two-dimensional cone is labeled with a ``$+$'' if and only if both its boundary edges are labeled with a ``$+$''. 
It can be easily proved that $N_{\R} - Z(m)$ is a deformation retract of $N_{\R} - Z^+(m)$.  Therefore 
\[ b_0(N_{\R} - Z(m)) =   b_0(N_{\R} - Z^+(m)). \]
We can also consider a different labeling. Label a cone with a ``$+$'' if it is entirely contained in $W(m)$ and denote by $W^+(m)$ the union of the cones of $\Sigma_C$ which are labeled with a ``$+$''. Again, one can show that  $N_{\R} - W(m)$ is a deformation retract of $N_{\R} - W^+(m)$, so that 
\[ b_0(N_{\R} - W(m)) =   b_0(N_{\R} - W^+(m)). \]
On the other hand Lemma \ref{inq_psi:theta} implies that $Z^+(m) = W^+(m)$. This completes the proof of the Lemma. 
\end{proof}
This lemma implies that we can compute the cohomology of $\mathscr L_{\psi}$ by replacing the subsets $Z(m)$ with the subsets $W(m)$.

Given an oriented curve $\gamma$ and a point $m \notin \gamma$, the winding number of $\gamma$ around $m$ (see Definition \ref{w_num}) coincides with the intersection number between $\gamma$ and a generic straight ray $\rho$ emanating from $m$, oriented outward from $m$. We write this as 
\begin{equation} \label{rho:gamma}
   w_{\gamma}(m) = \rho \cdot \gamma.
\end{equation}
If $\gamma$ is smooth (or piecewise smooth) the generic ray $\rho$ intersects $\gamma$ transversely, so $\rho \cdot \gamma$ is defined by counting each intersection point with a sign depending on orientations. 

Let us now study the case $\gamma = \gamma_{\ell}$, where $\gamma_{\ell}$ is
defined in \eqref{gammaelleq}. With some abuse of notation we will denote by $\gamma_{\ell}$ its image in $M_{\R}$. We will say that an index $j \in {1, \ldots, r}$ corresponds to a {\it non-degenerate edge} of $\gamma_{\ell}$ if $\tilde \theta_{j+1} - \tilde \theta_{j} \neq 0$.  Choose some $\bar m \in M_{\R}$, $\bar m \neq 0$ and consider the ray 
\[ \rho_{\bar m}(t) = m + t \bar m, \quad t > 0. \]
If $\bar m$ is generic we can assume that every intersection point between $\rho_{\bar m}$ and $\gamma_{\ell}$ is in the relative interior of some non-degenerate edge of $\gamma_{\ell}$ . Moreover, we can also assume that 
\[ \inn{\bar m}{n_{\check e_j}} \neq 0, \quad \forall j \in \{1, \ldots, r \}. \]

\begin{defi} For every $p \in \rho_{\bar m} \cap \gamma_{\ell}$ and $j \in \{1, \ldots, r \}$, let $\delta(p,j)$ be defined as follows. If $p$ does not belong to the segment from $\tilde \theta_j$ to $\tilde \theta_{j+1}$, then $\delta(p,j)=0$.  If $p$ belongs to the segment from $\tilde \theta_j$ to $\tilde \theta_{j+1}$ then $\delta(p,j) =1$  (resp. $\delta(p,j) =-1$) if $\{ \bar m, \tilde \theta_{j+1} - \tilde \theta_j \}$ is a positively (resp. negatively) oriented basis of $M_{\R}$. Define
\[ \delta(p) = \sum_{j=1}^{r} \delta(p,j). \]
\end{defi}

Then formula (\ref{rho:gamma}) for $w_{\gamma}$ becomes
\[ w_{\gamma}(m) = \sum_{p \in \rho_{\bar m} \cap \gamma_{\ell}} \delta(p). \]

\begin{lem} \label{sign:inters} If $p \in \rho_{\bar m} \cap \gamma_{\ell}$ belongs to the edge from $\tilde \theta_j$ to $\tilde \theta_{j+1}$, then, for any $n$ in the interior of the cone $\nu_{j}$,
\[ \delta(p,j) = - \sgn( \inn{\tilde \theta_{j+1} - \tilde \theta_{j}}{ n} \inn{\bar m }{ n_{\check e_j}}). \]
\end{lem} 

\begin{proof}
Recall that $\check e_j$ is the common intersection between the cones $\nu_{j}$ and $\nu_{j+1}$. For any $n$ in the interior of $\nu_{j}$, we have that $\{ n, n_{\check e_{j}} \}$ is a positively oriented basis of $N_{\R}$, therefore $\{ \bar m , \tilde \theta_{j+1} - \tilde \theta_{j} \}$ is a positively oriented basis of $M_{\R}$ if 
\[ \inn{\tilde \theta_{j+1} - \tilde \theta_{j}}{ n_{\check e_{j}}} \inn{\bar m }{ n} - \inn{\tilde \theta_{j+1} - \tilde \theta_{j}}{ n} \inn{\bar m }{ n_{\check e_j}} > 0, \]
otherwise it is negatively oriented. On the other hand, by definition of the support function, we have
\[ \inn{\tilde \theta_{j+1} - \tilde \theta_{j}}{ n_{\check e_{j}}}  = 0. \]
Therefore the claim follows.
\end{proof}

Now let $S \subset N_{\R}$ be a circle centered at the origin and define semicircles 
\[ S^+ = S \cap \{ n \in N_{\R} \, | \, \inn{\bar m}{n} > 0 \}, \quad S^-= S \cap \{ n \in N_{\R} \, | \,  \inn{\bar m}{n} < 0\}. \]
Consider the function $T: S^- \cup S^+ \rightarrow \R$ given by 
\[ T(q) = \frac{\vartheta(q) - \inn{m}{q}}{\inn{\bar m}{q}}. \]
We now prove that the intersections between $\rho_{\bar m}$ and $\gamma_{\ell}$ are in correspondence with local maxima and minima of $T(q)$. 

We have the following:

\begin{lem} \label{min:edge}
For a generic choice of $\bar m$ we have that a point $q \in S^+ \cup S^-$ is a local maximum or minimum of $T$ if and only if $q \in \check e_j$ for some $j \in \{1, \ldots, r \}$ corresponding to some non-degenerate edge of $\gamma_{\ell}$ and $m + T(q) \bar m$ is on the edge joining $\tilde \theta_j$ and $\tilde \theta_{j+1}$.
\end{lem}
\begin{proof}
Let us first show that if $q$ is a local maximum or minimum of $T$, then $\vartheta$  is not linear in a neighborhood of $q$. By definition $T(q)$ satisfies 
\begin{equation} \label{T:def}
   \inn{m + T(q) \bar m}{q} - \vartheta(q) = 0.
\end{equation}
We assume that $q \in S^-$, the case $q \in S^+$ is analogous. In particular we have
\[ \inn{ \bar m}{q} < 0. \]
If $q$ is a local minimum of $T$ on $S^-$ then $q$ is also a local minimum of $T$ restricted to the line which is tangent to $S^-$ at $q$.  Write this line as $t \mapsto q+tq^{\perp}$, where $q^{\perp}$ is a tangent vector to $S^-$ at $q$ and $t \in \R$. 
If $\vartheta$ is linear in a neighborhood of $q$ then it is easy to see that for small values of $|t|$ we must have
\[ T(q+tq^{\perp}) = \frac{a + bt}{c+dt}. \]
for some real numbers $a,b,c,d$. From the genericity assumption of $\bar m$ we can also assume that this function is not constant. Then this function in $t$ does not have a local minimum in $t=0$. So $q$ cannot be a local minimum (or maximum) of $T$ on $S^-$. Hence we must have $q \in \check e_j$ for some $j$ such that $\tilde \theta_{j+1} - \tilde \theta_j \neq 0$.

%Now observe that if $\bar m$ is chosen generically, we can assume that $m + T(q) \bar m$ is not equal to $\tilde \theta_j$. Then they coincide on the line $\R q$ but they differ on both sides of it. In particular, let $H^+$ and $H^-$ be the two connected components of $N_{\R} - \R q$, so that $m+T(q) \bar m - \tilde \theta_j > 0$ on $H^+$ and $m+T(q) \bar m - \tilde \theta_j < 0$ on $H^-$. If $\vartheta$ coincides with $\tilde \theta_j$ in a neighbourhood $U$ of $q$, then for all $q' \in U \cap H^+$ and $q'' \in U \cap H^-$ we have
%\begin{equation} \label{notmax1}
%  \begin{split}
 %     \inn{m + T(q) \bar m}{q'} - \vartheta(q') & > 0 \\
 %      \inn{m + T(q) \bar m}{q''} - \vartheta(q'') & < 0.
 %  \end{split} 
%\end{equation}
%If $q',q'' \in S^-$, by \eqref{T:def}, we have
%\begin{equation} \label{notmax2}
   %\begin{split}
      % \inn{m + T(q') \bar m}{q'} - \vartheta(q')   & = 0, \\ 
      % \inn{m + T(q'') \bar m}{q''} - \vartheta(q'') & = 0.
   %\end{split}
%\end{equation}
%Taking the difference between the first inequality of (\ref{notmax1}) and the first equality of (\ref{notmax2}) we obtain
%\[(T(q)-T(q'))  \inn{\bar m}{q'} >  0\]
%Since $q' \in S^-$, this gives 
%\[ T(q) < T(q') \]
%Similarily if we take the difference between the second inequality of (\ref{notmax1}) and the second equality of (\ref{notmax2}), the same argument gives
%\[ T(q) > T(q''). \]
%This contradicts the fact that $q$ is a local maximum or minimum. 

We now prove that $m + T(q) \bar m$ is on the segment joining $\tilde \theta_j$ and $\tilde \theta_{j+1}$. 
We assume that $q$ is a local minimum, the local maximum case being similar. We have that for all $q'  \in \nu_{j+1} \cap S^-$
\[ T(q) \leq T(q') \]
and for all $q'' \in \nu_{j} \cap S^-$
\[ T(q) \leq T(q'') \]
hence
\begin{equation} \label{minq}
T(q) \inn{\bar m}{q'}  \geq T(q') \inn{\bar m}{q'}, \quad T(q) \inn{\bar m}{q''}  \geq T(q'') \inn{\bar m}{q''}.
\end{equation}
If we add $\inn{m}{q'}$ to both sides of the first inequality, we apply (\ref{T:def}) to $q'$  and we observe that $\vartheta(q') = \inn{\tilde \theta_{j+1}}{q'}$  we get
\begin{equation} \label{minimum}
       \inn{m + T(q) \bar m}{q'} \geq  \inn{\tilde \theta_{j+1}}{q'},
\end{equation}
Similarly from the second inequality of \eqref{minq} we get
\begin{equation} \label{minimum2}
       \inn{m + T(q) \bar m}{q''} \geq  \inn{\tilde \theta_{j}}{q''}.
\end{equation}
Equation (\ref{T:def}) implies that
\[ \inn{m + T(q) \bar m}{n_{\check e_j}} =  \inn{\tilde \theta_{j+1}}{n_{\check e_j}} = \inn{\tilde \theta_{j}}{n_{\check e_j}}. \]
This implies that $m + T(q) \bar m = s (\tilde \theta_{j+1}- \tilde \theta_j) + \tilde \theta_j$ for some $s \in \R$. Let us show that  $s \in [0,1]$.  Using $m + T(q) \bar m = s (\tilde \theta_{j+1}- \tilde \theta_j) + \tilde \theta_j$ in inequalities (\ref{minimum}) and \eqref{minimum2} we obtain
\begin{equation} \label{segment_ineq}
    (s-1) \inn{\tilde \theta_{j+1}- \tilde \theta_j}{q'} \geq 0 \quad \text{and} \quad s \inn{\tilde \theta_{j+1}- \tilde \theta_j}{q''} \geq 0. 
\end{equation}
Since $\inn{\tilde \theta_{j+1}- \tilde \theta_j}{n_{\check e_j}} = 0$ only the following two things can happen:
\begin{equation} \label{min3}
\inn{\tilde \theta_{j+1}- \tilde \theta_j}{q'} < 0 \quad \text{and} \quad \inn{\tilde \theta_{j+1}- \tilde \theta_j}{q''} > 0 
\end{equation}
or 
 \[\inn{\tilde \theta_{j+1}- \tilde \theta_j}{q'} > 0 \quad \text{and} \quad \inn{\tilde \theta_{j+1}- \tilde \theta_j}{q''} < 0. \]
The latter inequalities and \eqref{segment_ineq} cannot hold together, while the former inequalities and  \eqref{segment_ineq} imply that $s \in [0,1]$. We conclude that if  $q$ is a local minimum of $T(q)$, i.e. \eqref{segment_ineq} holds, then the former inequalities must be satisfied and $s \in [0,1]$. 
The converse is also true and we leave it to the reader.
\end{proof}

This lemma implies that if $p \in \rho_{\bar m} \cap \gamma_{\ell}$, then 
\[ p = m + T(q) \bar m \]
for some $q \in S^{+} \cup S^{-}$, with $T(q) > 0$ and $q$ is a local maximum or minimum.  Notice that this may happen for more than one value of $q$. In this case $p$ belongs to more than one edge, i.e., it is a multiple intersection. We have the following
\begin{lem} If $p \in \rho_{\bar m} \cap \gamma_{\ell}$ and $p = m + T(q) \bar m$ for some $q \in \check e_j$, then 
\[ \delta (p, j) = \begin{cases}
	1 &\ \text{if $q$ is a local minimum of $T$}, \\
      -1 &\ \text{if $q$ is a local maximum of $T$}.
       \end{cases}. \]
\end{lem}

\begin{proof} Let us assume first that $q \in S^-$ and that $q$ is a local minimum. In particular 
\begin{equation} \label{s-}
 \inn{ \bar m}{n_{\check e_j}} < 0 
\end{equation}
since $n_{\check e_j}$ is a positive multiple of $q$. Moreover, since $q$ is a local minimum, 
we saw in the proof of Lemma \ref{min:edge} that  for any $q'' \in S^- \cap \nu_{j}$ the second inequality in \eqref{min3} holds. Hence Lemma \ref{sign:inters}, applied with $n = q''$ gives $\delta(p,j) =1$. 

Similarly we can reason in the cases where $q \in S^-$, but it is a maximum;  $q'' \in S^+$ and it is a minimum; $q'' \in S^+$ and it is a maximum.
\end{proof}

\begin{cor} \label{win_min_max}
Let $\operatorname{Max}_{T \geq 0}$ be the set of local maxima $q$ of $T$ such that $T(q) > 0$ and let $\operatorname{Min}_{T \geq 0}$ be the set of local minima $q$ of $T$ such that $T(q) > 0$, then 
\begin{equation} \label{min_max_eq}
    w_{\gamma_{\ell}}(m) =\# \operatorname{Min}_{T \geq 0} \ - \ \# \operatorname{Max}_{T \geq 0} .
\end{equation}
\end{cor}

\begin{proof}[Proof of Theorem \ref{coho_winding}] Let us first prove that
\begin{equation} \label{winding_one}
        m \in  Q_0 \Longrightarrow w_{\gamma_{\ell}}(m) = 1.
\end{equation}
It follows from Lemma \ref{ZW} that $m \in Q_0$ if and only if 
\[ \inn{m}{n} - \vartheta(n) > 0, \quad \forall  n \in N_{\R}- \{ 0 \}.\]
Therefore if $m \in Q_0$, then $T|_{S^-} > 0$ and $T|_{S^+} < 0$. We then have that all points of $\operatorname{Max}_{T \geq 0}$ and $\operatorname{Min}_{T \geq 0}$ must be in $S^-$. Now let $q_{-\infty}$ and $q_{+\infty}$ be the two boundary points of $S^{-}$ (and of $S^+$). We have that 
\[ \lim_{q \rightarrow q_{+ \infty}} T(q) = \lim_{q \rightarrow q_{- \infty}} T(q) = + \infty \]
where the limits are taken for $q \in S^-$.  Then (see Figure \ref{graph_T}), it can be easily seen that 
\[ \# \operatorname{Min}_{T \geq 0} = \# \operatorname{Max}_{T \geq 0} + 1. \]

\begin{figure}[!ht] 
\begin{center}
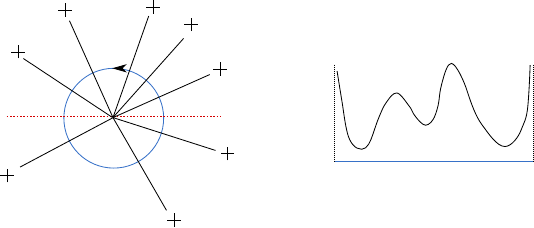
\caption{The fan and the graph of $T$ in the case $m \in Q_0$.} \label{graph_T}
\end{center}
\end{figure}
Now (\ref{winding_one}) follows from Corollary \ref{win_min_max}. 
Similarly one can show 
\begin{equation} \label{winding_one2}
        m \in  Q_2 \Longrightarrow w_{\gamma_{\ell}}(m) = 1.
\end{equation}
In fact in this case $m \in Q_2$ if and only if 
\[ \inn{m}{n} - \vartheta(n) < 0, \quad \forall  n \in N_{\R}- \{ 0 \}. \]
Then the same proof holds replacing $S^-$ with $S^+$.

We now prove that 
\begin{equation} \label{winding_neg}
        m \in  Q_1 \Longrightarrow w_{\gamma_{\ell}}(m) = 1- b_0(N_{\R} - W(m)).
\end{equation}
If $m \in Q_1$, the set $W(m)$ is a union of a finite number of closed angular sectors (which may also be non-convex). We only consider maximal sectors, i.e., those which are not contained in strictly larger sectors which are also contained in $W(m)$. We call these sectors ``white sectors''.
The complement $N_{\R} - W(m)$ is also a union of open angular sectors, i.e., its connected components. We call these sectors ``black sectors''. Clearly $b_0(N_{\R} - W(m))$ is either the number of white sectors or of black sectors. We call the intersection of a white sector (resp. black sector) with $S^+$ or $S^-$ a white arc (resp. black arc). It can be easily seen that $T$ is non-negative on white arcs contained in $S^-$ or on black arcs contained in $S^+$. Vice versa $T$ is negative on black arcs contained in $S^-$ or on white arcs contained in $S^+$. It is clear that an arc in $S^-$, either black or white, can be of three types: i) it coincides with $S^-$; ii) its boundary consists of a point in $S^-$ and one of the points $q_{\pm \infty}$; iii) both its boundary points are in $S^-$. Similarly we classify arcs in $S^+$.  The number of local minima of $T$ inside a white arc in $S^-$ of type (i) is one plus the number of local maxima inside the same arc. Therefore, from formula (\ref{min_max_eq}), a white arc in $S^-$ of type (i) contributes $1$ to the computation of $w_{\gamma_{\ell}}(m)$. Similarly we can say the same of a black arc in $S^+$ of type (i). 
On the other hand, the number of local minima of $T$ inside a white arc in $S^-$ of type (ii) is the same as the number of local maxima inside the same arc (see Figure \ref{graph_T_2}). 
So that a white arc in $S^-$ of type (ii) does not contribute at all to the computation of $w_{\gamma_{\ell}}(m)$. Similarly we can say the same of a black arc in $S^+$ of type (ii). 
Finally, the number of local minima of $T$ inside a white arc in $S^-$ of type (iii) is one less than the number of local maxima inside the same arc (see Figure \ref{graph_T_2}). Similarly we can say of a black arc in $S^+$ of type (iii). Therefore a white arc in $S^-$, or black arc in $S^+$, of type (iii) contributes $-1$ to the computation of $w_{\gamma_{\ell}}(m)$. 

\begin{figure}[!ht] 
\begin{center}
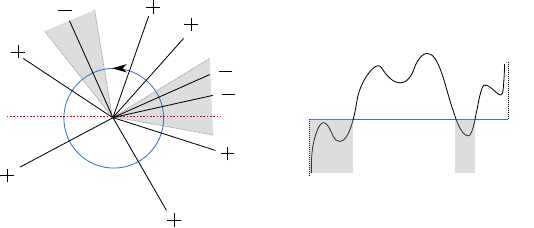
\caption{The fan and the graph of $T$ when $m \in Q_1$. Black sectors are shaded. In this case $S^-$ has one white arc of type (iii) and one of type (ii). } \label{graph_T_2}
\end{center}
\end{figure}
The final step of the proof is just a simple combinatorial problem. Let $w^+$ and $w^-$ be the number of white arcs of type (iii) respectively in $S^+$ and in $S^-$. Similarly let $b^+$ and $b^-$ be the number of black arcs of type (iii) respectively in $S^+$ and $S^-$. If $S^-$ consists of one white arc of type (i), then all black arcs in $S^+$ are of type (iii) and, from the previous considerations, we have
\[ w_{\gamma_{\ell}}(m) = 1-b^+. \]
It can be seen that $b^+ = w^++1$ and that
\[ b_0(N_{\R} - W(m)) = 1 + w^+ = b^+, \]
which implies (\ref{winding_neg}). Similarly one deals with the case when $S^+$ consists of one black arc. 

Now suppose all white arcs of $S^-$ are of type (iii) and $w^+ \geq 1$. Then we have 
\begin{align*} w_{\gamma_{\ell} }(m)&= - w^- - b^+, \\
 b^{0}(N_{\R} - W(m)) &= b^- + b^+ +2.
\end{align*}
Moreover $w^- = b^- + 1$. Therefore we obtain (\ref{winding_neg}). The case where all black arcs of $S^+$ are of type (iii) and $b^- \geq 1$ is similar. 

We do the case when $S^-$ has one white arc of type (ii) and one black arc of type (ii). In this case the same is true for $S^+$. So we have 
\begin{align*} w_{\gamma_{\ell} }&= - w^- - b^+, \\
 b^{0}(N_{\R} - W(m)) &= b^- + b^+ +1.
\end{align*}
Moreover $b^- = w^-$, which gives (\ref{winding_neg}). The last cases to consider is when $S^-$ does not contain white arcs or $S^+$ does not contain black arcs. We leave these to the reader. 

In particular this shows that if $m \in Q_1$, then either $w_{\gamma_{\ell}}(m)$ is zero or it is negative and equal to $- h_{\psi}(m)$. Therefore $w_{\gamma_{\ell}}(m)$ is positive if and only if $m \in Q_0$ or $m \in Q_2$, in which case $w_{\gamma_{\ell}}(m) = 1$. The conclusion of the theorem now follows from (\ref{coho_L}). 
\end{proof}

\begin{ex}
Let $\check X$ be the total space of the canonical bundle of the one-point blowup of $\PP^2$. It is obtained from the polytope $$P = \conv \{ (0,1), (-1,1), (-1,0), (1,-1) \}$$ with the subdivision whose interior edges connect each vertex of $P$ with the unique interior integral point $(0,0)$. Then $M_{\R} - \Gamma$ has just one bounded component $C$ corresponding to $v_C = (0,0)$. Clearly $D_C$ is the one point blowup of $\PP^2$.  The fan $\Sigma_C$ has one-dimensional cones generated respectively by $\check n_1 = (0,1)$, $\check n_2=(-1,1)$, $\check n_3= (-1,0)$, $\check n_4= (1,-1)$. Then consider the twisting numbers 
\[ \ell = (-14, 5,-14,-9). \]
The curve $\gamma_{\ell}$ is pictured in Figure \ref{winding_example}. There are $10$ integral points with winding number $1$ and $3$ points with winding number $-1$. We have 
\[ K_C = (-2,-1, -2, -3), \]
and
\[ \frac{K_C - \ell}{2} = (6, -3, 6, 3). \]

\begin{figure}[!ht] 
\begin{center}
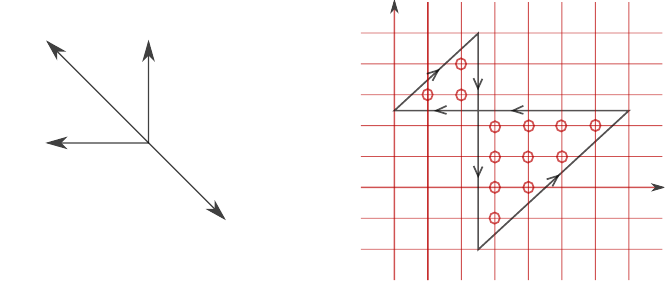\\
\vspace{1ex}\refstepcounter{figure}Figure \arabic{figure}. \label{winding_example}
\end{center}
\end{figure}
According to Conjecture \ref{hms:spheres}, these four numbers are the intersection numbers of the sheaf $\mathscr E_{L_{\ell}}$ corresponding to $L_{\ell}$ with
the curves on the toric boundary of $D_C$. If we denote by $H$ the hyperplane section in $D_C$ and by $E$ the exceptional curve of the blowup, one can see that 
\[ \mathscr E_{L_{\ell}} = \mathcal{O}_{D_C}(3H+3E). \]
We have 
\[ H^{0}(D_C, \mathscr E_{L_{\ell}}) = \C^{10}, \quad H^{1}( D_C, \mathscr E_{L_{\ell}}) = \C^3, \quad H^{2}( D_C, \mathscr E_{L_{\ell}}) = 0.\]
\end{ex}

\section{Spherical objects and $A_m$-configurations} \label{spherical:an}
In this section we will use the results of the previous sections to find sheaves which, according to our Conjecture \ref{hms:spheres}, are mirror to the $A_{2d-1}$-configuration of vanishing cycles in a smoothing of a $A_{2d-1}$-singularity. This refines a conjecture of Seidel and Thomas \cite{ST}.

\subsection{Definitions and examples} When $L$ is an embedded Lagrangian sphere of dimension $n \geq 2$, the Floer homology $HF^*(L,L) \otimes_{\R} \C$ is isomorphic to the standard cohomology of the sphere with $\C$ coefficients, i.e., it is $\C$ in degrees $0$ and $n$ and $0$ elsewhere. So morphisms of the mirror object $\mathcal E_L$  should satisfy the same property. This justifies the following definition of Seidel and Thomas \cite{ST}:

\begin{defi}
Given a Calabi-Yau manifold $X$, an object $\mathcal E \in \DC(X)$ is \textit{spherical} if $\Hom^r_{\DC(X)}(\mathcal E, \mathcal E)$ is $\C$ when $r=0$ or $n$, and zero in all other degrees. 
\end{defi}

Now suppose we have a chain of embedded Lagrangian spheres $L_1, \ldots, L_n$ such that $L_i$ intersects $L_{j}$ transversely in one point if $|i-j| = 1$ and $L_i \cap L_j = \emptyset$ if $|i-j| > 1$. Then we have that $HF^*(L_i,L_j) \otimes_{\R} \C$
 is $\C$ when $|i-j| = 1$ and zero when $|i-j| > 1$. Such a configuration of spheres is called an $A_m$-configuration. This justifies the following:

\begin{defi}
Given a Calabi-Yau manifold $X$ and $m \geq 1$, an $A_m$ configuration in $\DC(X)$ is a collection of spherical objects $\mathcal E_1, \dots, \mathcal E_m$ such that
 \[ \dim_{\C} \Hom^*_{\DC(X)}(\mathcal E_i, \mathcal E_j) = \begin{cases} 
     1 \quad |i-j| = 1, \\
     0 \quad |i-j| > 1.
    \end{cases}
\] 
\end{defi}

%Seidel and Thomas prove that given an $A_m$ configuration in $\DC(X)$, the twist functors associated to the objects $\mathcal E_1, \dots, \mathcal E_m$ satisfy braid group relations and they generate a subroup of the group of autoequivalences of $\DC(X)$ isomorphic to the braid group in $m$ generators. Moreover Seidel \cite{seidel:dehn} (maybe also \cite{seidel:twists}?) \todo{look for reference, the given one proves it in dim=2} showed that also the generalized Dehn twists associated to an $A_m$-configuration of Lagrangian spheres satisfy braid group relations (up to isotopy) inside $\Symp^{gr}(\check X)$. They conjecture that their representatives inside $\pi_0(\Symp^{gr}(\check X))$ will generate a subgroup isomorphic to the braid group with $m$-generators. 

Let $X$ be an $n$-dimensional Calabi-Yau manifold. Assuming that $X$ is compact, a simple class of spherical objects in $\DC(X)$ is given by line bundles.  More generally consider a submanifold $Y \subset X$ and sheaves supported on $Y$. We denote by $\iota: Y \hookrightarrow X$ the inclusion. If $\mathcal L$ is a sheaf on $Y$, by slight abuse of notation we will denote by $\mathcal L$ also the sheaf on $X$ given by $\iota_{\ast} \mathcal L$. We have the following (see Lemma 3.7 of \cite{ST} for a proof): 
\begin{lem} \label{y:spher}
Let $X$ be an $n$-dimensional quasi-projective smooth variety and $Y \subseteq X$ a connected compact submanifold of codimension $r$, whose normal bundle $\nu$ satisfies $H^{i}(Y, \bigwedge^j \nu) = 0$ when $0< i+j < n$. Then $\mathcal O_Y \in  \DC(X)$ is spherical. 
\end{lem}

%\begin{proof}
%Assume $X$ is projective. Using the local Koszul resolution of $\mathcal O_Y$, we have that 
%$\ext_{X}^{i}(\mathcal O_Y, \mathcal O_Y) \cong \Lambda^i \nu$. Now the $\Ext^k$ groups are computed as the hypercohomology of the spectral sequence whose $E^2$ page is given by the groups $H^{j}(X, \ext^i(\mathcal O_Y, \mathcal O_Y)) = H^{j}(Y, \Lambda^i \nu )$. By the hypothesis these are zero when $0 < i+j <n$. Then we have $H^{0}(X, \ext^0(\mathcal O_Y, \mathcal O_Y)) = H^0(Y, \mathcal O_Y) = \C$. The remaining case is 
%$H^{n-r}(Y, \Lambda^r \nu )$. Since $X$ is Calabi-Yau, we have $\Lambda^r \nu = \omega_Y$. Therefore 
%$H^{n-r}(Y, \Lambda^r \nu ) = \C$ (by Serre duality). This shows that $E^2 = E^{\infty}$ and the Lemma holds. The case when $X$ is quasi-projective can be achieved by taking the completion $\bar X$. 
%[DIEGO: I don't think we need to worry about this compactification because
%we are only applying Serre duality to the compact $Y$.]
%Then we have $Y \subset \bar X$ and we can proceed as before. Although $\bar X$ may not be Calabi-Yau, we still have $\omega_{X}|_{Y} = \mathcal O_Y$. 
%\end{proof}

More generally we have

\begin{cor} \label{spherical_ondiv}
With the same hypothesis of Lemma~\ref{y:spher}, let $\mathcal L$ be a line bundle on $X$ and $\mathcal L_Y = \mathcal L \otimes \mathcal O_Y$. Then $\mathcal L_Y \in \DC(X)$ is spherical.
\end{cor}
\begin{proof}
If follows from the fact that $\Ext_{X}^k(\mathcal L_Y, \mathcal L_Y) = \Ext_{X}^k(\mathcal O_Y, \mathcal O_Y)$ (see \cite{harshorne}, Proposition 6.7, pg. 235). 
\end{proof}

Observe that when $Y$ is a hypersurface in a Calabi-Yau manifold $X$, 
then the condition on the cohomology of the normal bundle of $Y$ becomes 
\begin{equation} \label{nu:hs}
  H^i(Y, \mathcal O_Y) = 0, \, i=1, \ldots, n-1. 
\end{equation}
Indeed, this is clearly the necessary condition for $j=0$, and for $j=1$,
one needs $H^i(Y,\omega_Y)=0$ for $0 < i+1 < n$, or by Serre duality
$H^{n-1-i}(Y,\mathcal O_Y)=0$ for $0<i+1<n$, which is equivalent to the
above condition.
In the case $X$ is a $3$-fold, these conditions hold when $Y \subset X$ is a 
rational surface. 

In the case of a rational curve $C$ in a $3$-fold $X$, then we must have that 
\begin{equation} \label{nu:curve}
 \nu = \mathcal O_C(-1) \oplus \mathcal O_C(-1).
\end{equation} 
These are called $(-1,-1)$-curves.

Let us now give some examples of $A_m$-configurations.

\begin{lem} \label{inters:point} Let $X$ be a quasi-projective Calabi-Yau $3$-fold, $Y$ a smooth, compact, embedded surface in $X$ satisfying (\ref{nu:hs}) and $C$ an embedded rational curve satisfying (\ref{nu:curve}). Assume that $C$ and $Y$ intersect transversely at a point. Given a line bundle $\mathcal L$ on $X$, let $\mathcal E_1 = \mathcal L_{Y}$ and $\mathcal E_2 = \mathcal O_{C}$. Then $\mathcal E_1$ and $\mathcal E_2$
form an $A_2$-configuration of objects in $\DC(X)$.
\end{lem}

This is a simple calculation which we leave to the reader. Another possible configuration is the following:

\begin{lem} \label{curv_in_surf} Let $X$ be a quasi-projective Calabi-Yau $3$-fold, $Y$ a smooth, compact, embedded surface in $X$ satisfying (\ref{nu:hs}) and $C$ an embedded rational curve satisfying (\ref{nu:curve}). Assume that $C$ is contained in $Y$ and that it has self-intersection $-1$ in $Y$. Given a line bundle $\mathcal L$ on $X$, let $\mathcal E_1 = \mathcal L_{Y}$ and $\mathcal E_2 = \mathcal O_{C}$. Then $\mathcal E_1$ and $\mathcal E_2$
form an $A_2$-configuration of objects in $\DC(X)$ if and only if 
\[ \mathcal L|_{C} = \mathcal O_{C} \ \ \text{or} \ \ \mathcal{O}_{C}(1). \]
 \end{lem}

\begin{proof}
Consider the standard locally free resolution of $\mathcal O_{Y}$
\begin{equation} \label{a2:res}
 0 \longrightarrow \mathcal{O}_X(-Y) \stackrel{\otimes s}{\longrightarrow} \mathcal O_X \longrightarrow \mathcal{O}_Y \rightarrow 0,
\end{equation}
where $s$ is a section of $\mathcal{O}_X(Y)$ vanishing along $Y$. Now apply $\ch(\cdot, \mathcal{O}_C)$ to the sequence
and we obtain the complex calculating $\ext^*(\mathcal{O}_Y,\mathcal{O}_C)$:
\begin{equation} \label{a2:seq}
  0 \rightarrow \ch(\mathcal{O}_X, \mathcal{O}_C ) \longrightarrow  \ch(\mathcal{O}_X(-Y), \mathcal{O}_C ). 
\end{equation}
Now 
\[ \ch(\mathcal{O}_X, \mathcal{O}_C ) = \mathcal{O}_C  \]
and 
\begin{align*}
  \ch(\mathcal{O}_X(-Y), \mathcal{O}_C ) & =  \mathcal{O}_X(Y)|_{C} =   \omega_Y |_{C}  \\ 
  & =  \nu_{C|Y}^{-1}  \otimes \omega_C = \mathcal{O}_C(-1).  
\end{align*}
where $\nu_{C|Y}$ denotes the normal bundle of $C$ inside $Y$.
Since there is no non-trivial map from $\mathcal{O}_C$ to $\mathcal{O}_C(-1)$, the second map in (\ref{a2:seq}) is zero, 
hence the only non-trivial $\ext$ sheaves are 
\[ \ext^0(\mathcal{O}_Y, \mathcal{O}_C) = \mathcal{O}_C, \ \ \ \ \ext^1(\mathcal{O}_Y, \mathcal{O}_C) = \mathcal{O}_C(-1).  \]
Hence the $E_2$ page of the local-global spectral sequence computing $\Ext$ groups 
(i.e., $E_2^{i,j} = H^i(X, \ext^j(\mathcal{O}_Y, \mathcal{O}_C)$) is non-zero at $E_2^{0,0} = \C$. 
Hence $\mathcal E_1 = \mathcal{O}_Y$ and $\mathcal E_2 = \mathcal{O}_C$ form an $A_2$-configuration since $\Ext^0(\mathcal{O}_Y, \mathcal{O}_C) = \C$ and $\Ext^r(\mathcal{O}_Y, \mathcal{O}_C) = 0$ when $r >0$. 
In the case $\mathcal E_1 = \mathcal L_Y$, then a resolution of $\mathcal L_Y$ is obtained by tensoring (\ref{a2:res}) by $\mathcal L$. 
Then the same calculation leads to the following non-trivial $\ext$ sheaves
\[ \begin{split}
            \ext^0(\mathcal L, \mathcal{O}_C) & = \mathcal L^{-1} \otimes \mathcal{O}_C = \mathcal{O}_C(-k), \\
            \ext^1(\mathcal L, \mathcal{O}_C) &= \mathcal L^{-1} \otimes \mathcal{O}_C(-1) = \mathcal{O}_C(-1-k) 
   \end{split} \]
using $\mathcal L|_C=\mathcal{O}_C(k)$.
The spectral sequence then gives that $$\Hom^*_{\DC(X)}(\mathcal L_Y, \mathcal{O}_Y)$$
is only one-dimensional in the cases $k=0,1$. 
\end{proof}

We will also need the following:

\begin{lem} \label{inters_curv}
Let $X$ be a quasi-projective Calabi-Yau $3$-fold, $Y_1$ and $Y_2$ a pair of smooth, connected, compact, embedded surfaces in $X$. Assume that $Y_1$ and $Y_2$ intersect transversely along an embedded rational curve $C$ and let $m$ be the self-intersection number of $C$ inside $Y_2$. Given a line bundle $\mathcal L$ on $X$, suppose that 
\[ \mathcal L|_{C} = \mathcal O_{C}(k). \]
Then $\dim (\Hom^*_{\DC(X)}(\mathcal{O}_{Y_1}, \mathcal L_{Y_2})) = 0$ if and only if 
\[ k + m = -1. \]
\end{lem}

\begin{proof}
Consider the sequence (\ref{a2:res}) applied to $Y_1$ and take $\ch( \cdot, \mathcal L_{Y_2})$. We have 
the complex 
\begin{equation*} 
  0 \rightarrow \ch(\mathcal{O}_X, \mathcal L_{Y_2} ) \longrightarrow  \ch(\mathcal{O}_X(-Y_1), \mathcal L_{Y_2} ) 
\end{equation*}
which becomes 
\begin{equation*} 
  0 \rightarrow \mathcal L_{Y_2} \longrightarrow \mathcal{O}_X(Y_1) \otimes \mathcal L_{Y_2}  
\end{equation*}
where the first map is tensoring with a section $s$ of $\mathcal{O}_X(Y_1)$ vanishing on $Y_1$. In particular the first map is injective. So the cohomology of the complex gives
\[ \ext^0(\mathcal{O}_{Y_1}, \mathcal L_{Y_2}) = 0 \] 
and 
\[ \ext^1(\mathcal{O}_{Y_1}, \mathcal L_{Y_2}) = \nu_{C|Y_2} \otimes \mathcal L|_{C} = \mathcal{O}_{C}(k+m). \]
Then, $\dim (\Hom^*_{\DC(X)}(\mathcal{O}_{Y_1}, \mathcal L_{Y_2})) = 0$ if and only if $\mathcal{O}_{C}(k+m)$ has no cohomology, i.e., if and only if $k+m =  -1$.
\end{proof}

In particular, let $Y_1, Y_2$ and $C= Y_1 \cap Y_2$ be as in the above lemma,  such that $Y_1$ and $Y_2$ satisfy (\ref{nu:hs}) and $C$ satisfies (\ref{nu:curve}). If $C$ has self-intersection $-1$ inside $Y_2$,  then $\mathcal E_1 = \mathcal{O}_{Y_1}$, $\mathcal E_2 = \mathcal{O}_C$ and $\mathcal E_3 = \mathcal{O}_{Y_2}$ form an $A_3$-configuration.

There are other possibilities, for instance suppose that $Y_1$, $Y_2$, $\mathcal L$ and $C$ are as in the Lemma above, with $k+m = -1$. Suppose $C'$ is another embedded rational curve contained as a $-1$ curve inside $Y_2$ and intersecting $Y_1$ transversely in one point. If $Y_1$ and $Y_2$ satisfy (\ref{nu:hs}) and $C'$ satisfies (\ref{nu:curve}), then $\mathcal E_1 = \mathcal{O}_{Y_1}$, $\mathcal E_2 = \mathcal L|_{C'}$ and $\mathcal E_3 = \mathcal L_{Y_2}$ also form an $A_3$-configuration.

\subsection{$A_{2d-1}$-singularities} 
$A_m$-configurations of Lagrangian spheres appear naturally as vanishing cycles of certain singularities. In fact consider the singularity defined by the equation
\begin{equation} \label{a-sing}
 x^2 + y^2 + u^2 + v^{2d} = 0.
\end{equation}
This singularity is said to be of type $A_{2d-1}$. If $X$ is the smoothing of this singularity given by equation
\begin{equation} \label{a-sing-smooth}
 x^2 + y^2 + u^2 + v^{2d} = \epsilon.
\end{equation}
then $X$ contains $2d-1$ Lagrangian spheres $L_1, \ldots, L_{2d-1}$ representing vanishing cycles in $H_{3}(X)$ and forming an $A_{2d-1}$-configuration of Lagrangian spheres. 
Seidel and Thomas propose a possible mirror manifold $\check X$ and make a guess at what the corresponding $A_{2d-1}$ configuration of objects in $\check X$ might look like. They propose that $\check X$ is a Calabi-Yau manifold which should contain embedded smooth surfaces $S_{2}, S_{4}, \ldots, S_{2d-2}$ and curves $C_{1}, C_{3}, \ldots,\linebreak C_{2d-1}$ such that the following holds
\begin{itemize}
\item[1)] each $S_{2i}$ is isomorphic to $\PP^2$ with two points blown up;
\item[2)] $S_{2i} \cap S_{2j} = \emptyset$ if $|i-j| >1$; 
\item[3)] $S_{2i-2}$ and $S_{2i}$ are transverse and intersect in $C_{2i-1}$, which is a rational curve and has self-intersection $-1$ both in $S_{2i-2}$ and in $S_{2i}$.
\end{itemize} 
Let $\mathcal E_{2j} = \mathcal{O}_{S_{2j}}$ and $\mathcal E_{2i-1} = \mathcal{O}_{C_{2i-1}}$. Then it follows from the discussion in the previous section that $\mathcal E_{1}, \ldots, \mathcal E_{2d-1}$ do indeed form an $A_{2d-1}$-configuration of objects in $\DC(\check X)$. 

\subsection{Mirror $A_{2d-1}$ configurations} \label{mirroran}
Here we use our results to make a different and more precise proposal than the one by Seidel and Thomas described above. First we describe the vanishing cycles using the constructions of Sections \ref{vc:d3} and \ref{vc:edges}, then we use Conjecture \ref{hms:spheres} to find their mirror objects in $\DC(\check X)$ and we prove that they form an $A_{2d-1}$ configuration. As in Example \ref{a2d:mirror},  let $X$ be given by equation (\ref{a2d:smoothing}). We observed that $X$ is a smoothing of the singularity $xy=z^2+w^{2d}$, which is equivalent to the one in \eqref{a-sing} by a simple change of coordinates. Here we choose a subdivision of $P$ such that every compact toric divisor of $\check X$ is a one point blow-up of a Hirzebruch surface (see Figure \ref{a2d:trop} for the case $d=3$). The fibration $f: X \rightarrow \R^3$ has as discriminant locus a fattening of the tropical curve in Figure \ref{a2d:trop} (for the case $d=3$). The Lagrangian spheres will be of two types. First we will have $d-1$ spheres which are of the type constructed over the $d-1$ bounded regions $C_1, \ldots, C_{d-1}$ of the complement of $\Gamma$ (as in \S \ref{vc:d3}). Let us denote these by $M_1, \ldots, M_{d-1}$ (see next paragraph for a precise definition).  For  $j=1, \ldots, d$, consider the edge, in the subdivision of $P$, which connects the point $(j-1,1)$ to the point $(j,0)$. The dual to this edge in $\Gamma$ is an edge over which we can construct spheres as in \S \ref{vc:edges}. Over every such edge we take the sphere constructed with $k=0$. Let us denote these spheres by $N_1, \ldots, N_d$.  
The divisor $D_j$ in $\check X$ corresponding to $C_j$ is isomorphic to the one point blowup of a Hirzebruch surface of type $j$.  Let us label and order the edges of $C_j$ as in Figure \ref{sections_a2d}. According to this ordering we have 
\[ K_{C_j}=(j-2, -1,-1,-j-1, -2). \]

Recall that a sphere over $C_j$ is determined by numbers $\ell_j=(\ell_{j1}, \ldots, \ell_{j5})$ whose entries have the same parity as the corresponding entries in $K_{C_j}$.

Let us consider spheres with minimal twisting numbers, i.e., such that the entries of $\ell_j$ are either $0$ or $\pm 1$. So we can define
\[ \ell_j = \begin{cases}
                    (-1,1,-1,0,0) \quad \text{when} \ \ j \ \ \text{is odd}, \\
                    (0,-1,1,-1,0) \quad \text{when} \ \ j \ \ \text{is even},
            \end{cases} \]
see Example \ref{a2d:mirrorredux}.
It can be easily checked that the above choices satisfy (\ref{vc:balance}). We define $M_j = L_{\ell_j}$. It is not hard to see that $M_j  \cap M_k = \emptyset$ when $j \neq k$.  This is clear if $|k-j| \geq 2$. In the case $k=j+1$, this is because the only possible place where they could intersect is along the common edge of $C_j$ and $C_{j+1}$, but along this edge the twisting numbers are the same for both spheres. This implies that along this edge the image of the maps $\lambda_j$ and $\lambda_{j+1}$, defining the spheres over $C_j$ and $C_{j+1}$ respectively, can run parallel to each other without intersecting.  

We now use Conjecture \ref{hms:spheres} to construct the sheaves which are mirror to the spheres $M_j$. 
If we let
\[
\kappa_j = \frac{ K_{C_j}- \ell_j}{2}, 
\]
then formula \eqref{hms:vanishing3} tells us that 
\[
 \mathscr E_{L_{\ell_j}}  = \mathscr L_{\kappa_j}|_{D_j}. 
\]
We have that
\begin{equation} \label{bundles:vc3}
 \kappa_j = \begin{cases}
                    (\frac{j-1}{2}, -1,0, -\frac{j+1}{2} ,-1) \quad \text{when} \ \ j \ \ \text{is odd}, \\
                    (\frac{j-2}{2}, 0, -1, -\frac{j}{2} ,-1) \quad \text{when} \ \ j \ \ \text{is even}.
            \end{cases} 
\end{equation}

It is not hard to see that $\#(N_j \cap M_j) = \#(N_{j+1} \cap M_j) = 1$. Figure \ref{an_sing_inters} shows how these spheres intersect. It is clear that $\#(N_j \cap M_k) = 0$ when $j-k \neq 0,1$ and that $\#(N_j \cap N_k)= 0$ when $j\neq k$.  The fact that $\#(M_j \cap M_k) = 0$ when $j \neq k$ is explained above. Therefore the collection of spheres $(N_1, M_1, \ldots, M_{d-1}, N_{d})$ gives an $A_{2d-1}$-configuration of objects of $\DF(X)$. One can also show that the Lagrangian spheres can be constructed so that the intersections are transversal. 

\begin{figure}[!ht] 
\begin{center}
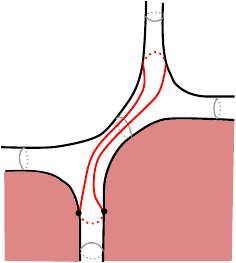
\caption{The picture shows the intersection of $N_j$ with the surface $S$ and the spheres $M_j$ and $M_{j+1}$ (shaded areas). The intersection points of $N_j$ with $M_j$ and $M_{j+1}$ are the two black dots.} \label{an_sing_inters}
\end{center}
\end{figure}

Now observe that the edge in the subdivision of $P$ connecting $(j-1,1)$ to $(j,0)$ corresponds to a rational curve in $\check X$ satisfying \eqref{nu:curve}. Denote this curve by $\PP^1_{j}$.  As mirror to the spheres $N_j$, let us propose the sheaf
\[ \mathscr E_{N_{j}} = \mathcal{O}_{\PP^1_j}(-1). \]
We have the following

\begin{prop} \label{a2d:sequence}
The collection of sheaves $(\mathscr E_{N_{1}}, \mathscr E_{M_{1}}, \ldots, \mathscr E_{M_{d-1}}, \mathscr E_{E_{d}} )$ defines an $A_{2d-1}$-configuration of objects of $\DC(\check X)$.
\end{prop}

\begin{proof}
Since the divisors $D_j$ are one-point blow-up of Hirzebruch surfaces, they all satisfy (\ref{nu:hs}). As we said, the curves $\PP^1_j$ all satisfy (\ref{nu:curve}). Therefore Lemma \ref{y:spher} and Corollary \ref{spherical_ondiv} imply that the sheaves $\mathscr E_{M_{j}}$ and $\mathscr E_{N_{j}}$ are spherical objects. Observe that $\PP^1_j$ and $D_j$ intersect transversely at one point, therefore Lemma \ref{inters:point} implies that, for every $j$, $\mathscr E_{N_{j}}$ and $\mathscr E_{M_{j}}$ form an $A_2$-configuration of objects. The curve $\PP^1_{j+1}$ is contained in $D_j$ and has self-intersection $-1$ inside $D_j$. Observe that $\PP^1_{j+1}$ corresponds to the edge $e_{j3}$ of Figure \ref{sections_a2d}. This implies that $\mathscr E_{M_{j}}|_{\PP^1_{j+1}}$ is determined by the third entry of $\kappa_j$ in formula (\ref{bundles:vc3}), i.e., we have
\[ \mathscr E_{M_{j}}|_{\PP^1_{j+1}} = \mathscr L_{\kappa_j}|_{\PP^1_{j+1}} =  \begin{cases}
                                                                  \mathcal{O}_{\PP^1_{j+1}} \quad \text{when} \ \ j \ \ \text{is odd}, \\
                                                                  \mathcal{O}_{\PP^1_{j+1}}(-1) \quad \text{when} \ \ j \ \ \text{is even}.
                                                                  \end{cases} \]
Now let $\mathscr L$ be any line bundle on $\check X$ such that $\mathscr L|_{\PP^1_{j+1}} = \mathcal{O}_{\PP^1_{j+1}}(1)$. We have
\begin{align*}
\Hom^*_{\DC(\check X)}(\mathscr E_{M_j}, \mathscr E_{N_{j+1}}) & =  \Hom^*_{\DC(\check X)}(\mathscr E_{M_j}, \mathcal{O}_{\PP^1_{j+1}}(-1))  \\
    	  & =   \Hom^*_{\DC(\check X)}(\mathscr E_{M_j} \otimes \mathscr L, \mathcal{O}_{\PP^1_{j+1}}(-1) \otimes \mathscr L)  \\
	  & =  \Hom^*_{\DC(\check X)}(\mathscr E_{M_j} \otimes \mathscr L, \mathcal{O}_{\PP^1_{j+1}}).
\end{align*}
Now observe that 
\[ \mathscr E_{M_j} \otimes \mathscr L|_{\PP^1_{j+1}}= \begin{cases}
                                                                  \mathcal{O}_{\PP^1_{j+1}}(1) \quad \text{when} \ \ j \ \ \text{is odd}, \\
                                                                  \mathcal{O}_{\PP^1_{j+1}} \quad \text{when} \ \ j \ \ \text{is even}.
                                                                  \end{cases} \]
It follows from Lemma \ref{curv_in_surf} that $\mathscr E_{M_j} \otimes \mathscr L$ and $\mathcal{O}_{\PP^1_{j+1}}$ form an $A_2$-configuration of objects, therefore also $\mathscr E_{M_j}$ and  $\mathscr E_{N_{j+1}}$.

It remains to show that 
\begin{equation} \label{hom_o} 
\dim \left( \Hom^*_{\DC(\check X)}(\mathscr E_{M_j}, \mathscr E_{M_{j+1}}) \right) = 0. 
\end{equation}
Notice that $D_j$ and $D_{j+1}$ intersect in a rational curve $Q$ corresponding to the edge $e_{j4}$ of $C_j$, which also coincides with the edge $e_{(j+1)1}$ of $C_{j+1}$. We have that
\[  \Hom^*_{\DC(\check X)}(\mathscr E_{M_j}, \mathscr E_{M_{j+1}}) =  \Hom^*_{\DC(\check X)}(\mathcal{O}_{D_j}, \mathscr L_{-\kappa_j} \otimes \mathscr L_{\kappa_{j+1}}|_{D_{j+1}}). \]
Observe that we always have $-\kappa_{j4} + \kappa_{(j+1)1} = j$. This implies
\[ \mathscr L_{-\kappa_j} \otimes \mathscr L_{\kappa_{j+1}}|_{Q} = \mathcal{O}_Q(j). \]
Considering that the self-intersection of $Q$ inside $D_{j+1}$ is $-(j+1)$, equality (\ref{hom_o}) follows from Lemma \ref{inters_curv}. This completes the proof.
\end{proof}

Recall that a Hirzebruch surface of type $j$ is defined as $\Sigma_j = \PP(\mathcal{O}_{\PP^1} \oplus \mathcal{O}_{\PP^1}(j))$. The second cohomology $H^2(\Sigma_j, \C)$ is spanned by divisors $B$ and $F$, where $B = \PP(\mathcal{O}_{\PP^1}(j)) \subset \Sigma_j$ and $F$ is the fibre of the projection onto $\PP^1$. The one point blowup of $\Sigma_j$  has second cohomology spanned by $B$, $F$ and the exceptional curve $E$. We have
\[ B^2 = -j, \ \ F^2 = 0, \ \ E^2 = -1, \ \ B \cdot F = 1, \ \ E \cdot F = E \cdot B = 0. \]
In our picture for $D_j$ (in Figure \ref{sections_a2d}), the divisor $B$ corresponds to the edge $e_{j1}$, $F$ corresponds to $e_{j5}$ and $E$ to $e_{j3}$. Using these facts, it is easy to show that 
\[ \mathscr E_{M_{j}} = \mathscr L_{\kappa_j}|_{D_j} = \begin{cases}
	-B - \frac{j+1}{2}F &\ j \  \text{odd}, \\
  -B- \frac{j+2}{2}F +E &\ j \  \text{even}.
\end{cases} \]

\bibliography{biblio}
\bibliographystyle{plain}

\vspace{1cm}
\begin{flushleft}
Mark GROSS \\
DPMMS, Centre for Mathematical Sciences \\
University of Cambridge \\
Wilberforce Road \\
Cambridge, CB3 0WB \\
United Kingdom\\
E-mail address: \email{mgross@dpmms.cam.ac.uk} \\

\vspace{1cm}
Diego MATESSI \\
Dipartimento di Matematica \\
Universit\`a degli Studi di Milano \\
Via Saldini 50 \\
I-20133 Milano, Italy \\
E-mail address: \email{diego.matessi@unimi.it}
\end{flushleft}

\end{document}

%% file: negative_thick.pdf_tex
%% Creator: Inkscape inkscape 0.48.1, www.inkscape.org
%% Accompanies image file 'negative_thick.pdf' (pdf, eps, ps)

\begingroup
  \makeatletter
  \providecommand\color[2][]{%
    \errmessage{(Inkscape) Color is used for the text in Inkscape, but the package 'color.sty' is not loaded}
    \renewcommand\color[2][]{}%
  }
  \providecommand\transparent[1]{%
    \errmessage{(Inkscape) Transparency is used (non-zero) for the text in Inkscape, but the package 'transparent.sty' is not loaded}
    \renewcommand\transparent[1]{}%
  }
  \providecommand\rotatebox[2]{#2}
  \ifx\svgwidth\undefined
    \setlength{\unitlength}{112.60479736pt}
  \else
    \setlength{\unitlength}{\svgwidth}
  \fi
  \global\let\svgwidth\undefined
  \makeatother
  \begin{picture}(1,0.96447753)%
    \put(0,0){\includegraphics[width=\unitlength]{negative_thick.pdf}}%
    \put(0.23611884,0.47716411){\color[rgb]{0,0,0}\makebox(0,0)[lb]{\smash{{$D_{v^-}$}
}}}%
  \end{picture}%
\endgroup

%% file: smoothing_affine.pdf_t
\begin{picture}(0,0)%
\includegraphics{smoothing_affine.pdf}%
\end{picture}%
\setlength{\unitlength}{4144sp}%
\begingroup\makeatletter\ifx\SetFigFontNFSS\undefined%
\gdef\SetFigFontNFSS#1#2#3#4#5{%
  \reset@font\fontsize{#1}{#2pt}%
  \fontfamily{#3}\fontseries{#4}\fontshape{#5}%
  \selectfont}%
\fi\endgroup%
\begin{picture}(1640,1519)(3114,-1676)
\end{picture}%

%% file: can_p2_trop.pdf_t
\begin{picture}(0,0)%
\includegraphics{can_p2_trop.pdf}%
\end{picture}%
\setlength{\unitlength}{4144sp}%
\begingroup\makeatletter\ifx\SetFigFontNFSS\undefined%
\gdef\SetFigFontNFSS#1#2#3#4#5{%
  \reset@font\fontsize{#1}{#2pt}%
  \fontfamily{#3}\fontseries{#4}\fontshape{#5}%
  \selectfont}%
\fi\endgroup%
\begin{picture}(4276,2002)(526,-1961)
\put(541,-1861){\makebox(0,0)[lb]{\smash{{\SetFigFontNFSS{12}{14.4}{\rmdefault}{\mddefault}{\updefault}{\color[rgb]{0,0,0}$(-1, -1)$}%
}}}}
\put(1891,-1411){\makebox(0,0)[lb]{\smash{{\SetFigFontNFSS{12}{14.4}{\rmdefault}{\mddefault}{\updefault}{\color[rgb]{0,0,0}$(1,0)$}%
}}}}
\put(1531,-601){\makebox(0,0)[lb]{\smash{{\SetFigFontNFSS{12}{14.4}{\rmdefault}{\mddefault}{\updefault}{\color[rgb]{0,0,0}$(0,1)$}%
}}}}
\end{picture}%

%% file: a2d_trop.pdf_tex
%% Creator: Inkscape inkscape 0.48.1, www.inkscape.org
%% Accompanies image file 'a2d_trop.pdf' (pdf, eps, ps)

\begingroup
  \makeatletter
  \providecommand\color[2][]{%
    \errmessage{(Inkscape) Color is used for the text in Inkscape, but the package 'color.sty' is not loaded}
    \renewcommand\color[2][]{}%
  }
  \providecommand\transparent[1]{%
    \errmessage{(Inkscape) Transparency is used (non-zero) for the text in Inkscape, but the package 'transparent.sty' is not loaded}
    \renewcommand\transparent[1]{}%
  }
  \providecommand\rotatebox[2]{#2}
  \ifx\svgwidth\undefined
    \setlength{\unitlength}{300.15478516pt}
  \else
    \setlength{\unitlength}{\svgwidth}
  \fi
  \global\let\svgwidth\undefined
  \makeatother
  \begin{picture}(1,0.35431093)%
    \put(0,0){\includegraphics[width=\unitlength]{a2d_trop.pdf}}%
    \put(0.55464307,0.08023613){\color[rgb]{0,0,0}\makebox(0,0)[lb]{\smash{$(6,0)$}}}%
    \put(-0.00132744,0.08023617){\color[rgb]{0,0,0}\makebox(0,0)[lb]{\smash{$(0,0)$}}}%
    \put(-0.00132744,0.28081632){\color[rgb]{0,0,0}\makebox(0,0)[lb]{\smash{$(0,2)$}}}%
  \end{picture}%
\endgroup

%% file: sections_dim2.pdf_t
\begin{picture}(0,0)%
\includegraphics{sections_dim2.pdf}%
\end{picture}%
\setlength{\unitlength}{4144sp}%
\begingroup\makeatletter\ifx\SetFigFontNFSS\undefined%
\gdef\SetFigFontNFSS#1#2#3#4#5{%
  \reset@font\fontsize{#1}{#2pt}%
  \fontfamily{#3}\fontseries{#4}\fontshape{#5}%
  \selectfont}%
\fi\endgroup%
\begin{picture}(1419,1034)(537,-823)
\put(946,-601){\makebox(0,0)[lb]{\smash{{\SetFigFontNFSS{12}{14.4}{\rmdefault}{\mddefault}{\updefault}{\color[rgb]{0,0,0}$\pi$}%
}}}}
\end{picture}%

%% file: a2d_sections.pdf_tex
%% Creator: Inkscape inkscape 0.48.1, www.inkscape.org
%% Accompanies image file 'a2d_sections.pdf' (pdf, eps, ps)

\begingroup
  \makeatletter
  \providecommand\color[2][]{%
    \errmessage{(Inkscape) Color is used for the text in Inkscape, but the package 'color.sty' is not loaded}
    \renewcommand\color[2][]{}%
  }
  \providecommand\transparent[1]{%
    \errmessage{(Inkscape) Transparency is used (non-zero) for the text in Inkscape, but the package 'transparent.sty' is not loaded}
    \renewcommand\transparent[1]{}%
  }
  \providecommand\rotatebox[2]{#2}
  \ifx\svgwidth\undefined
    \setlength{\unitlength}{91.79382324pt}
  \else
    \setlength{\unitlength}{\svgwidth}
  \fi
  \global\let\svgwidth\undefined
  \makeatother
  \begin{picture}(1,0.99399741)%
    \put(0,0){\includegraphics[width=\unitlength]{a2d_sections.pdf}}%
    \put(0.3962262,0.47979982){\color[rgb]{0,0,0}\makebox(0,0)[lb]{\smash{$C_j$}}}%
    \put(0.05288321,0.36095028){\color[rgb]{0,0,0}\makebox(0,0)[lb]{\smash{$e_{j4}$}}}%
    \put(0.55292251,0.17959046){\color[rgb]{0,0,0}\makebox(0,0)[lb]{\smash{$e_{j5}$}}}%
    \put(0.76507304,0.61050535){\color[rgb]{0,0,0}\makebox(0,0)[lb]{\smash{$e_{j1}$}}}%
    \put(0.52562679,0.95426843){\color[rgb]{0,0,0}\makebox(0,0)[lb]{\smash{$e_{j2}$}}}%
    \put(0.1646987,0.81077757){\color[rgb]{0,0,0}\makebox(0,0)[lb]{\smash{$e_{j3}$}}}%
  \end{picture}%
\endgroup

%% file: fan.pdf_tex
%% Creator: Inkscape inkscape 0.48.1, www.inkscape.org
%% Accompanies image file 'fan.pdf' (pdf, eps, ps)

\begingroup
  \makeatletter
  \providecommand\color[2][]{%
    \errmessage{(Inkscape) Color is used for the text in Inkscape, but the package 'color.sty' is not loaded}
    \renewcommand\color[2][]{}%
  }
  \providecommand\transparent[1]{%
    \errmessage{(Inkscape) Transparency is used (non-zero) for the text in Inkscape, but the package 'transparent.sty' is not loaded}
    \renewcommand\transparent[1]{}%
  }
  \providecommand\rotatebox[2]{#2}
  \ifx\svgwidth\undefined
    \setlength{\unitlength}{134.62246094pt}
  \else
    \setlength{\unitlength}{\svgwidth}
  \fi
  \global\let\svgwidth\undefined
  \makeatother
  \begin{picture}(1,1.01405288)%
    \put(0,0){\includegraphics[width=\unitlength]{fan.pdf}}%
    \put(0.08897529,0.41364359){\color[rgb]{0,0,0}\makebox(0,0)[lb]{\smash{{\small $p_j$}}}}%
    \put(0.46144547,0.11545526){\color[rgb]{0,0,0}\makebox(0,0)[lb]{\smash{{\small $p_{j+1}$}}}}%
    \put(0.05501783,0.18337007){\color[rgb]{0,0,0}\makebox(0,0)[lb]{\smash{{\small $\check e_j$}}}}%
    \put(0.42642689,0.37438037){\color[rgb]{0,0,0}\makebox(0,0)[lb]{\smash{{\small $v_C$}}}}%
    \put(0.04669142,0.68551499){\color[rgb]{0,0,0}\makebox(0,0)[lb]{\smash{$\nu_j$}}}%
    \put(0.73220639,0.25043579){\color[rgb]{0,0,0}\makebox(0,0)[lb]{\smash{$\nu_{j+1}$}}}%
    \put(0.26512927,0.93892219){\color[rgb]{0,0,0}\makebox(0,0)[lb]{\smash{{\small $\check e_{j-1}$}}}}%
  \end{picture}%
\endgroup

%% file: p+and-.pdf_tex
%% Creator: Inkscape inkscape 0.48.1, www.inkscape.org
%% Accompanies image file 'p+and-.pdf' (pdf, eps, ps)

\begingroup
  \makeatletter
  \providecommand\color[2][]{%
    \errmessage{(Inkscape) Color is used for the text in Inkscape, but the package 'color.sty' is not loaded}
    \renewcommand\color[2][]{}%
  }
  \providecommand\transparent[1]{%
    \errmessage{(Inkscape) Transparency is used (non-zero) for the text in Inkscape, but the package 'transparent.sty' is not loaded}
    \renewcommand\transparent[1]{}%
  }
  \providecommand\rotatebox[2]{#2}
  \ifx\svgwidth\undefined
    \setlength{\unitlength}{232.30239258pt}
  \else
    \setlength{\unitlength}{\svgwidth}
  \fi
  \global\let\svgwidth\undefined
  \makeatother
  \begin{picture}(1,0.45029582)%
    \put(0,0){\includegraphics[width=\unitlength]{p+and-.pdf}}%
    \put(0.19453847,0.32881117){\color[rgb]{0,0,0}\makebox(0,0)[lb]{\smash{$P_e^+$}}}%
    \put(0.01598528,0.13183293){\color[rgb]{0,0,0}\makebox(0,0)[lb]{\smash{$P_e^-$}}}%
    \put(0.16698002,0.23054962){\color[rgb]{0,0,0}\makebox(0,0)[lb]{\smash{$\check e$}}}%
    \put(0.30265132,0.23054962){\color[rgb]{0,0,0}\makebox(0,0)[lb]{\smash{$\check e^+$}}}%
    \put(0.16524066,0.05835146){\color[rgb]{0,0,0}\makebox(0,0)[lb]{\smash{$\check e^-$}}}%
    \put(0.85925127,0.260119){\color[rgb]{0,0,0}\makebox(0,0)[lb]{\smash{$p^-$}}}%
    \put(0.70423697,0.10461534){\color[rgb]{0,0,0}\makebox(0,0)[lb]{\smash{$p^+$}}}%
    \put(0.72531944,0.24446463){\color[rgb]{0,0,0}\makebox(0,0)[lb]{\smash{$e$}}}%
    \put(0.7844582,0.37491777){\color[rgb]{0,0,0}\makebox(0,0)[lb]{\smash{$e^-$}}}%
    \put(0.5844301,0.15401713){\color[rgb]{0,0,0}\makebox(0,0)[lb]{\smash{$e^+$}}}%
  \end{picture}%
\endgroup

%% file: van_cycles_constr.pdf_tex
%% Creator: Inkscape inkscape 0.48.4, www.inkscape.org
%% Accompanies image file 'van_cycles_constr.pdf' (pdf, eps, ps)
\begingroup%
  \makeatletter%
  \providecommand\color[2][]{%
    \errmessage{(Inkscape) Color is used for the text in Inkscape, but the package 'color.sty' is not loaded}%
    \renewcommand\color[2][]{}%
  }%
  \providecommand\transparent[1]{%
    \errmessage{(Inkscape) Transparency is used (non-zero) for the text in Inkscape, but the package 'transparent.sty' is not loaded}%
    \renewcommand\transparent[1]{}%
  }%
  \providecommand\rotatebox[2]{#2}%
  \ifx\svgwidth\undefined%
    \setlength{\unitlength}{217.46728516bp}%
    \ifx\svgscale\undefined%
      \relax%
    \else%
      \setlength{\unitlength}{\unitlength * \real{\svgscale}}%
    \fi%
  \else%
    \setlength{\unitlength}{\svgwidth}%
  \fi%
  \global\let\svgwidth\undefined%
  \global\let\svgscale\undefined%
  \makeatother%
  \begin{picture}(1,0.647369)%
    \put(0,0){\includegraphics[width=\unitlength]{van_cycles_constr.pdf}}%
    \put(0.02587244,0.17635072){\color[rgb]{0,0,0}\makebox(0,0)[lb]{\smash{$1$}}}%
    \put(0.94432441,0.18344727){\color[rgb]{0,0,0}\makebox(0,0)[lb]{\smash{$2$}}}%
    \put(0.49388012,0.61925673){\color[rgb]{0,0,0}\makebox(0,0)[lb]{\smash{$3$}}}%
  \end{picture}%
\endgroup%

%% file: conifold_res.pdf_t
\begin{picture}(0,0)%
\includegraphics{conifold_res.pdf}%
\end{picture}%
\setlength{\unitlength}{4144sp}%
\begingroup\makeatletter\ifx\SetFigFontNFSS\undefined%
\gdef\SetFigFontNFSS#1#2#3#4#5{%
  \reset@font\fontsize{#1}{#2pt}%
  \fontfamily{#3}\fontseries{#4}\fontshape{#5}%
  \selectfont}%
\fi\endgroup%
\begin{picture}(2498,2498)(3059,-3187)
\put(3362,-1229){\makebox(0,0)[lb]{\smash{{\SetFigFontNFSS{12}{14.4}{\rmdefault}{\mddefault}{\updefault}{\color[rgb]{0,0,0}$\theta$}%
}}}}
\put(3239,-1335){\makebox(0,0)[lb]{\smash{{\SetFigFontNFSS{12}{14.4}{\rmdefault}{\mddefault}{\updefault}{\color[rgb]{0,0,0}$\psi$}%
}}}}
\put(4639,-2335){\makebox(0,0)[lb]{\smash{{\SetFigFontNFSS{12}{14.4}{\rmdefault}{\mddefault}{\updefault}{\color[rgb]{0,0,0}$\psi$}%
}}}}
\put(4732,-2219){\makebox(0,0)[lb]{\smash{{\SetFigFontNFSS{12}{14.4}{\rmdefault}{\mddefault}{\updefault}{\color[rgb]{0,0,0}$\theta$}%
}}}}
\put(3742,-2489){\makebox(0,0)[lb]{\smash{{\SetFigFontNFSS{12}{14.4}{\rmdefault}{\mddefault}{\updefault}{\color[rgb]{0,0,0}$p$}%
}}}}
\put(4762,-1479){\makebox(0,0)[lb]{\smash{{\SetFigFontNFSS{12}{14.4}{\rmdefault}{\mddefault}{\updefault}{\color[rgb]{0,0,0}$p'$}%
}}}}
\end{picture}%

%% file: compact_sup.pdf_tex
%% Creator: Inkscape inkscape 0.48.1, www.inkscape.org
%% Accompanies image file 'compact_sup.pdf' (pdf, eps, ps)

\begingroup
  \makeatletter
  \providecommand\color[2][]{%
    \errmessage{(Inkscape) Color is used for the text in Inkscape, but the package 'color.sty' is not loaded}
    \renewcommand\color[2][]{}%
  }
  \providecommand\transparent[1]{%
    \errmessage{(Inkscape) Transparency is used (non-zero) for the text in Inkscape, but the package 'transparent.sty' is not loaded}
    \renewcommand\transparent[1]{}%
  }
  \providecommand\rotatebox[2]{#2}
  \ifx\svgwidth\undefined
    \setlength{\unitlength}{399.39807129pt}
  \else
    \setlength{\unitlength}{\svgwidth}
  \fi
  \global\let\svgwidth\undefined
  \makeatother
  \begin{picture}(1,0.29244007)%
    \put(0,0){\includegraphics[width=\unitlength]{compact_sup.pdf}}%
    \put(0.8275041,0.25169349){\color[rgb]{0,0,0}\makebox(0,0)[lb]{\smash{{\small $-m_{e}^+ =(1,1)$}}}}%
    \put(0.59867244,0.25169349){\color[rgb]{0,0,0}\makebox(0,0)[lb]{\smash{{\small $-m_{e}^-=(0,1)$}}}}%
    \put(0.69660148,0.02316687){\color[rgb]{0,0,0}\makebox(0,0)[lb]{\smash{$(0,0)$}}}%
    \put(0.0901153,0.25157147){\color[rgb]{0,0,0}\makebox(0,0)[lb]{\smash{$e^+$}}}%
    \put(0.27140842,0.28142349){\color[rgb]{0,0,0}\makebox(0,0)[lb]{\smash{$p_C$}}}%
    \put(0.21853046,0.15265706){\color[rgb]{0,0,0}\makebox(0,0)[lb]{\smash{$e$}}}%
    \put(0.46130991,0.21128697){\color[rgb]{0,0,0}\makebox(0,0)[lb]{\smash{$e^-$}}}%
    \put(0.3788188,0.0205939){\color[rgb]{0,0,0}\makebox(0,0)[lb]{\smash{$d^-$}}}%
    \put(0.0274728,0.02251556){\color[rgb]{0,0,0}\makebox(0,0)[lb]{\smash{$d^+$}}}%
    \put(0.23519148,0.19219786){\color[rgb]{0,0,0}\makebox(0,0)[lb]{\smash{$Q_e^+$}}}%
    \put(0.21799739,0.09006464){\color[rgb]{0,0,0}\makebox(0,0)[lb]{\smash{$Q_e^-$}}}%
    \put(0.39158927,0.13119743){\color[rgb]{0,0,0}\makebox(0,0)[lb]{\smash{$p^-$}}}%
    \put(0.09101513,0.14129388){\color[rgb]{0,0,0}\makebox(0,0)[lb]{\smash{$p^+$}}}%
  \end{picture}%
\endgroup

%% file: intersection_singularities.pdf_tex
%% Creator: Inkscape inkscape 0.48.1, www.inkscape.org
%% Accompanies image file 'intersection_singularities.pdf' (pdf, eps, ps)

\begingroup
  \makeatletter
  \providecommand\color[2][]{%
    \errmessage{(Inkscape) Color is used for the text in Inkscape, but the package 'color.sty' is not loaded}
    \renewcommand\color[2][]{}%
  }
  \providecommand\transparent[1]{%
    \errmessage{(Inkscape) Transparency is used (non-zero) for the text in Inkscape, but the package 'transparent.sty' is not loaded}
    \renewcommand\transparent[1]{}%
  }
  \providecommand\rotatebox[2]{#2}
  \ifx\svgwidth\undefined
    \setlength{\unitlength}{328.52456055pt}
  \else
    \setlength{\unitlength}{\svgwidth}
  \fi
  \global\let\svgwidth\undefined
  \makeatother
  \begin{picture}(1,0.33352478)%
    \put(0,0){\includegraphics[width=\unitlength]{intersection_singularities.pdf}}%
    \put(-0.0016643,0.13366572){\color[rgb]{0,0,0}\makebox(0,0)[lb]{\smash{$L_1$ and $L_{-1}$}}}%
    \put(0.24729471,0.18462584){\color[rgb]{0,0,0}\makebox(0,0)[lb]{\smash{$L_3$}}}%
    \put(0.44779334,0.25145867){\color[rgb]{0,0,0}\makebox(0,0)[lb]{\smash{$L_5$}}}%
    \put(0.71512484,0.31829157){\color[rgb]{0,0,0}\makebox(0,0)[lb]{\smash{$L_7$}}}%
  \end{picture}%
\endgroup

%% file: function_T.pdf_tex
%% Creator: Inkscape inkscape 0.48.4, www.inkscape.org
%% Accompanies image file 'function_T.pdf' (pdf, eps, ps)
\begingroup%
  \makeatletter%
  \providecommand\color[2][]{%
    \errmessage{(Inkscape) Color is used for the text in Inkscape, but the package 'color.sty' is not loaded}%
    \renewcommand\color[2][]{}%
  }%
  \providecommand\transparent[1]{%
    \errmessage{(Inkscape) Transparency is used (non-zero) for the text in Inkscape, but the package 'transparent.sty' is not loaded}%
    \renewcommand\transparent[1]{}%
  }%
  \providecommand\rotatebox[2]{#2}%
  \ifx\svgwidth\undefined%
    \setlength{\unitlength}{256.18137207bp}%
    \ifx\svgscale\undefined%
      \relax%
    \else%
      \setlength{\unitlength}{\unitlength * \real{\svgscale}}%
    \fi%
  \else%
    \setlength{\unitlength}{\svgwidth}%
  \fi%
  \global\let\svgwidth\undefined%
  \global\let\svgscale\undefined%
  \makeatother%
  \begin{picture}(1,0.42569008)%
    \put(0,0){\includegraphics[width=\unitlength]{function_T.pdf}}%
    \put(0.7762627,0.08100411){\color[rgb]{0,0,0}\makebox(0,0)[lb]{\smash{{\small $S^-$}}}}%
    \put(0.70415845,0.28312722){\color[rgb]{0,0,0}\makebox(0,0)[lb]{\smash{{\small $T$}}}}%
    \put(0.1734752,0.31373002){\color[rgb]{0,0,0}\makebox(0,0)[lb]{\smash{{\small $S^-$}}}}%
    \put(0.16566823,0.06725281){\color[rgb]{0,0,0}\makebox(0,0)[lb]{\smash{{\small $S^+$}}}}%
  \end{picture}%
\endgroup%

%% file: function_T_2.pdf_tex
%% Creator: Inkscape inkscape 0.48.4, www.inkscape.org
%% Accompanies image file 'function_T_2.pdf' (pdf, eps, ps)
\begingroup%
  \makeatletter%
  \providecommand\color[2][]{%
    \errmessage{(Inkscape) Color is used for the text in Inkscape, but the package 'color.sty' is not loaded}%
    \renewcommand\color[2][]{}%
  }%
  \providecommand\transparent[1]{%
    \errmessage{(Inkscape) Transparency is used (non-zero) for the text in Inkscape, but the package 'transparent.sty' is not loaded}%
    \renewcommand\transparent[1]{}%
  }%
  \providecommand\rotatebox[2]{#2}%
  \ifx\svgwidth\undefined%
    \setlength{\unitlength}{256.18137207bp}%
    \ifx\svgscale\undefined%
      \relax%
    \else%
      \setlength{\unitlength}{\unitlength * \real{\svgscale}}%
    \fi%
  \else%
    \setlength{\unitlength}{\svgwidth}%
  \fi%
  \global\let\svgwidth\undefined%
  \global\let\svgscale\undefined%
  \makeatother%
  \begin{picture}(1,0.42569008)%
    \put(0,0){\includegraphics[width=\unitlength]{function_T_2.pdf}}%
    \put(0.72910382,0.16022186){\color[rgb]{0,0,0}\makebox(0,0)[lb]{\smash{{\small $S^-$}}}}%
    \put(0.72949284,0.34784633){\color[rgb]{0,0,0}\makebox(0,0)[lb]{\smash{{\small $T$}}}}%
    \put(0.19578083,0.33603562){\color[rgb]{0,0,0}\makebox(0,0)[lb]{\smash{{\small $S^-$}}}}%
    \put(0.1612071,0.06836809){\color[rgb]{0,0,0}\makebox(0,0)[lb]{\smash{{\small $S^+$}}}}%
  \end{picture}%
\endgroup%

%% file: sphere_section.pdf_tex
%% Creator: Inkscape inkscape 0.48.1, www.inkscape.org
%% Accompanies image file 'sphere_section.pdf' (pdf, eps, ps)

\begingroup
  \makeatletter
  \providecommand\color[2][]{%
    \errmessage{(Inkscape) Color is used for the text in Inkscape, but the package 'color.sty' is not loaded}
    \renewcommand\color[2][]{}%
  }
  \providecommand\transparent[1]{%
    \errmessage{(Inkscape) Transparency is used (non-zero) for the text in Inkscape, but the package 'transparent.sty' is not loaded}
    \renewcommand\transparent[1]{}%
  }
  \providecommand\rotatebox[2]{#2}
  \ifx\svgwidth\undefined
    \setlength{\unitlength}{318.94814453pt}
  \else
    \setlength{\unitlength}{\svgwidth}
  \fi
  \global\let\svgwidth\undefined
  \makeatother
  \begin{picture}(1,0.42272022)%
    \put(0,0){\includegraphics[width=\unitlength]{sphere_section.pdf}}%
    \put(0.03079645,0.3745944){\color[rgb]{0,0,0}\makebox(0,0)[lb]{\smash{{\small $\check n_2$}}}}%
    \put(0.00519563,0.20006678){\color[rgb]{0,0,0}\makebox(0,0)[lb]{\smash{{\small $\check n_3$}}}}%
    \put(0.34521826,0.06216603){\color[rgb]{0,0,0}\makebox(0,0)[lb]{\smash{{\small $\check n_4$}}}}%
    \put(0.2082925,0.38096596){\color[rgb]{0,0,0}\makebox(0,0)[lb]{\smash{{\small $\check n_1$}}}}%
  \end{picture}%
\endgroup

%% file: an_sing_inters.pdf_tex
%% Creator: Inkscape inkscape 0.48.1, www.inkscape.org
%% Accompanies image file 'an_sing_inters.pdf' (pdf, eps, ps)

\begingroup
  \makeatletter
  \providecommand\color[2][]{%
    \errmessage{(Inkscape) Color is used for the text in Inkscape, but the package 'color.sty' is not loaded}
    \renewcommand\color[2][]{}%
  }
  \providecommand\transparent[1]{%
    \errmessage{(Inkscape) Transparency is used (non-zero) for the text in Inkscape, but the package 'transparent.sty' is not loaded}
    \renewcommand\transparent[1]{}%
  }
  \providecommand\rotatebox[2]{#2}
  \ifx\svgwidth\undefined
    \setlength{\unitlength}{113.11033936pt}
  \else
    \setlength{\unitlength}{\svgwidth}
  \fi
  \global\let\svgwidth\undefined
  \makeatother
  \begin{picture}(1,1.11534752)%
    \put(0,0){\includegraphics[width=\unitlength]{an_sing_inters.pdf}}%
    \put(0.71117113,0.29834005){\color[rgb]{0,0,0}\makebox(0,0)[lb]{\smash{$M_j$}}}%
    \put(0.06199469,0.10383972){\color[rgb]{0,0,0}\makebox(0,0)[lb]{\smash{$M_{j+1}$}}}%
    \put(0.67833342,0.66208097){\color[rgb]{0,0,0}\makebox(0,0)[lb]{\smash{$N_j$}}}%
  \end{picture}%
\endgroup

%% file: ms_singularities.bbl
\begin{thebibliography}{10}

\bibitem{abo_coord_ring}
M.~Abouzaid.
\newblock {Homogeneous Coordinate Rings and Mirror Symmetry for Toric
  Varieties}.
\newblock {\em Geom. Topol.}, 10:1097--1157, 2006.
\newblock \href{http://arxiv.org/abs/math/0511644}{arXiv:math/0511644}.

\bibitem{Ab_HMS}
M.~Abouzaid.
\newblock {Morse homology, tropical geometry, and homological mirror symmetry
  for toric varieties}.
\newblock {\em Selecta Math. (N.S.)}, 15(2):189--270, 2009.
\newblock \href{http://arxiv.org/abs/math/0610004}{arXiv:math/0610004}.

\bibitem{AAK}
M.~Abouzaid, D.~Auroux, and L.~Katzarkov.
\newblock {Lagrangian fibrations on blowups of toric varieties and mirror
  symmetry for hypersurfaces}.
\newblock {\em {Publ. Math. Inst. Hautes \'{E}tudes Sci.}}, 123(199--282),
  2016.
\newblock \href{http://arxiv.org/abs/1205.0053}{arXiv:1205.0053}.

\bibitem{A13}
Denis Auroux.
\newblock A beginner's introduction to {F}ukaya categories.
\newblock In {\em Contact and symplectic topology}, volume~26 of {\em Bolyai
  Soc. Math. Stud.}, pages 85--136. J\'anos Bolyai Math. Soc., Budapest, 2014.
\newblock \href{http://arxiv.org/abs/1301.7056}{arXiv:1301.7056}.

\bibitem{caputo}
R.~Caputo.
\newblock {Topology of the complex tropical line and plane}.
\newblock {Master's Thesis, University of Hamburg}, 2015.

\bibitem{RCB1}
R.~Castano-Bernard.
\newblock {Symplectic invariants of some families of Lagrangian $T^3$
  fibrations}.
\newblock {\em J. of Symplectic Geom. \textbf{2}}, pages 279--308, 2004.

\bibitem{CB-M}
R.~Castano-Bernard and D.~Matessi.
\newblock {Lagrangian 3-torus fibrations}.
\newblock {\em J. of Differential Geom.}, 81(3):483--573, 2009.
\newblock \href{http://arxiv.org/abs/math/0611139}{arXiv:math/0611139}.

\bibitem{CBMS}
R.~Castano-Bernard, D.~Matessi, and J.~Solomon.
\newblock {Symmetries of Lagrangian Fibrations}.
\newblock {\em Adv. in Math.}, 225(3):1341--1386, 2010.

\bibitem{CB-M-Con}
Ricardo Casta{\~n}o-Bernard and Diego Matessi.
\newblock Conifold transitions via affine geometry and mirror symmetry.
\newblock {\em Geom. Topol.}, 18(3):1769--1863, 2014.
\newblock \href{http://arxiv.org/abs/1301.2930}{arXiv:1301.2930}.

\bibitem{ChanChoLauTseng}
K.~Chan, C.-H. Cho, S.-C. Lau, and H.-H. Tseng.
\newblock Gross fibrations, {SYZ} mirror symmetry, and open {G}romov-{W}itten
  invariants for toric {C}alabi-{Y}au orbifolds.
\newblock {\em {J. Differential Geom.}}, {103}({2}):{207--288}, {2016}.
\newblock \href{http://arxiv.org/abs/1306.0437}{arXiv: 1306.0437}.

\bibitem{ChanLauLeung}
K.~Chan, S.-C. Lau, and N.~C. Leung.
\newblock {SYZ mirror symmetry for toric Calabi-Yau manifolds}.
\newblock {\em J. Differential Geom.}, 90(2):177--250, 2012.
\newblock \href{http://arxiv.org/abs/1006.3830}{arXiv:1006.3830}.

\bibitem{CPU:sections}
K.~Chan, D.~Pomerleano, and K.~Ueda.
\newblock {Lagrangian sections on mirrors of toric Calabi-Yau 3-folds}.
\newblock \href{http://arxiv.org/abs/1602.07075}{arXiv:1602.07075 [math.SG]}.

\bibitem{CPU:hms:conifold}
K.~Chan, D.~Pomerleano, and K.~Ueda.
\newblock {Lagrangian torus fibrations and homological mirror symmetry for the
  conifold}.
\newblock {\em {Comm. Math. Phys.}}, 341(1):{135--178}, 2016.
\newblock \href{http://arxiv.org/abs/1305.0968}{arXiv:1305.0968}.

\bibitem{chan:an}
Kwokwai Chan.
\newblock {Homological mirror symmetry for {$A\sb n$}-resolutions as a
  {$T$}-duality}.
\newblock {\em J. Lond. Math. Soc. (2)}, 87(1):204--222, 2013.
\newblock \href{http://arxiv.org/abs/1112.0844}{arXiv:1112.0844}.

\bibitem{chan:ueda}
Kwokwai Chan and Kazushi Ueda.
\newblock {Dual torus fibrations and homological mirror symmetry for {$A\sb
  n$}-singlarities}.
\newblock {\em Commun. Number Theory Phys.}, 7(2):361--396, 2013.
\newblock \href{http://arxiv.org/abs/1210.0652}{arXiv:1210.0652}.

\bibitem{CKYZ}
T.-M. Chiang, A.~Klemm, S.-T. Yau, and E.~Zaslow.
\newblock {Local Mirror Symmetry: Calculations and Interpretations}.
\newblock {\em Adv. Theor. Math. Phys. \textbf{3}}, pages 495--565, 1999.
\newblock \href{http://arxiv.org/abs/hep-th/9903053}{hep-th/9903053}.

\bibitem{cox:little:schenck}
D.~A. Cox, J.~B. Little, and H.~K. Schenck.
\newblock {\em {Toric varieties}}, volume 124 of {\em {Graduate Studies in
  Mathematics}}.
\newblock {American Mathematical Society, Providence, RI}, 2011.

\bibitem{fulton:toric}
W.~Fulton.
\newblock {\em {Introduction to toric varieties}}, volume 131 of {\em {Annals
  of Mathematics Studies}}.
\newblock {Princeton University Press, Princeton, NJ}, 1993.

\bibitem{Goldstein1}
E.~Goldstein.
\newblock {Calibrated fibrations on noncompact manifolds via group actions}.
\newblock {\em Duke Math. J. \textbf{110}}, (2):309--343, 2001.
\newblock \href{http://arxiv.org/abs/math.DG/0002097}{math.DG/0002097}.

\bibitem{splagI}
M.~Gross.
\newblock {Special Lagrangian Fibrations I: Topology}.
\newblock In River Edge~NJ World Sci.~Publishing, editor, {\em Integrable
  systems and algebraic geometry (Kobe/Kyoto, 1997)}, pages 156--193, 1998.
\newblock \href{http://arxiv.org/abs/alg-geom/9710006}{alg-geom/9710006}.

\bibitem{Gross_spLagEx}
M.~Gross.
\newblock {Examples of special Lagrangian fibrations}.
\newblock In {\em Symplectic geometry and mirror symmetry (Seoul, 2000)}, pages
  81--109. World Sci. Publishing, River Edge, NJ, 2001.
\newblock \href{http://arxiv.org/abs/math.AG/0012002}{math.AG/0012002}.

\bibitem{TMS}
M.~Gross.
\newblock {Topological Mirror Symmetry}.
\newblock {\em Invent. Math.}, 144:75--137, 2001.

\bibitem{Gross_SYZ_rev}
M.~Gross.
\newblock {{The Strominger-Yau-Zaslow conjecture: From torus fibrations to
  degenerations}}.
\newblock In {\em Algebraic Geometry, Seattle 2005}, volume~80 of {\em
  Proceedings of Symposia in Pure Mathematics}, pages 149--192. AMS, 2005.

\bibitem{GHJ}
M.~Gross, D.~Huybrechts, and D.~Joyce.
\newblock {\em ``Calabi-Yau manifolds and related geometries'' Lecture notes at
  a summer school in Nordfjordeid, Norway, June 2001}.
\newblock Springer Verlag, 2003.

\bibitem{G-Siebert2003}
M.~Gross and B.~Siebert.
\newblock {Mirror Symmetry via Logarithmic degeneration data I}.
\newblock {\em J. Differential Geom.}, 72(2):169--338, 2006.
\newblock \href{http://arxiv.org/abs/math.AG/0309070}{math.AG/0309070}.

\bibitem{GrSi_re_aff_cx}
M.~Gross and B.~Siebert.
\newblock {From real affine geometry to complex geometry}.
\newblock {\em Ann. of Math. (2)}, 174(3):1301--1428, 2011.
\newblock \href{http://arxiv.org/abs/math/0703822}{arXiv:math/0703822}.

\bibitem{GS:ICM}
M.~Gross and B.~Siebert.
\newblock {Local mirror symmetry in the tropics}.
\newblock In {\em {Proceedings of the {I}nternational {C}ongress of
  {M}athematicians---{S}eoul 2014. {V}ol. {II}}}, pages {723--744}. {Kyung Moon
  Sa, Seoul}, 2014.
\newblock \href{http://arxiv.org/abs/1404.3585}{arXiv:1404.3585}.

\bibitem{harshorne}
R.~Hartshorne.
\newblock {\em Algebraic geometry}.
\newblock {Graduate Studies in Mathematics}. Springer, 1977.

\bibitem{kkv}
S.~Katz, A.~Klemm, and C.~Vafa.
\newblock Geometric engineering of quantum field theories.
\newblock {\em Nucl. Phys. B}, 497:173--195, 1997.
\newblock \href{https://arxiv.org/abs/hep-th/9609239}{arXiv:hep-th/9609239}.

\bibitem{mikh_pants}
G.~Mikhalkin.
\newblock Decomposition into pairs-of-pants for complex hypersurfaces.
\newblock {\em Topology \textbf{43}}, pages 1035 -- 1065, 2004.

\bibitem{Mikh-tropical}
G.~Mikhalkin.
\newblock {Enumerative tropical algebraic geometry in $\mathbb{R}^2$}.
\newblock {\em J. Amer. Math. Soc. \textbf{18}}, pages 313--377, 2005.
\newblock \href{http://arxiv.org/abs/math/0312530}{math/0312530}.

\bibitem{ST}
P.~Seidel and R.~Thomas.
\newblock Braid group actions on derived categories of coherent sheaves.
\newblock {\em Duke Math. J.}, 108(1):37--108, 2001.

\bibitem{Seidel:susp}
Paul Seidel.
\newblock Suspending {L}efschetz fibrations, with an application to local
  mirror symmetry.
\newblock {\em Comm. Math. Phys.}, 297(2):515--528, 2010.
\newblock \href{http://arxiv.org/abs/0907.2063}{arXiv:0907.2063}.

\bibitem{STY}
I.~Smith, R.~Thomas, and S-Y. Yau.
\newblock Symplectic conifold transitions.
\newblock {\em J. Differential Geom. \textbf{62}}, pages 209--242, 2002.

\bibitem{SYZ}
A.~Strominger, S-T. Yau, and E.~Zaslow.
\newblock {Mirror symmetry is $T$-duality}.
\newblock {\em Nucl. Phys. B \textbf{479}}, pages 243--259, 1996.
\newblock \href{http://arxiv.org/abs/hep-th/9606040}{hep-th/9606040}.

\bibitem{TY}
R.~P. Thomas and S-Y. Yau.
\newblock {Special Lagrangians, stable bundles and mean curvature flow}.
\newblock {\em Comm. Anal. Geom. \textbf{10}}, (5):1075--1113, 2002.

\bibitem{viro}
O.~Viro.
\newblock {Patchworking real algebraic varieties}.
\newblock Preprint, 2006.

\bibitem{viro_patch}
O.~Ya. Viro.
\newblock Gluing of plane real algebraic curves and constructions of curves of
  degrees {$6$} and {$7$}.
\newblock In {\em Topology ({L}eningrad, 1982)}, volume 1060 of {\em Lecture
  Notes in Math.}, pages 187--200. Springer, Berlin, 1984.

\end{thebibliography}
